\documentclass[EJP, preprint]{ejpecp} 

\SHORTTITLE{Branching stable processes and motion by mean curvature flow}

\TITLE{Branching stable processes and motion by mean curvature flow\support{The third author is supported by the ANID/Doctorado en el extranjero doctoral scholarship, grant number 2018-72190055. The first author was supported by was supported by the Engineering and Physical Sciences Research Council [EP/R513131/1].}} 

\AUTHORS{%
  Kimberly~Becker\footnote{University of Oxford, United Kingdom.
    \EMAIL{kbecker865@gmail.com}}
  \and 
  Alison~Etheridge\footnote{University of Oxford, United Kingdom. \EMAIL{etheridg@stats.ox.ac.uk}}
  \and 
  Ian~Letter\footnote{University of Oxford, United Kingdom. \EMAIL{iletter@dim.uchile.cl}}}


\AMSSUBJ{60J85; 92D15} 

\ABSTRACT{We prove a new result relating solutions of the scaled fractional Allen--Cahn equation to motion by mean curvature flow, motivated by the motion of hybrid zones in populations that exhibit long range dispersal. Our proof is purely probabilistic and takes inspiration from Etheridge et al. \cite{etheridge2017branching} to describe solutions of the fractional Allen--Cahn equation in terms of ternary branching $\alpha$-stable motions. To overcome technical difficulties arising from the heavy-tailed nature of the stable distribution, we couple ternary branching stable motions to ternary branching Brownian motions subordinated by truncated stable subordinators.}

\usepackage{tkz-graph}
\usepackage{bbm}
\usepackage{bigints}
\usepackage{tzplot}
\usepackage{epsfig,epstopdf,url,array,tikz,tikz-cd}
\usepackage{mathrsfs}
\usepackage{tipa}
\usetikzlibrary{matrix,arrows,decorations.pathmorphing,positioning,calc}
\usepackage[shortlabels]{enumitem}   
\usepackage{comment}
\usepackage{accents}

\makeatletter
\newcommand{\leqnomode}{\tagsleft@true}
\newcommand{\reqnomode}{\tagsleft@false}

\newcommand{\dd}{\mathbbm{d}}
\newcommand{\RR}{\mathbb{R}}
\newcommand{\B}{{{B}}}
\newcommand{\BB}{\boldsymbol{B}}

\newcommand{\PP}{\mathbb{P}}
\newcommand{\PE}{\widehat{\PP}}

\newcommand{\PPmarked}{{\accentset{\times}{\PP}}}
\newcommand{\EE}{\mathbb{E}}
\newcommand{\QQ}{\mathbb{Q}}
\newcommand{\NN}{\mathbb{N}}

\newcommand{\conditional}{\left \vert \right.}
\newcommand{\X}{\boldsymbol{X}}
\newcommand{\Xtrunc}{\boldsymbol{B}_{R^\varepsilon}}
\newcommand{\Ytrunc}{\boldsymbol{W}_{\hspace{-.1cm} R^\varepsilon}}
\newcommand{\Y}{{\boldsymbol{Y}}}
\newcommand{\T}{{\cal T}}
\newcommand{\TrX}{{{\cal T}(\X(t))}}
\newcommand{\TrY}{{{\cal T}(\Y(t))}}
\newcommand{\Trpre}{{{\cal T}(\Xtrunc(t))}}

\newcommand{\VotemajY}{\mathbb{V}_p(\Y(t))}
\newcommand{\VotemajXnop}{\mathbb{V}(\X(t))}

\newcommand{\Voteprenop}{\widehat{\V}(\Xtrunc(t))}
\newcommand{\V}{{\mathbb{V}}}
\newcommand{\VM}{\mathbb{V}^\times_p}
\newcommand{\VV}{\mathbb{V}_p}
\newcommand{\Votepre}{\widehat{\V}_p(\Xtrunc(t))}
\newcommand{\Votemarknop}{\mathbb{V}^{\times}(\BB_{R^\varepsilon}(t))}
\newcommand{\PPt}{\PPmarked^{_{_{\text{\scriptsize{$t$}    } }}}}
\newcommand{\PPttau}{\PPmarked^{_{_ {\text{\scriptsize{$t\hspace{-.2em}-\hspace{-.2em}\tau$}}}}}}

\newcommand{\gee}{{g_\times}}
\newcommand{\geeone}{g^{(1)}_\times}
\newcommand{\geen}{g^{(n)}_\times}
\newcommand{\geenplusone}{g^{(n+1)}_\times}
\newcommand{\Indicator}{{\mathbbm{1}_{\{x\geq0\}}}}
\newcommand{\Zbf}{\boldsymbol{Z}}

\usepackage{endnotes}
\usepackage[splitrule]{footmisc} 
\usepackage{bera}

\begin{document}



\section{Introduction}
Reaction-diffusion equations have been used to study a wide range of phenomena within the natural sciences, and are a topic of great mathematical intrigue in their own right. They appear as models for the spread of populations \cite{fisher1937wave, kolmogorov1937etude, aronson1978multidimensional}, phase transitions \cite{henkel2004non, wang, soner}, combustion \cite{berestycki1985traveling, berestycki1989multi}, and chemical reactions \cite{allen1979microscopic, bramson1991asymptotic}. In this work, we explore \textit{fractional} reaction-diffusion equations that model populations that exhibit long range dispersal. Fractional reaction-diffusion equations are reaction-diffusion equations in which the diffusive term is replaced by the generator of a pure jump process (namely, a stable process). As a result, they present new challenges that have not yet been fully explored in a probabilistic context. \\
\indent In recent decades, fractional reaction-diffusion equations and reaction-diffusion equations with anomalous diffusion have surged in popularity. This is in part due to their relevance as physical models. From a purely mathematical viewpoint, they pose technical difficulties that classical parabolic equations like the Fisher-KPP equation do not. Much of the theoretical work on these topics, to date, has been led by the PDE community \cite{mellet, gui2015traveling, cabre, cabre1, cabre2, Mellet2}. The effect of long range dispersal on various models has also been studied numerically \cite{Mancinelli, del2009truncation, hallatschek2014acceleration}, and applied to data from epidemics \cite{brockmann2013hidden} and plant populations \cite{clark2003estimating, marco2011comparing, cannas2006long}. \\ 
\indent In the context of mathematical biology, fractional reaction-diffusion equations arise naturally as models for populations that exhibit long range dispersal (i.e. the capacity for offspring to, on rare occasions, establish very far away from their parent). This behaviour is ubiquitous in nature and is a crucial survival mechanism for many organisms, particularly those insular species that must travel vast distances to populate new regions. Examples include the dispersal of plant seeds (which can travel by wind, water, and can be transported internally or externally by animals) \cite{Cain:2000}, fungi \cite{Buschbom:2007}, and small insects such as flies, moths, and bees (which have the secondary effect of facilitating the hybridisation of the flora they pollinate) \cite{barrett1996reproductive,reatini2021influence}. The ability for certain organisms to populate islands through long range dispersal can have a profound impact on the biological composition of the land, increasing the genetic diversity of isolated regions \cite{weigelt2015global,heaney2000dynamic,reatini2021influence}. One example of this is the Hawaiian angiosperm flora, that cannot be attributed to a single mainland source, but instead have the genetic composition of flora from across circum-Pacific regions \cite{carlquist1967biota, weigelt2015global}. Another recently observed instance of long range dispersal was that of a single finch that travelled over 100 km to an island in the Gal\'apagos where it went on to produce hybrid offspring with the resident population \cite{lamichhaney2018rapid,reatini2021influence}.\\
\indent In this work, we use the fractional Allen--Cahn equation to model the motion of \textit{hybrid zones} in populations exhibiting long range dispersal. Hybrid zones are narrow geographical regions where two genetically distinct species meet and reproduce, resulting in individuals of mixed ancestry (hybrid individuals).  Hybrid zones have been observed extensively in nature. Examples include the European house mouse \cite{hunt1973biochemical} and North American warbler birds \cite{birds} (see \cite{barton1989adaptation} for an extensive list of examples). There are two primary mechanisms acting to maintain hybrid zones. This first is due to a change in environment where the two populations meet. The second, which will be of interest to this work, is when selection acts against hybrid individuals. In this setting, the hybrid zone is maintained for large times through a balance of negative selection with the dispersal of individuals. We will show that the long-time behaviour of hybrid zones maintained by selection in populations that exhibit long range dispersal converges to motion by mean curvature flow under a large range of possible spatial scalings. This new family of scalings, as well as our explicit description of the interface width and speed of convergence, distinguishes our work from that of \cite{imbert2009phasefield}, where it was shown that the solution interface of the fractional Allen--Cahn equation with a diffusive scaling converges to mean curvature flow with respect to the scaling parameter.
\subsection{The Allen--Cahn equation and hybrid zones}
The one-dimensional {Allen--Cahn equation} is a reaction-diffusion equation that takes the form \begin{align}\label{allencahnequation} \partial_t {u^\varepsilon(t,x) }= \Delta u^\varepsilon(t,x) - \frac{1}{\varepsilon^{2}}f(u^\varepsilon(t,x)) \end{align}
for $\varepsilon>0$ fixed and all $t>0$, $x\in \RR$. This equation can be obtained from the unscaled equation $\partial_t u(t,x) = \Delta u(t,x) - f(u(t,x))$ by defining $u^\varepsilon(t,x):= u(\varepsilon^2 t, \varepsilon x).$ Here, $f\in C^2(\RR)$ is assumed to have precisely three zeros, $v_-, v_0$ and $v_+$, such that \begin{align*} &f<0 \text{ on } (-\infty, v_-)\cup (v_0, v_+), \\  &f>0 \text{ on } (v_-, v_0)\cup (v_+, \infty), \\ &f'(v_-), f'(v_+)>0 \text{ and } f'(v_0)<0. \end{align*} 
\indent In equation~(\ref{allencahnequation}), the diffusive term $\Delta$ acts to smooth solutions (in the context of a biological model, this could be viewed as the `mixing' of the population), while the nonlinearity $f$ drives solutions towards one of two states, $v_-$ or $v_+$ (in a population, these states might correspond to the dominance of a particular allele). These opposing effects, which are characteristic of reaction-diffusion equations, create a solution interface separating the two states. Over large spatial and temporal scales (corresponding to $\varepsilon\to 0$) this interface appears sharp and one can study its motion.\\
\indent The Allen--Cahn equation was originally introduced in \cite{allen1979microscopic} to model the motion of curved \textit{antiphase boundaries} (APB) in crystalline solids. These are defective regions in the crystal lattice where atoms have the opposite configuration to that predicted by their lattice system, producing a positive excess of free energy in the system \cite{allen1979microscopic}. This is a non-equilibrium state of the lattice, resulting in the diffusive movement of the APB to minimise the total area of the boundary. The motion of the APB can then be modelled by the solution interface of the Allen--Cahn equation. In fact, in their original work \cite{allen1979microscopic}, Allen and Cahn already note the relevance of long range dispersal to interfacial motion. Allen and Cahn mention that interfacial motion sometimes requires long range dispersal, citing the growth of a (solid) precipitate from a supersaturated solution and the motion of interfaces with impurities as two examples. It was observed by Allen and Cahn in \cite{allen1979microscopic} that, in the case of local dispersion, the velocity of the interfacial motion described by equation~(\ref{allencahnequation}) was proportional to the mean curvature of the boundary. Bronsard and Kohn \cite{bronsard1990, bronsard1991} and Demottoni and Schatzman \cite{mottoni1989, mottoni1990} provided a rigorous proof of this under a variety of dimensional and regularity restrictions, and in 1992, Chen proved this result in all dimensions under relatively weak regularity assumptions \cite{chen}.\\
\indent By viewing the Allen--Cahn equation as a model for a {hybrid zone} in a population with local dispersal, Chen's result has a biological interpretation. As explained in \cite{etheridge2017branching}, the connection between the hybrid zones and the Allen--Cahn equation can be motivated as follows. Consider a single genetic locus in diploid population with allele types $a$ and $A$ that is in Hardy-Weinberg proportions. That is, if the proportion of $a$-alleles in the parental population is $w$, then the proportions of parents of type $aa$, $aA$ and $AA$ are $w^2, 2w(1-w)$ and $(1-w)^2$, respectively. To reflect our assumption that the hybrid zone is maintained by selection against heterozygotes, we assign to each of the allele pairs $aa, aA$ and $AA$ the relative fitnesses $1, 1-s$ and $1$, for $s>0$ a small selection parameter.  These fitnesses refer to the relative proportion of germ cells produced by heterozygotes and homozygotes during reproduction. It follows that if $w$ is the proportion of type $a$ alleles before reproduction, then the proportion of type $a$ alleles after reproduction is $$w^* = \frac{w^2 + w(1-w)(1-{s})}{1 - 2{s}w(1-w)} = w+ {s}w(1-w)(2w-1)+\mathcal{O}(s^2).$$ Taking ${s}= \frac{1}{N}$ and measuring time in units of $N$ generations, the above calculation implies that, as $N\to \infty$, $\frac{dw}{dt} = w(1-w)(2w-1).$ Adding (local) dispersal and applying a diffusive scaling $t\mapsto \varepsilon^2 t$, $x\mapsto \varepsilon x$ for $t>0$ and $x\in \RR^2$ gives \begin{align} \label{ACE} \frac{\partial w}{\partial t} = \Delta w + \frac{1}{\varepsilon^2}w(1-w)(2w-1). \end{align} Chen's result tells us that solutions to scaled Allen--Cahn equation (\ref{ACE}) (also known as the Nagumo's equation, see \cite{nagumo1962active, nagumoref}) converge as $\varepsilon\to 0$ to the indicator function of a set whose boundary evolves according to motion by mean curvature flow. Since the hybrid zone of a diploid population should correspond to the level set $w(t,x) = \tfrac{1}{2}$ for $w$ a solution of (\ref{ACE}), Chen's result tells us that hybrid zones will follow motion by mean curvature flow over large spatial and temporal scales. We note that, although $x\in \RR^2$ is the biologically relevant case, Chen's result and our own hold in all spatial dimensions $\dd \geq 2$.\\
\indent In Etheridge et al. \cite{etheridge2017branching}, the authors used purely probabilstic techniques to reprove Chen's result. This was accomplished by constructing a probabilistic dual to (\ref{ACE}) in terms of ternary branching Brownian motions. Additionally, their proof could be adapted to incorporate genetic drift, which refers to the randomness inherent in reproduction within finite populations. This was achieved via a so-called 
Spatial-$\Lambda$-Fleming-Viot process. It was shown in \cite{etheridge2017branching} that, provided the genetic drift is not too strong, the limiting mean curvature flow behaviour observed in the deterministic case is preserved by the scaled stochastic model. This stochastic result can be extended to the stable setting, and will appear in the PhD thesis of the first author. An interesting avenue for further research would be to find the critical strength of genetic drift that determines if motion by mean curvature flow is preserved. This has been accomplished by Etheridge et al. \cite{ian1} in the Brownian setting, but is more difficult to identify in the stable setting.\\
\indent To model the motion of hybrid zones in populations that exhibit long range dispersal, we consider the fractional Allen--Cahn equation \reqnomode
\begin{align}\label{dodeedo}
    {\partial_t u(t,x) }= -(-\Delta)^{\tiny \frac{\alpha}{2}} u(t,x) +su(t,x)(1-u(t,x))(2u(t,x)-1),
\end{align}
for all $t>0$ and $x\in \RR$ where ${s}$ is a small selection parameter and $-(-\Delta)^{\frac{\alpha}{2}}$ is the generator of an $\alpha$-stable process. In view of Chen's result, it is natural to ask: will mean curvature flow be preserved in populations that exhibit long range dispersal? The answer to this should, of course, depend on the strength of the dispersal mechanism. This is determined by the index $\alpha \in (0,2]$. When $\alpha = 2$, the fractional Laplacian and ordinary Laplacian coincide, so in view of Chen's result \cite{chen}, we expect mean curvature flow to be preserved for $\alpha$ sufficiently close to $2$. As $\alpha$ approaches $0$, the severity and frequency of large jumps increases, so for small $\alpha$ it seems unlikely that motion by mean curvature flow will be seen in the limit. This intuition is supported by results in the PDE literature, with the threshold between the two behaviours occurring at $\alpha=1$. For example, in \cite{caffarelli2010convergence}, Caffarelli and Souganidis consider a threshold dynamics type algorithm to simulate a moving front governed by the fractional heat equation. The resulting interface was shown to converge to mean curvature flow for $\alpha\geq 1$ and `weighted' mean curvature flow for $0<\alpha<1$. The results from the unpublished manuscript \cite{imbert2009phasefield} suggest that equation~(\ref{dodeedo}), with the scaling $I(\varepsilon) = \varepsilon$, should converge as $\varepsilon\to 0$ to motion by mean curvature flow when $\alpha\in (1,2)$, however such a result is certainly out of reach with our techniques. As we will see, our result is stated for a large family of possible scalings of equation~(\ref{dodeedo}), which does not include the diffusive scaling taken in \cite{imbert2009phasefield}. \\
\indent We now briefly outline the structure of this paper. First, in Section~\ref{secone}, we state our main result, Theorem~\ref{maintheorem}. In Section~\ref{ch:2}, we go on to construct a probabilistic representation of solutions to the fractional Allen--Cahn equation. We then prove a one-dimensional analogue of our main result in Section~\ref{sectionmajorityvotinginonedimension}. This will enable us to prove Theorem~\ref{maintheorem} in Section~\ref{ch:3}. Supplementary calculations will be provided in the appendix. 

\section{Main Results}\label{secone}
\indent To begin, we recall the definition of the mean curvature at {a} point on a hypersurface $M\subset \RR^\dd$. Let $n:M\to \RR^\dd$ be the Gauss map, i.e.\ the map that assigns to each point $p\in M$ the outward facing unit vector $n(p)$ orthogonal to the tangent space of $M$ at $p$, denoted $T_pM$. By choosing an appropriate {orthonormal} basis of the tangent space $T_pM$ for all $p\in M$, we can define the shape operator $\mathbb{S}_p$ at $p$ (locally) as the negative Jacobian of $n$ at $p$. It can be shown that there exists an inner product on $T_pM$ (called the first fundamental form) and $\mathbb{S}_p$ can be diagonalised, $\mathbb{S}_p = \text{diag}(\kappa_1(p), ..., \kappa_{\dd-1}(p)).$ The \textit{mean curvature at $p$} is then $$\kappa(p) := \frac{1}{\dd-1}\sum_{i=1}^{\dd-1}\kappa_i(p). $$
\begin{definition}[Motion by mean curvature flow]
Fix $\mathscr{T}>0$. Let $S^{\dd-1}\subset \RR^{\dd}$ be the unit sphere and $(\mathbf{\Gamma}_t)_{0\leq t < \mathscr{T}}$ be a family of smooth embeddings $S^{\dd-1} \to \RR^\dd$. Let $\mathbf{n} = \mathbf{n}_t(\phi)$ be the unit inward normal vector to $\mathbf{\Gamma}_t$ at $\phi$ and let $\kappa = \kappa_t(\phi)$ be the mean curvature of $\mathbf{\Gamma}_t$ at $\phi$. Then $(\mathbf{\Gamma}_t)_{0\leq t < \mathscr{T}}$ is a \textit{mean curvature flow} if, for all $t$ and $\phi$, \begin{align}\label{mcfeqn}
    \frac{\partial \mathbf{\Gamma}_t(\phi)}{\partial t} = \kappa_t(\phi)\mathbf{n}_t(\phi).
\end{align}
\end{definition}
It can be shown that mean curvature flow in dimension $\dd=2$ (also called curve shortening flow) terminates after a finite time $\mathscr{T}$, and by the theorems of Gage and Hamilton (1986) and Grayson (1987), any smoothly embedded closed curve shrinks to a point as $t\uparrow \mathscr{T}$. When $\dd>2$, this does not always hold as singularities may develop. Following Chen \cite{chen}, we shall impose sufficient regularity conditions to ensure the existence of a finite time $\mathscr{T}$ before which the mean curvature flow exists and does not develop a singularity. For an {overview} of existence results for mean curvature flow see \cite{etheridge2017branching}.\\
\indent Let $\dd \geq 1$ and denote the standard Euclidean distance in $\RR^\dd$ by $|\cdot|$. The fractional Laplacian is defined on functions $u:\RR^\dd\to \RR$ with sufficient decay by \begin{align*}-(-\Delta)^{\tiny \frac{\alpha}{2}}u(x) := C_\alpha \lim_{r\to 0} \int_{\RR^\dd\backslash \B_r(x)} \frac{u(y)-u(x)}{| y-x |^{\dd+\alpha}}dy, \end{align*} where $C_\alpha:= \frac{\alpha 2^{\alpha-1} \Gamma\left(\frac{\alpha+\dd}{2}\right)}{\pi^{\frac{\dd}{2}}\Gamma\left(1-\frac{
\alpha}{2}\right)}$ and $\B_r(x)\subset \RR^\dd$ is the sphere of radius $r$ about $x$. We will show that in dimension $\dd\geq 2$, for suitable initial conditions, the solution of the scaled fractional Allen--Cahn equation \reqnomode
\begin{align}\label{mainequation22} \begin{cases} \partial_t u^\varepsilon = -\sigma_\alpha {I(\varepsilon)^{\alpha-2}}(-\Delta)^{\tiny \frac{\alpha}{2}} u^\varepsilon + {\varepsilon^{-2}}u^\varepsilon(1-u^\varepsilon)(2u^\varepsilon-1), \ \ \ t\geq 0, \ x\in \RR^\dd \\  u^\varepsilon(0,x) = p(x), \ \ \ x\in \RR^\dd\end{cases} \end{align} 
converges as $\varepsilon \to 0$ to the indicator function of a set whose boundary evolves according to mean curvature flow. Here, \begin{align}\label{defn_sig_alpha}
      \sigma_\alpha:=  \left(\tfrac{2-\alpha}{\alpha}\right)^{\frac{\alpha}{2}} {\Gamma\left(1-\tfrac{\alpha}{2}\right)}
    \end{align} is a normalising constant that will simplify our later calculations, and $I$ can be any function satisfying the following.
\leqnomode 
\begin{assumptions}\label{assumptions2} For some $\delta>0$, assume that $I: (0,\delta) \to (0,\infty)$ satisfies
\begin{enumerate}[(A)]
\item $\lim \limits_{\varepsilon \rightarrow 0} I(\varepsilon) |\log \varepsilon|^k=0 \, \, \, \, \, \forall\, k \in \mathbb{N},$\label{assumptions2_A}
\item $\lim \limits_{\varepsilon \rightarrow 0} \dfrac{\varepsilon^2 }{I(\varepsilon)^2}|\log \varepsilon| = 0,$ \label{assumptions2_B}
\item $\lim \limits_{\varepsilon \rightarrow 0} \dfrac{I(\varepsilon)^{2 \alpha}}{\varepsilon^2} |\log \varepsilon|^\alpha = 0.$\label{assumptions2_C} 
\end{enumerate} 
\end{assumptions}\reqnomode
The rate of convergence and width of the `solution interface' in our convergence result will ultimately depend on this choice of $I$. Note that Assumption~\ref{assumptions2}~\ref{assumptions2_B} and Assumption~\ref{assumptions2}~\ref{assumptions2_C} are incompatible as soon as $\alpha \leq 1$. When $\alpha >1$ a standing example that fulfills all the conditions is $I(\varepsilon)=\varepsilon |\log \varepsilon|$.\\
\indent In Imbert and  Souganidis' work \cite{imbert2009phasefield}, they consider a class of reaction-diffusion equations with diffusive term given by a singular integral operator. In the case when this operator is the fractional Laplacian $(-\Delta)^{\frac{\alpha}{2}}$ with $\alpha\in (1,2)$, their scaled equation amounts to \begin{align}\label{imbertsequation}
    \partial_t u^\varepsilon = -\varepsilon^{\alpha-2}(-\Delta)^{\frac{\alpha}{2}} + \varepsilon^{-2}f(u^\varepsilon)
\end{align} for a bistable nonlinearity $f$. The authors show that solutions to (\ref{imbertsequation}) converge as $\varepsilon\to 0$ to the indicator function of a set whose boundary evolves under motion by mean curvature flow. This result is distinct from our own, since we consider a family of possible scalings $I(\varepsilon)$. In particular, our result does not include the case when $I(\varepsilon)=\varepsilon$, which is when equations~(\ref{mainequation22}) and (\ref{imbertsequation}) coincide. 
\begin{remark}
Equation~(\ref{mainequation22}) can be obtained from the unscaled fractional Allen--Cahn equation by scaling time and space by 
$$t\mapsto \varepsilon^2 t  \text{ and } x\mapsto \left(\sigma_
\alpha {I(\varepsilon)^{\alpha-2}}\varepsilon^2\right)^{\frac{1}{\alpha}}x. $$
To see this, let $\iota_\varepsilon:= \left(\sigma_
\alpha {I(\varepsilon)^{\alpha-2}}\varepsilon^2\right)^{\frac{1}{\alpha}}$ and define $$u^\varepsilon(t,x):= u(\varepsilon^2 t, \iota_\varepsilon x).$$ Then, by definition of the fractional Laplacian, $$-(-\Delta)^{\frac{\alpha}{2}}u(\varepsilon^2 t, \iota_\varepsilon x) = -\iota_\varepsilon^\alpha (-\Delta)^{\frac{\alpha}{2}} u^\varepsilon(t,x),$$ so using that $u$ is a solution to the unscaled equation $\partial_t u = -(-\Delta)^{\frac{\alpha}{2}}u + f(u)$ where $f(u) = u(1-u)(2u-1)$, we have \begin{align*}
    \partial_t u^\varepsilon(t,x) &= \varepsilon^{-2} \,\partial_t u(\varepsilon^2 t, \iota_\varepsilon x)\\
    &= \varepsilon^{-2}\left( - (-\Delta)^{\frac{\alpha}{2}} u(\varepsilon^2t, \iota_\varepsilon x) + f(u(\varepsilon^2 t, \iota_\varepsilon x)\right)\\
    &= -\varepsilon^{-2}\iota_\varepsilon^\alpha(-\Delta)^{\frac{\alpha}{2}} u^\varepsilon(t,x) + \varepsilon^{-2}f(u^\varepsilon(t,x)),
\end{align*} which is equivalent to (\ref{mainequation22}) by definition of $\iota_\varepsilon$. \\ 
\indent Using the definition of $\sigma_\alpha$ from (\ref{defn_sig_alpha}), as $\alpha\to 2^-$, $\left(\sigma_\alpha {I(\varepsilon)^{\alpha-2}}\varepsilon^2\right)^{\frac{1}{\alpha}}x \to \varepsilon x$ which is consistent with the diffusive scaling considered in the local setting of \cite{etheridge2017branching}. Note that the spatial scaling factor in the fractional setting, $\left(\sigma_\alpha I(\varepsilon)^{\alpha-2}\varepsilon^2\right)^{\frac{1}{\alpha}}$, converges to zero faster than the spatial scaling factor in the local setting, $\varepsilon$. This follows by Assumption~\ref{assumptions2}~\ref{assumptions2_B}, since $$\lim_{\varepsilon\to0} \frac{\left(\sigma_\alpha I(\varepsilon)^{\alpha-2}\varepsilon^2\right)^{\frac{1}{\alpha}}}{\varepsilon}=\lim_{\varepsilon\to0} \left(\frac{\varepsilon}{I(\varepsilon)}\right)^{\frac{2-\alpha}{\alpha}}=0.$$ This suggests that, with our method of proof, to observe a hybrid zone evolving by motion by mean curvature in the original spatial coordinates, one must `zoom out' much more in the stable (non-local) setting than in the local setting. We discuss the origin of our chosen scaling more in Section~\ref{choice_scaling_sec}.
\end{remark}
\indent Our assumptions on the initial condition $p$ in (\ref{mainequation22}) mirror those in \cite{etheridge2017branching}. First, $p$ is assumed to take values in $[0,1]$. Set \begin{align}\label{gammadefinition}\mathbf{\Gamma}_0 := \left\{x\in \RR^\dd : p(x) = \tfrac{1}{2}\right\}.\end{align}
\leqnomode Assume $\mathbf{\Gamma}_0$ is a smooth hypersurface and is the boundary of an open set homeomorphic to a sphere. Further suppose the following. \begin{assumptions}\label{assumptions1} Let $p: \RR^\dd \to [0,1]$ and $\mathbf{\Gamma}_0$ be as in (\ref{gammadefinition}). Denote the shortest Euclidean distance between a point $x\in \RR^\dd$ and the surface $\mathbf{\Gamma}_0$ by $\text{dist}(x, \mathbf{\Gamma}_0).$
\begin{enumerate}[(A)]
\item $\mathbf{\Gamma}_0$ is $C^a$ for some $a>3$.
\item For $x$ inside $\mathbf{\Gamma}_0$, $p(x)<\frac{1}{2}$, and for $x$ outside $\mathbf{\Gamma}_0$, $p(x)>\frac{1}{2}$.\label{assumptions1_B}
\item There exist $r, \gamma >0$ such that, for all $x\in \RR^\dd$, $\left|p(x) - \frac{1}{2}\right| \geq \gamma(\text{dist}(x, \mathbf{\Gamma}_0)\wedge r).$\label{assumptions1_C}
\end{enumerate}
\end{assumptions}
Assumption~\ref{assumptions1}~\ref{assumptions1_B} establishes a sign convention, and Assumption~\ref{assumptions1}~\ref{assumptions1_C} ensures $p(x)$ is bounded away from $\tfrac{1}{2}$ when $x$ is away from the interface. Together, these conditions ensure that the mean curvature flow started from $\mathbf{\Gamma}_0$, $(\mathbf{\Gamma}_t(\cdot))_{t\geq 0}$, exists up until some finite time $\mathscr{T}$.\\ 
\indent Following  \cite{etheridge2017branching}, we let $d(x,t)$ be the signed distance from $\mathbf{\Gamma}_{t}$ to $x$, which we choose to be negative inside $\mathbf{\Gamma}_{t}$ and positive outside $\mathbf{\Gamma}_{t}$. Then, as sets, $$\mathbf{\Gamma}_{t} = \left\{x \in \RR^\dd : d(x,t) =0\right\}. $$ 
\noindent Lastly, define $F(\varepsilon)$ by \reqnomode \begin{align} \label{defnofF} \hspace{5 mm} F(\varepsilon) =\frac{I(\varepsilon)^{2}}{\varepsilon^\frac{2}{\alpha}}|\log\varepsilon| + I(\varepsilon)^{\alpha-1}. \end{align}\leqnomode
Note that, for any $\alpha \in (1,2)$ and function $I:(0,\delta)\to (0,\infty)$ satisfying Assumptions~\ref{assumptions2},  $F(\varepsilon)\to 0$ as $\varepsilon\to 0$. \\
\indent The scaling function $I$ and parameter $\alpha$ will be fixed throughout this work. Just as we have done for $F$, when defining new functions, we will typically not make explicit their dependence on the choice of $I$ and choice of $\alpha$. Our main theorem is the following. 

\begin{theorem}\label{maintheorem}
Let $\alpha \in (1,2)$, $\dd \geq 2$ and fix a function $I$ satisfying Assumptions~\ref{assumptions2}. Suppose $u^\varepsilon$ solves equation~(\ref{mainequation22}) with initial condition $p$ satisfying Assumptions~\ref{assumptions1}. Let $\mathscr{T}$ and $d(x, t)$ be as above, $F$ be as in (\ref{defnofF}) and fix $T^*\in (0, \mathscr{T})$. Then there exists $\varepsilon_\dd(\alpha), a_\dd(\alpha), c_\dd(\alpha), m >0$  such that, for $\varepsilon \in (0, \varepsilon_\dd)$ and $a_\dd \varepsilon^2 |\log\varepsilon|\leq t\leq T^*$, \begin{enumerate}[(1)]
\item for $x$ with $d(x, t)\geq c_\dd I(\varepsilon)|\log\varepsilon|$, we have $u^\varepsilon(t, x)\geq 1-m\frac{\varepsilon^2}{I(\varepsilon)^2}-mF(\varepsilon),$
\item for $x$ with $d(x, t) \leq - c_\dd I(\varepsilon)|\log\varepsilon|$, we have $u^\varepsilon(t, x)\leq m\frac{\varepsilon^2}{I(\varepsilon)^2}+mF(\varepsilon)$. 
\end{enumerate}
\end{theorem}

\indent Of course, we could have stated Theorem~\ref{maintheorem} in terms of an error function $F'(\varepsilon):= F(\varepsilon) + \frac{\varepsilon^2}{I(\varepsilon)^2}$. We choose to make the $\frac{\varepsilon^2}{I(\varepsilon)^2}$ term explicit since it will also appear in the one-dimensional analogue of Theorem~\ref{maintheorem}.\\
\indent Throughout this work, we often often discuss the \textit{solution interface} and its corresponding width. The term {solution interface} refers to the spatial region outside of which the solution $u^\varepsilon(t,\cdot)$ is very close to zero or one. Explicitly, in Theorem~\ref{maintheorem} the solution interface at time $t$ is the set of $x\in \RR^\dd$ for which $|d(x,t)|\leq c_\dd I(\varepsilon)|\log\varepsilon|$, and we call $2c_\dd I(\varepsilon)|\log\varepsilon|$ the \textit{interface width}. We will refer to the error bounds on $u^\varepsilon$ in Theorem~\ref{maintheorem} as the \textit{sharpness} of the interface. In the following example, we observe that neither the $F(\varepsilon)$ or the $\frac{\varepsilon^2}{I(\varepsilon)^2}$ terms in the sharpness of the interface are the dominant term, in general.
\begin{example}
\begin{enumerate}[(1)]
    \item It is easy to verify that $I(\varepsilon) = \varepsilon|\log\varepsilon|$ satisfies Assumptions~\ref{assumptions2}, so for this choice of $I$ the interface width is $\mathcal{O}\left(\varepsilon|\log\varepsilon|^{2}\right)$. One can also check that $F(\varepsilon) = \mathcal{O}(\varepsilon^{\alpha-1}|\log \varepsilon|^{\alpha-1})$, which is dominated by $\frac{\varepsilon^2}{I(\varepsilon)^2}$, so the sharpness of the interface in Theorem~\ref{maintheorem} is $\mathcal{O}\left({\varepsilon^2}{I(\varepsilon)^{-2}}\right)=\mathcal{O}\left({|\log\varepsilon|^{-2}}\right)$. 
    \item We now provide an example in which $F(\varepsilon)$ dominates ${\varepsilon^2}{I(\varepsilon)^{-2}}$. Set $I(\varepsilon) = \varepsilon^{q}$ where $q = \tfrac{3\alpha +1}{2\alpha(1+\alpha)}.$ This choice {of} $I$ satisfies Assumptions~\ref{assumptions2}, and the resulting interface width and sharpness given by Theorem~\ref{maintheorem} are $\mathcal{O}\left(\varepsilon^{q}|\log\varepsilon|\right)$ and $\mathcal{O}\left(\varepsilon^{\frac{\alpha-1}{\alpha+1}}|\log\varepsilon|^\alpha\right)$, respectively. 
\end{enumerate}
\end{example}
\indent We often reference the Brownian analogue of Theorem~\ref{maintheorem} proved using probabilistic techniques in \cite{etheridge2017branching}. This result, originally due to Chen \cite{chen}, is given as follows (here we state the version found in \cite[Theorem~1.3]{etheridge2017branching}).\reqnomode
\begin{theorem}[\cite{chen}] \label{browniantheorem} Let $u^\varepsilon$ solve 
    \begin{align}\label{BMeth} \begin{cases} \partial_t u^\varepsilon = \Delta u^\varepsilon + {\varepsilon^{-2}}u^\varepsilon(1-u^\varepsilon)(2u^\varepsilon-1), \ \ \ t\geq 0, \ x\in \RR^\dd \\  u^\varepsilon(0,x) = p(x), \ \ \ x\in \RR^\dd\end{cases} \end{align} with initial condition $p$ satisfying Assumptions~\ref{assumptions1}. Let $\mathscr{T}$ and $d(x,t)$ be as above. Fix $T^*\in (0,\mathscr{T})$ and $k\in \NN$. Then there exists $\varepsilon_\dd(k), a_\dd(k), c_\dd(k)>0$ such that, for all $\varepsilon\in (0, \varepsilon_\dd)$ and $t$ with $a_\dd \varepsilon^2|\log \varepsilon|\leq t \leq T^*,$\begin{enumerate}[(1)]
        \item for $x$ such that $d(x,t)\geq c_\dd \varepsilon|\log \varepsilon|,$  we have $u^\varepsilon(t,x) \geq 1-\varepsilon^k$,
        \item for $x$ such that $d(x,t)\leq -c_\dd \varepsilon|\log\varepsilon|,$ we have $u^\varepsilon(t,x)\leq \varepsilon^k.$
    \end{enumerate}
\end{theorem}
\begin{remark} The interface width $\mathcal{O}(\varepsilon |\log\varepsilon|)$ in Theorem~\ref{browniantheorem} is not achievable in Theorem~\ref{maintheorem}, but we may approximate it by choosing $I(\varepsilon)= \varepsilon^{\delta}$ where $\tfrac{1}{\alpha}<\delta <1$, and considering $\delta\to 1$. The interface width and sharpness in this case are $\mathcal{O}(\varepsilon^\delta |\log\varepsilon|)$ and $\varepsilon^{2-2\delta}$, respectively. \end{remark}
\indent We note some key differences between our result (Theorem~\ref{maintheorem}) and Theorem~\ref{browniantheorem}. Firstly, in Theorem~\ref{browniantheorem}, the width of the solution interface is $\mathcal{O}(\varepsilon|\log\varepsilon|)$, compared to the strictly larger width of $\mathcal{O}(I(\varepsilon)|\log\varepsilon|)$ in the fractional setting. Secondly, the sharpness of the interface in Theorem~\ref{browniantheorem}, $\varepsilon^k$, is much better than the sharpness in Theorem~\ref{maintheorem}. Both of these differences are consistent with our intuition by considering the (soon to be formalised) probabilistic representation of solutions to the ordinary and fractional Allen--Cahn equations in terms of Brownian and $\alpha$-stable motions, respectively. In the stable case, the rare large jumps of the $\alpha$-stable motion act to `fatten' the interface compared to the Brownian case. While the effect of these large jumps is small enough that we still observe mean curvature flow (as $\alpha>1$), it does result in a much less sharp interface. \\
\indent As we have mentioned, to prove our result, we adapt techniques from Etheridge et al. \cite{etheridge2017branching}. However, our work significantly differs from that of Etheridge et al. since the method of proof in \cite{etheridge2017branching} cannot be applied to the stable setting in a straightforward way. To overcome this, we devote a significant portion of this paper to constructing a series of couplings that enable us to compare the probabilistic dual to equation~(\ref{mainequation22}) to another quantity for which the proofs in \cite{etheridge2017branching} can more easily be adapted. 
\section{Majority voting in one dimension}\label{ch:2}
\subsection{A probabilistic dual}
 In this section, we define a probabilistic dual to the scaled fractional Allen--Cahn equation (\ref{mainequation22}), which is very similar to the probabilistic dual to the ordinary Allen--Cahn equation developed in \cite{etheridge2017branching}.  In our setting, we consider a ternary-branching $\alpha$-stable motion in which each individual, independently, follows an $\alpha$-stable motion, until the end of its exponentially distributed lifetime (with mean $\varepsilon^2$) at which point it splits into three particles. Let $Y(t)$ denote a $\dd$-dimensional $\alpha$-stable process and $\Y(t)$ denote a $\dd$-dimensional historical ternary branching $\alpha$-stable motion. That is, $\Y(t)$ traces out the space-time trees that record the position of all particles alive at time $s$ for $s\in [0,t]$. Throughout this work, we adopt the following convention. Recall that $\sigma_\alpha := \left(\tfrac{2-\alpha}{\alpha}\right)^{\frac{\alpha}{2}}\Gamma\left(1-\tfrac{\alpha}{2}\right).$ \leqnomode \begin{assumption}\label{assumption3}
All $\alpha$-stable motions have generator  $-\sigma_\alpha I(\varepsilon
)^{\alpha-2}(-\Delta)^{\tiny \frac{\alpha}{2}}$ for a fixed $\alpha \in (1,2).$
 \end{assumption}

\reqnomode
To record the genealogy of the process we employ the Ulam-Harris notation to label individuals by elements of $\mathcal{U} = \bigcup_{m=0}^\infty \{1, 2, 3\}^m.$ For example, $(3,1)$ represents the particle which is the first child of the third child of the initial ancestor $\emptyset$. Let $N(t)\subset \mathcal{U}$ denote the set of individuals alive at time $t$.\\
\indent We call $\T$ a \textit{time-labelled ternary tree} if $\T$ is a finite subset of $\mathcal{U}$ with each internal vertex $v$ labelled with a time $t_v>0$, where $t_v$ is strictly greater than the label of the parent vertex of $v$. Ignoring the spatial position of individuals, we see that $\Y(t)$ traces out a time-labelled ternary tree which associates to each branch point the time of the branching event. Let $Y_i(t)$ be the $\alpha$-stable motion traced out by individual $i$ in $\Y(t)$. Denote the time-labelled ternary tree traced out by $\Y(t)$ by $\TrY.$ 
\begin{definition}[$\VV$]\label{definitionmajority} Fix $p:\RR^\dd \to [0,1]$ and define the \textit{majority voting procedure} on $\TrY$ as follows.
\begin{enumerate}[(1)]
    \item each leaf $i$ of $\TrY$ independently votes $1$ with probability $p(Y_i(t))$ and otherwise votes $0$;
    \item at each branching event in $\TrY$, the vote of the parent particle $j$ is given by the {majority vote} of its offspring $(j,1), (j,2), (j,3)$.
\end{enumerate}
This voting procedure runs inward from the leaves of $\TrY$ to the root $\emptyset$. Under this voting procedure, define $\VotemajY$ to be the vote associated to the root $\emptyset$ of the ternary branching stable tree. \end{definition} 
For $x\in \RR^\dd$, we shall write $\PP_x$ and $\EE_x$ for the probability measure and expectation associated to the law of a stable motion starting at $x$. Write $\PP_x^{\varepsilon}$ for the probability measure under which $(\Y(t), t\geq 0)$ has the law of the historical process of a ternary branching $\alpha$-stable motion in $\RR^\dd$ with branching rate $\varepsilon^{-2}$ started from a single particle at location $x$ at time $0$. We write $\EE_x^\varepsilon$ for the corresponding expectation. We emphasise that the variable $\varepsilon$ used to define the speed of the stable process, $I(\varepsilon)^{\alpha-2}$, is the same variable $\varepsilon$ that defines the branch rate, $\varepsilon^{-2}.$ Then the root vote $\VotemajY$ provides us with a dual to equation~(\ref{mainequation22}) in the following sense. \begin{theorem}\label{votedual}
Let $p:\RR^\dd\to [0,1]$. Then \begin{align}\label{milk} u^\varepsilon(t,x) := \PP^\varepsilon_x[\VotemajY=1]\end{align} is a solution to equation~(\ref{mainequation22}) with initial condition $u^\varepsilon(0,x) = p(x)$. 
\end{theorem}

\begin{proof}[Sketch of proof of Theorem~\ref{votedual}] This proof follows closely that of \cite[Theorem~2.2]{etheridge2017branching}. In this proof we neglect the superscript $\varepsilon$ on $\PP^\varepsilon_x, \EE^\varepsilon_x$ and $u^\varepsilon$. Let $u$ be as in (\ref{milk}) and consider $u(t+\delta t, x)$. If $\tau$ is the time of the first branching event, we have  \begin{align} u(t+\delta t, x) &= \PP_x[\VV(\Y(t+\delta t))=1 \conditional \tau< \delta t]\, \PP[\tau\leq \delta t]\nonumber \\ 
& \ \  +  \PP_x[\VV(\Y(t+\delta t))=1 \conditional \tau> \delta t]\, \left(1-\PP[\tau\leq \delta t]\right).\label{week} \end{align} We will consider each of the terms in (\ref{week}) separately. Let $V_1, V_2,$ and $V_3$ be the votes of the three offspring created at time $\tau$. Conditional on $\tau\leq \delta t$, the probability of a second branching event before time $\delta t$ is $\mathcal{O}(\delta t)$, so for $s\leq \delta t$ $$\EE_x\left[V_1 \conditional (\tau, Y_\tau) = (s, y)\right] = \EE_y\left[u(t, Y_{\delta t -s})\right] + \mathcal{O}(\delta t). $$ Assuming sufficient regularity of $u$ (which follows from that of the fractional heat semigroup \cite{bassregularity}), it follows that $$\EE_x[V_1 \conditional \tau< \delta t] = u(t,x) + \mathcal{O}(\delta t).$$  
Since the root vote of $\T(\Y(t))$ will equal one if and only if at most one of $V_1, V_2,$ and $V_3$ is zero, it follows by conditional independence of each $V_i$ given $(\tau, Y_\tau)$ that \begin{align*}
 \PP_x[\VV(\Y(t))=1 \conditional \tau\leq \delta t] &= u(t,x)^3 +3u(t,x)^2(1-u(t,x)) + o(1)\\ 
 &= u(t,x)(1-u(t,x))(2u(t,x)-1) + u(t,x) + o(1). \end{align*}
For the second term in (\ref{week}) note that $$\PP_x[\VV(\Y(t+\delta t))=1 \conditional \tau> \delta t] = \EE_x[\PP_{Y_{ \delta t}}(\VV(\Y( t))=1)] = \EE_x[u(t, Y_{\delta t})] $$ by the Markov property of $\Y$ at time $\delta t$. Since $\PP[\tau\leq \delta t] \approx \delta t \varepsilon^{-2} + \mathcal{O}(\delta t^2)$, putting the above estimates into (\ref{week}) gives \begin{align*}
  &  \lim_{\delta t \to 0} \frac{u(t+\delta t,x) -u(t,x)}{\delta t} \\
  &=\lim_{\delta t \to 0}\frac{\EE_x(u(t, Y_{\delta t}))-u(t,x)}{\delta t}+ {\varepsilon^{-2}}u(t,x)(1-u(t,x))(2u(t,x)-1) \\
    &=- \sigma_\alpha I(\varepsilon)^{\alpha-2}(-\Delta)^{\tiny \frac{\alpha}{2}} u(t,x)  +{\varepsilon^{-2}}u(t,x)(1-u(t,x))(2u(t,x)-1). \ \ \ \ \ \ \ \ \ \ \ \ \ \ \ \ \ \ \ \ \ \ \qedhere
\end{align*}
 \end{proof} 

 \noindent With this in mind, we can restate Theorem~\ref{maintheorem} as follows.
\begin{theorem}\label{maintheorem2}
Fix a function $I$ satisfying Assumptions~\ref{assumptions2}. Suppose the initial condition $p$ satisfies Assumptions~\ref{assumptions1}. Let $\mathscr{T}$ and $d(x, t)$ be as in Section~\ref{secone}, $F$ be as in (\ref{defnofF}) and fix $T^*\in (0, \mathscr{T})$. Then there exist $\varepsilon_\dd(\alpha),$ $a_\dd(\alpha),$ $c_\dd(\alpha),$  $m>0$  such that, for $\varepsilon \in (0, \varepsilon_\dd)$ and $a_\dd \varepsilon^2 |\log\varepsilon|\leq t\leq T^*$, \begin{enumerate}[(1)]
\item for $x$ with $d(x, t)\geq c_\dd I(\varepsilon)|\log\varepsilon|$, $\PP_x^\varepsilon\left[\VV(\Y(t)) = 1\right]\geq 1- m\dfrac{\varepsilon^2}{I(\varepsilon)^2}-mF(\varepsilon),$
\item for $x$ with $d(x, t) \leq - c_\dd I(\varepsilon)|\log\varepsilon|$, $\PP_x^\varepsilon\left[\VV(\Y(t)) = 1\right]\leq m\dfrac{\varepsilon^2}{I(\varepsilon)^2}+mF(\varepsilon)$.
\end{enumerate}
\end{theorem}     \indent For the sake of completeness, we mention the Brownian analogue of Theorem~\ref{votedual} and Theorem~\ref{maintheorem2} from \cite{etheridge2017branching}. There, the authors considered a historical ternary branching $\dd$-dimensional Brownian motion $(\boldsymbol{W}(t), t\geq 0)$ with branching rate $\varepsilon^{-2}$. Let $\PP_x^\varepsilon$ and $\VV$ be defined as above but for the process $(\boldsymbol{W}(t), t\geq 0)$. Then, by \cite[Theorem~2.2]{etheridge2017branching}, given $p:\RR^\dd\to [0,1],$ \begin{align}\label{prob_brown}u^\varepsilon(t,x):= \PP^\varepsilon_x[\VV(\boldsymbol{W}(t))=1]\end{align} is a solution to equation~(\ref{BMeth}), and Theorem~\ref{browniantheorem} can be restated in terms of $\PP_x^\varepsilon[\VV(\boldsymbol{W}(t))=1]$. We will refer to the restatement of Theorem~\ref{browniantheorem} in terms of $\PP_x^\varepsilon[\VV(\boldsymbol{W}(t))=1]$ as the `probabilistic version of Theorem~\ref{browniantheorem}'. \begin{remark}
     The duality representation (\ref{prob_brown}) developed in Etheridge et al. \cite{etheridge2017branching} is similar to that of solutions to the Fisher--KPP equation in terms of binary branching Brownian motions developed by Skorohod and McKean \cite{skorokhod1964branching, McKean}. The dual described in \cite{etheridge2017branching} was novel in that it generalised Skorohod and McKean's result to equations with an Allen--Cahn type non-linearity. It was later found in the Master's thesis of \cite{zach}, and independently in \cite{an2022voting}, that a semilinear heat equation can be expressed in this way if and only if the nonlinearity of the equation belongs to a certain very general family of polynomials.
 \end{remark}
\begin{notation}\nonumber
It will be convenient to distinguish between one-dimensional and multidimensional $\alpha$-stable processes. We adopt the convention that $X(t)$ will denote the one-dimensional $\alpha$-stable process, with the corresponding historical branching stable process denoted by $\X(t).$ When $\dd>1$, we denote the $\alpha$-stable process by $Y(t)$ and denote the corresponding historical branching stable process by $\Y(t)$. 
\end{notation}
\subsection{Remarks on the choice of scaling}\label{choice_scaling_sec}
Now that we have described the probabilistic dual to the fractional Allen--Cahn equation, we can explain the origin of the scaling taken in equation~(\ref{mainequation22}). As we will see, this choice of scaling is intimately linked to our strategy of proof for Theorem~\ref{maintheorem2}. To explain this, we will need to consider stable processes run at varying speeds. For this reason, in this section (and this section only) we do \textit{not} adopt Assumption~\ref{assumption3}, so the speeds of all stable process, stable subordinators, and historical stable trees will be made explicit.\\ 
\indent Let $(Y_t^\varepsilon)_{t\geq 0}$ be a $\dd$-dimensional $\varepsilon$-truncated $\alpha$-stable process with $\alpha\in (1,2)$, i.e. $Y_t^\varepsilon$ is a L\'evy process with L\'evy measure given (up to a multiplicaitave constant) by $$\nu(dx) = |x|^{-\alpha - \dd} \mathbbm{1}_{|x|<\varepsilon}.$$ By \cite[Proposition~3.2]{cohen2007gaussian}, for some $c_\alpha>0$, \begin{align}\label{driftwood}c_\alpha \varepsilon^{\frac{\alpha}{2}-1}Y^\varepsilon_t \xrightarrow[]{w} W_t \ \text{ as } \varepsilon\to 0\end{align} where $(W_t)_{t\geq 0}$ is a standard $\dd$-dimensional Brownian motion and $\xrightarrow[]{w}$ denotes weak convergence. It is straightforward to show using characteristic exponents and the L\'evy-Khintchine formula that 
\begin{align}\label{scalingdiscussion}
    \varepsilon^{\frac{\alpha-2}{\alpha}}Y^{\varepsilon^{2/\alpha}}_t \stackrel{D}{=} Y^{\varepsilon}_{\varepsilon^{\alpha-2}t},\end{align} where $(Y^{\varepsilon^{2/\alpha}}_t)_{t\geq 0}$ denotes the $\varepsilon^{2/\alpha}$-truncated $\alpha$-stable process. Therefore, by replacing $\varepsilon$ by $\varepsilon^{2/\alpha}$ in (\ref{driftwood}), and applying (\ref{scalingdiscussion}), we see that
\begin{align}\label{epsilon_limit}c_\alpha Y^\varepsilon_{\varepsilon^{\alpha-2}t} \xrightarrow[]{w} W_t \ \text{ as } \varepsilon\to 0.\end{align}
\indent More generally, one could replace $\varepsilon$ in (\ref{epsilon_limit}) by a function $I(\varepsilon)$, satisfying $I(\varepsilon)\to 0$ as $\varepsilon\to 0$. The $I(\varepsilon)$-truncated stable process $Y^{I(\varepsilon)}_{I(\varepsilon)^{\alpha-2}t}$ will approximate a Brownian motion, so it is reasonable to expect that the probabilistic version of Theorem~\ref{browniantheorem} holds when the branching Brownian motion is replaced by a branching $I(\varepsilon)$-truncated stable process, run at speed $I(\varepsilon)^{\alpha-2}$. Therefore to prove Theorem~\ref{maintheorem2}, it should suffice to show:\begin{enumerate}[Step~1:, align=left]
\itshape \item \upshape the probabilistic version of Theorem~\ref{browniantheorem} holds when $\boldsymbol{W}(t)$ is replaced by a $\dd$-dimensional ternary branching $I(\varepsilon)$-truncated stable tree run at speed $I(\varepsilon)^{\alpha-2}$, denoted $\Y^{I(\varepsilon)}(I(\varepsilon)^{\alpha-2}t)$; 
\itshape \item \upshape there exists a coupling of the root votes of $\Y(I(\varepsilon)^{\alpha-2}t)$ and $\Y^{I(\varepsilon)}(I(\varepsilon)^{\alpha-2}t)$ in such a way that Step~1 implies Theorem~\ref{maintheorem2}.\end{enumerate}  The purpose of this two-step approach is that, by using a truncated stable process (which is not heavy tailed) we can more readily adapt proofs from the Brownian case of \cite{etheridge2017branching}.\\
\indent The discussion above explains why the fractional Laplacian in equation~(\ref{mainequation22}) is sped up by a factor of $I(\varepsilon)^{\alpha-2}$: this is the precise speed that the truncated stable process must run at in order for it to approximate a Brownian motion, allowing us to prove Step~1. However, we have not addressed our choice of $I(\varepsilon)$ as outlined by Assumptions~\ref{assumptions2}. In particular, we require $\lim_{\varepsilon\to 0} \frac{\varepsilon}{I(\varepsilon)}=0$, so $I(\varepsilon)\neq \varepsilon$, which might be unexpected in view of the limit (\ref{epsilon_limit}), and the result of \cite{imbert2009phasefield} where the scaling $I(\varepsilon)=\varepsilon$ was used. This is because we must carefully balance two opposing effects: the truncation level must be small enough that the truncated motion is `similar' to a Brownian motion (to prove Step~1), but it must be large enough so that the truncated motion and original stable motion are themselves `similar' (to prove Step~2). In particular, the truncation level must be large enough so that the probability of an individual in the ternary stable tree $\Y$ making a jump larger than the truncation is sufficiently small, enabling $\Y$ and $\Y^{I(\varepsilon)}$ to be coupled.\\
\indent Although Step~1 will hold when $I(\varepsilon)=\varepsilon$, Step~2 does not. More concretely, consider Assumption~\ref{assumptions2}~\ref{assumptions2_B}: $\lim_{\varepsilon \to 0}\frac{\varepsilon^2}{I(\varepsilon)^2}{|\log \varepsilon|} = 0.$ Recall that the ternary branching stable motion branches at rate $\varepsilon^{-2}$. Suppose $\tau\sim \mathit{Exp}(\varepsilon^{-2})$ is the time of one such branching event. Then, for $k\in \NN$, conditional on $\tau \leq k\varepsilon^2|\log\varepsilon|$ (which happens with probability $1-\varepsilon^k$), an individual in the tree $\Y(I(\varepsilon)^{\alpha-2}t)$ is expected to make, at most, \begin{align}\label{numjump}k\frac{\varepsilon^2}{I(\varepsilon)^2}|\log\varepsilon|\end{align}jumps larger than $I(\varepsilon)$ in its lifetime, because the arrival rate of jumps larger then $I(\varepsilon)$ made by a stable process run at speed $I(\varepsilon)^{\alpha-2}$ is
\begin{align}\label{donthaveany}\mathcal{O}\left( {I(\varepsilon)^{\alpha-2}}\int_{I(\varepsilon)}^\infty {x^{{-\alpha}-1}}dx \right)= \mathcal{O}\left({I(\varepsilon)^{-2}}\right).\end{align} This follows because the arrival rate of jumps larger than $I(\varepsilon)$ made by the $\dd$-dimensional process $Y_{I(\varepsilon)^{\alpha-2}t}$ is proportional to $I(\varepsilon)^{\alpha-2}\int_{x\in \RR^\dd: |x|> I(\varepsilon)}\nu(dx)$ where $\nu(dx) = |x|^{-\alpha - \dd}dx$ is the L\'evy measure of $(Y_t)_{t\geq 0}$. To prove Step~2, each individual stable motion in $\Y(I(\varepsilon)^{\alpha-2}t)$ should approximate (asymptotically in $\varepsilon$) an $I(\varepsilon)$-truncated process, so the quantity (\ref{numjump}) should converge to zero with $\varepsilon$, which is precisely Assumption~\ref{assumptions2}~\ref{assumptions2_B}. The remaining assumptions on $I(\varepsilon)$ from Assumptions~\ref{assumptions2} arise from several technical lemmas needed to prove Theorem~\ref{maintheorem2}.\\ 
\indent To the best of our knowledge, Theorem~\ref{maintheorem2} is the first result on the solution interface of equation~(\ref{mainequation22}) with our chosen scaling. The work of \cite{imbert2009phasefield} suggests that our result should hold even when $I(\varepsilon)=\varepsilon.$ However, it does not seem likely that we can achieve the $I(\varepsilon)=\varepsilon$ scaling using our current method of proof and, conversely, it is unclear if the method of proof in \cite{imbert2009phasefield} could be adapted to handle our chosen scaling.\\
\indent The two-step proof described above motivates our choice of scaling. However, in reality, we opt to work with $W\left(R^{I(\varepsilon)^2}_{I(\varepsilon)^{\alpha-2}t}\right),$ a Brownian motion subordinated by an $I(\varepsilon)^2$-truncated $\frac{\alpha}{2}$-stable subordinator run at speed $I(\varepsilon)^{\alpha-2}$, instead of the $I(\varepsilon)$-truncated stable process $Y^{I(\varepsilon)}_{I(\varepsilon)^{\alpha-2}t}$. The central idea of our proof remains the same as before, however, we found the subordinated process easier to work with (more on this below). Moreover, the error made by approximating the stable process $Y_{I(\varepsilon)^{\alpha-2}t}$ by $W\left(R^{I(\varepsilon)^2}_{I(\varepsilon)^{\alpha-2}t}\right)$ is roughly equal to the error we would obtain using $Y^{I(\varepsilon)}_{I(\varepsilon)^{\alpha-2}}$. Recall that $Y_{I(\varepsilon)^{\alpha-2}t}\stackrel{D}{=} W(R_{I(\varepsilon)^{\alpha-2}t})$, a Brownian motion subordinated by an $\frac{\alpha}{2}$-stable subordinator. The arrival rate of jumps larger that $I(\varepsilon)^2$ made by the $\frac{\alpha}{2}$-stable subordinator is $$ \mathcal{O}\left({I(\varepsilon)^{\alpha-2}}\int_{I(\varepsilon)^2}^\infty {x^{-\frac{\alpha}{2}-1}}dx\right) = \mathcal{O}\left({I(\varepsilon)^{-2}}\right),$$ which is the same order as the arrival rate of jumps larger than $I(\varepsilon)$ made by a stable process from (\ref{donthaveany}). \\
\indent Like the truncated stable process, the subordinated Brownian motion is not heavy tailed, and its explicit description in terms of a Brownian motion allows us to more elegantly adapt the Brownian proof of \cite{etheridge2017branching}. In Step~1 of our approach, it is much more straightforward to compare this subordinated Brownian motion to a standard Brownian motion, than it would have been to compare a truncated stable process to a Brownian motion. This error made by estimating the subordinated Brownian motion $W\left(R^{I(\varepsilon)^2}_{I(\varepsilon)^{\alpha-2}t}\right)$ by a standard Brownian motion can be quantified by considering the difference $$\left|R^{I(\varepsilon)^2}_{I(\varepsilon)^{\alpha-2}t}-t\right|.$$ We will see in Section~\ref{acouplingargument} that this difference is ultimately the source of the error term $F(\varepsilon)$ in our main result, Theorem~\ref{maintheorem}.

\begin{section}{Majority voting in one dimension}\label{sectionmajorityvotinginonedimension}
\indent In this section, we prove a one-dimensional analogue of Theorem~\ref{maintheorem2}, which will be used in the proof of Theorem~\ref{maintheorem2} in Section~\ref{ch:3}. This parallels the structure of proof for the Brownian result, Theorem~\ref{browniantheorem}.\\
\indent For each $x\in \RR$, let $\PP_x$ be the law of a one-dimensional $\alpha$-stable process $X(t)$ satisfying Assumption~\ref{assumption3} started at $x$, with corresponding expectation $\EE_x$. Write $\PP_x^\varepsilon$ for the probability measure under which $(\X(t), t\geq0)$ has the law of a one-dimensional historical ternary branching $\alpha$-stable motion, with branching rate $\varepsilon^{-2}$ started from a single particle at location $x$ at time $0$. Write $\EE_x^\varepsilon$ for the corresponding expectation. In accordance with Assumption~\ref{assumption3}, each particle in $(\X(t), t\geq 0)$ is assumed to run at speed $\sigma_\alpha I(\varepsilon)^{\alpha-2}$ for $\sigma_\alpha$ given in (\ref{defn_sig_alpha}). \\
\indent Define \begin{align}\label{xx} p_0(x)=\mathbbm{1}_{\{x \geq 0\}}.\end{align} With this choice of initial condition, under majority voting (Definition~\ref{definitionmajority}), the leaves of $\T(\X(t))$ will vote one if and only if they are on the right half line. Denote the root vote of $\T(\X(t))$ under majority voting with initial condition (\ref{xx}) by $\V(\X(t)):= \V_{p_0}(\X(t))$. The following result is the natural analogue of Theorem~\ref{maintheorem2} in one dimension.
\begin{theorem} \label{maintheorem1D}
Let $T^*\in (0,\infty)$. Suppose $I$ satisfies Assumptions~\ref{assumptions2}~\ref{assumptions2_A}-\ref{assumptions2_B}. Then there exist $c_1(\alpha), \varepsilon_1(\alpha)>0$ such that, for all $t\in [0,T^*]$ and all $\varepsilon \in (0, \varepsilon_1),$
\begin{enumerate}[(1)]
    \item for $x \geq c_1I(\varepsilon)|\log\varepsilon|$, we have $\PP_x^\varepsilon[\V(\X(t)) =1] \geq 1 - \frac{\varepsilon^2}{I(\varepsilon)^2},$
    \item for $x \leq -c_1I(\varepsilon)|\log\varepsilon|$,  we have \textup{$\PP_x^\varepsilon[\V(\X(t)) =1] \leq \frac{\varepsilon^2}{I(\varepsilon)^2}$}.
\end{enumerate}
\end{theorem}
This result tells us that, for positive $x$, `typical' leaves of the tree $\T(\X(t))$ based at $x$ are more likely to vote $1$ than
$0$. As we mentioned in our introduction, Theorem~\ref{maintheorem1D} is weaker than the actual one-dimensional result that will be used to prove Theorem~\ref{maintheorem2} (at this stage, we have not developed the technical jargon needed to state it). This stronger result (Theorem~\ref{mainteo1dmarked}) will be developed in Section~\ref{couplingmarkedsystems}, and shown to imply Theorem~\ref{maintheorem1D}.
\begin{remark} There is some evidence in the literature that the interface width in Theorem~\ref{maintheorem1D} could be improved. Let $f$ be any bistable nonlinearity and consider the one-dimensional equation \begin{align}\label{onedsystem}
   \begin{cases} (-\Delta)^{\frac{\alpha}{2}}u(y)  = f(u(y)) \ \ \forall \, y \in \RR,\\
     \lim \limits_{y\to  \infty} u(y) =  1,\ \lim \limits_{y\to - \infty} u(y) = 0.\end{cases}
\end{align} Then, by \cite[Proposition~3.2]{gui2015traveling}, a solution $u \in C^2(\RR)$ to (\ref{onedsystem}) satisfies $$\frac{A}{y^{\alpha}}\leq  1 - u(y) \leq \frac{B}{y^{\alpha}}$$ for all $y>1$ and some $A, B > 0$ (with a similar inequality holding if $y\leq -1$). We rewrite this equation under our scaling by setting $z = \varepsilon^{\frac{2}{\alpha}} I(\varepsilon)^{1-\frac{2}{\alpha}} y,$ and obtain
\begin{align*} \frac{A \varepsilon^2 I(\varepsilon)^{\alpha-2}}{z^\alpha} \leq 1-u(z)\leq \frac{B \varepsilon^2 I(\varepsilon)^{\alpha-2}}{z^\alpha}.\end{align*}
Therefore if $z\geq I(\varepsilon)$, $1-u(z)$ is of order $\frac{\varepsilon^2}{I(\varepsilon)^2}$, the interface sharpness from Theorem~\ref{maintheorem1D}. This suggests that Theorem~\ref{maintheorem1D} should hold for an interface of width $I(\varepsilon)$, and that this would be the narrowest width possible for the given interface sharpness. However, we were only able to prove Theorem~\ref{maintheorem1D} for an interface width of order $I(\varepsilon)|\log\varepsilon|$.
\end{remark}

\indent The choice of initial condition $p_0(x)=\Indicator$ affords us several useful inequalities. First, for all $x_1, x_2\in \RR$ with $x_1\leq x_2,$ we have \begin{align}\label{monotonicity1}\PP_{x_1}^\varepsilon[\V(\X(t))=1]\leq \PP_{x_2}^\varepsilon[\V(\X(t))=1].\end{align}
For any time-labelled ternary tree $\T$, write 
\[ \PP_x^t(\T) := \PP^\varepsilon_x[\V(\X(t))=1 \conditional \TrX = \T].\]
It then follows by symmetry of $\alpha$-stable motions and the definition of $p_0(x)$ that, for any $x\in \RR$ and $t>0$,
\begin{align} \label{symetrytree1}\PP_x^t(\T) = 1- \PP_{-x}^t(\T). \end{align} Setting $x=0$ in equation~(\ref{symetrytree1}) yields $\PP_0^t(\T) = \frac{1}{2}$, so by monotonicity (\ref{monotonicity1}) \begin{align}\label{tannn}\PP_x^t(\T)\geq \tfrac{1}{2} \text{ for } x>0, \text{  and } \, \PP_x^t(\T)\leq \tfrac{1}{2} \text{ for }  x<0. \end{align}
\indent It will be convenient in our later calculations to introduce notation for the majority voting system. Mimicking \cite{etheridge2017branching}, define the function $g:[0,1]^3\to [0,1]$ by \begin{align} \label{majorityvotingfunction} g(p_1,p_2,p_3) = p_1p_2p_3 + p_1p_2(1-p_3)+p_2p_3(1-p_1)+p_3p_1(1-p_2).\end{align} This is the probability that a majority vote gives the result $1$, in the special case when the three voters are independent and have probabilities $p_1, p_2$ and $p_3$ of voting $1$. We will abuse notation slightly and write $g(q):= g(q, q,q)$. Note that, for all $q\in [0,1],$ \begin{align}\label{symmetryofg}
    g(q) = 1-g(1-q). 
\end{align}
\subsection{A coupling of voting systems}\label{couplingmarkedsystems} 
In this section, we will couple the root vote of $\T(\X(t))$ under majority voting to the root vote of another ternary branching process under a different voting system. This other branching process will be a ternary branching subordinated Brownian motion, with subordinator given by a truncated $\tfrac{\alpha}{2}$-stable subordinator. We endow this process with a voting procedure that we call `marked majority voting'. Once we have achieved this coupling of root votes, we state a more general theorem in terms of this new branching process that will imply Theorem~\ref{maintheorem1D}. \\  
\indent The intuition behind {marked majority voting}, which we denote by $\V^\times$, is straightforward. However, formally proving a coupling of the majority and marked majority systems $\V$ and $\V^\times$ is more challenging. To aid us with this, we introduce an intermediate voting system $\widehat{\V}$ that can be readily compared to both voting systems. We call this the {`exponentially marked'} voting procedure. The definition of the exponentially marked voting procedure, together with a proof that it couples with the ordinary majority voting system in the appropriate sense (Theorem~\ref{teo:ineqmark}) make up the content of Section~ \ref{expmarkedsection}. After this, in Section \ref{sectionnmarked}, we define the marked majority voting procedure that will be carried through the one-dimensional proof, and prove (using the intermediate voting system) that it can be coupled to majority voting. Theorem~\ref{maintheorem1D} will then be a consequence of a more general theorem stated in terms of the marked voting system on the subordinated Brownian tree (Theorem~\ref{mainteo1dmarked}), which will be proved in Section~\ref{saltchips}.
\subsubsection{Exponentially marked voting}\label{expmarkedsection}
\indent \indent In this section we will couple the majority voting system on $\T(\X(t))$ with the exponentially marked voting system defined on a ternary branching subordinated Brownian motion. Fix $\varepsilon>0$ throughout. We will consider an $I(\varepsilon)^2$-truncated $\tfrac{\alpha}{2}$-stable subordinator denoted $R^\varepsilon_t$. Assumption~\ref{assumption3} will also be adopted for all stable subordinators. To be precise, if $$\nu(dx) := \frac{\alpha}{2 \Gamma\left(1-\tfrac{\alpha}{2}\right)}x^{-1-\frac{\alpha}{2}}dx$$ is the L\'evy measure of the $\tfrac{\alpha}{2}$-stable subordinator, then the L\'evy measure of $R^\varepsilon_t$ is given by \begin{align*} \sigma_\alpha I(\varepsilon)^{\alpha-2}\nu(dx)\mathbbm{1}_{0\leq x\leq \frac{2-\alpha}{\alpha} I(\varepsilon)^2}.\end{align*} In using the notation $R^\varepsilon_t$ we suppress the true speed of the process its truncation level, which was made explicit previously in Section~\ref{choice_scaling_sec}. Henceforth we will use the former suppressed notation. Moreover, although we technically truncate at level $\frac{2-\alpha}{\alpha} I(\varepsilon)^2$, we shall refer to this simply as the $I(\varepsilon)^2$-truncated stable subordinator. Let $\boldsymbol{B}_{R^\varepsilon}(t)$ denote the historical process of a ternary branching $R^\varepsilon_t$-subordinated Brownian motion with branching rate $\varepsilon^{-2}$.  Unless stated otherwise, all subordinators in this work will be zero at time zero.\\
\indent Let us now make precise the form of the coupling that we desire. As before, let $\mathbb{V}(\X(t))$ denote the root vote of $\T(\X(t))$ under majority voting (Definition~\ref{definitionmajority}). We will define a voting system on $\T(\Xtrunc(t))$ with root vote $\widehat{\mathbb{V}}(\Xtrunc(t))$ satisfying \begin{align}\label{cupcakes}\PP^\varepsilon_x\left[\VotemajXnop=1\right] \geq \PP^\varepsilon_x\left[\Voteprenop=1\right] \ \text{ for all } \ x \geq 0, \end{align}
 with the reverse inequality holding for $x< 0$. 
 Having obtained this, it will suffice to prove the analogue of Theorem~\ref{maintheorem1D} with $\mathbb{V}(\X)$ replaced by $\widehat{\mathbb{V}}(\Xtrunc).$ In this way, we will have incorporated the problematic `large jumps' of the $\alpha$-stable process $\X(t)$ into the voting system $\widehat{\mathbb{V}}$.\\ 
\indent To define the voting system on $\T(\Xtrunc(t))$, consider a collection of independent $\tfrac{\alpha}{2}$-stable subordinators, $\{R_i\}_{i\in M(t)}$, where $M(t)$ denotes the set of individuals that have ever been alive in $\T(\X(t))$. For each $i\in M(t)$, let $\tau^{\times}_i$ be the first time that $R_i$ makes a jump of size larger than $\frac{2-\alpha}{\alpha}I(\varepsilon)^2$ (these are the exponential times after which the voting system is named). Explicitly, 
$$\tau^{\times}_i := \inf\left\{t\geq0 : |R_i(t)-R_i(t-)| > \tfrac{2-\alpha}{\alpha}I(\varepsilon)^2\right\}. $$ For each $i$, $\tau^{\times}_i$ is exponentially distributed with parameter 
\begin{align*}
&\int_{\frac{2-\alpha}{\alpha}I(\varepsilon)^2}^\infty  \sigma_\alpha I(\varepsilon)^{\alpha-2}\nu(dx)\mathbbm{1}_{0\leq x \leq \frac{2-\alpha}{\alpha}I(\varepsilon)^2}\\ &=  \tfrac{\alpha}{2}\left(\tfrac{2-\alpha}{\alpha}\right)^{\frac{\alpha}{2}} I(\varepsilon)^{\alpha-2}\int_{\frac{2-\alpha}{\alpha}I(\varepsilon)^2}^\infty {x^{-\frac{\alpha}{2}-1}} dx \\ &= I(\varepsilon)^{-2}
 \end{align*} using the definition of $\sigma_\alpha$ from (\ref{defn_sig_alpha}) (indeed, we chose $\sigma_\alpha$ so that the above equality would hold). Therefore $\{\tau^{\times}_i\}_{i\in M(t)}$ is a family of i.i.d. $\mathit{Exp}(I(\varepsilon)^{-2})$ variables. We associate to the particle in $\T(\Xtrunc(t))$ with label $i$ their lifetime, $\tau_i\wedge t$, for $\tau_i \sim \mathit{Exp}(\varepsilon^{-2}).$ 
\begin{definition}[$\widehat{\mathbb{V}}_p$] Fix $\varepsilon>0$. For $p:\RR\to [0,1]$, define the \textit{exponentially marked voting procedure} on $\Trpre$ as follows. 
\begin{enumerate}[(1)]
    \item Each individual $B_i(R^\varepsilon_i)$ is said to be \textit{marked} if $\tau^{\times}_i <\tau_i.$ Each marked individual votes $1$ with probability $\tfrac{1}{2}$ and otherwise votes $0$.
    \item Each unmarked leaf $i$ of $\Trpre$, independently, votes $1$ with probability \\ $p(B_i(R^\varepsilon_i(t)))$ and otherwise votes $0$. 
    \item At each branch point in $\Trpre$, if the parent particle $k$ is unmarked, she votes according to the majority vote of her three offspring $(k,1), (k,2)$ and $(k,3)$.
\end{enumerate} 
\label{definitionnonmarkov} Under this voting procedure, define $\Votepre$ to be the vote associated to the root $\emptyset$ of $\T(\Xtrunc(t))$. \end{definition}
When an individual in $\T(\Xtrunc(t))$ is marked, its vote is independent of the votes of its ancestors. Therefore if at least two individuals born at the same branching event are marked, the vote of their parent is independent of all of its ancestors, making it effectively random. Reassuringly, this scenario is very unlikely, since down a `typical' line of descent in $\T(\Xtrunc(t))$, two individuals will not be marked at the same branching event.\\  
\indent We now describe the intuition behind the exponentially marked voting procedure. Recall that an $\frac{\alpha}{2}$-stable subordinated Brownian motion is equal in distribution to an $\alpha$-stable process, so we may consider the historical ternary branching $\tfrac{\alpha}{2}$-stable subordinated Brownian motion $\boldsymbol{B}_R(t)$ in place of $\X(t)$. Suppose the trees $\T(\boldsymbol{B}_R(t))$ and $\T(\Xtrunc(t))$ rooted at $x>0$ have been generated up to time $t$ and that they have the same branching structure, written as $\T(\Xtrunc(t)) = \T(\boldsymbol{B}_R(t))$. Then each individual $B_i(R_i)$ in $\T(\boldsymbol{B}_R(t))$ can be associated to the individual $B_i(R^\varepsilon_i)$ in $\T(\Xtrunc(t))$.  If the subordinator $R_i$ has not made a large jump (i.e. a jump bigger than $I(\varepsilon)^2$) in its lifetime ($\tau_i^\times \geq\tau_i)$, then $B_i(R^\varepsilon_i)$ votes in the same way as $B_i(R_i)$ according to majority voting. However, if $R_i$ does make a large jump ($\tau_i^\times <\tau_i$), then, since $x>0$, $B_i(R_i)$ is more likely to jump into right-half line than the left. Therefore the vote of $B_i(R_i)$ should be one with probability strictly greater than $1/2$. In contrast, when $\tau_i^\times<\tau_i$, $B_i(R^\varepsilon_i)$ votes one with probability exactly $1/2$, which reduces the probability that the root vote of $\T(\boldsymbol{B}_{R^\varepsilon}(t))$ will equal one, so we expect (\ref{cupcakes}) to hold.\\
\indent Now, instead of considering the initial condition  $p_0(x) = \mathbbm{1}_{\{x\geq 0\}}$ as we did for majority voting, we will use \begin{align}\label{phat}\widehat{p}_0(x) = u_+ \mathbbm{1}_{\{x \geq 0\}} + u_- \mathbbm{1}_{\{x \leq 0\}},\end{align} where $0< u_- < u_+ < 1$ satisfy $1-u_+ = u_-$. We will fix a choice of $u_-$ and $u_+$ later (see (\ref{asympfixedpoint1}) and (\ref{asympfixedpoint2})). For this choice of initial condition, we write $\widehat{\mathbb{V}} := \widehat{\mathbb{V}}_{{\widehat{p}}_0}$. Noting that $\widehat{p}_0$ is symmetric, for any $x_1\leq x_2 \in \RR$, \begin{align}\label{awake} \PP^\varepsilon_{x_1}\left[\Voteprenop =1\right] \leq \PP^\varepsilon_{x_2}\left[\Voteprenop =1\right]. \end{align}  Let $\T$ be a time-labelled ternary tree and define $$\widehat{\PP}_x^\varepsilon(\T) := \PP_x^\varepsilon\left[\Voteprenop =1 \conditional \T = \Trpre\right].$$ Then, since $0$ and $1$ are exchangeable in the exponentially marked voting system, \begin{align}\label{tired} \widehat{\PP}_x^t(\T) = 1- \widehat{\PP}_{-x}^t(\T)\end{align} for all $x\in \RR$ and $t\geq 0$. Setting $x=0$ in equation~(\ref{tired}) shows that $\widehat{\PP}_0^t(\T) = \tfrac{1}{2}$ for all $t>0$, and together with monotonicity (\ref{awake}), this gives $$\widehat{\PP}_x^\varepsilon(\T) \geq \tfrac{1}{2} \ \text{ for } \ x>0, \ \ \widehat{\PP}_x^\varepsilon(\T) \leq \tfrac{1}{2} \ \text{ for } \ x<0. $$ To conclude this section, we prove the claimed coupling of voting systems.
\begin{theorem}  \label{teo:ineqmark} Let $\varepsilon>0$ and $t\geq 0$. Then
\begin{enumerate}[(1)]
    \item for all $x\geq0$, $\PP^\varepsilon_x[\V(\X(t))=1] \geq \PP^\varepsilon_x\left[\Voteprenop=1\right],$
    \item for all $x\leq 0$, $\PP^\varepsilon_x[\V(\X(t))=1] \leq \PP^\varepsilon_x\left[\Voteprenop=1\right]$.
\end{enumerate}

\end{theorem} 

\begin{proof} We only prove the first inequality, since the second inequality will follow by the symmetry relations (\ref{symetrytree1}) and (\ref{tired}). Recall the initial conditions for $\V(\boldsymbol{X}(t))$ and $\widehat{\V}(\boldsymbol{B}_{R^\varepsilon}(t))$ are given by $$p_0(x) = \mathbbm{1}_{\{x\geq 0\}} \ \text{  and  } \  \widehat{p}_0(x) = u_+\mathbbm{1}_{\{x\geq 0\}} + u_-\mathbbm{1}_{\{x\leq 0\}}$$ respectively. To ease notation, let $p \equiv p_0$ and $\widehat{p}\equiv \widehat{p}_0$ for the remainder of this proof.\\
\indent First, by coupling branching structures and branching times of both trees, we can assume $\T(\Xtrunc(t)) = \T(\X(t)),$ so it suffices to show that
\begin{align} \label{teotreeversion}
\PP_x^t(\T) \geq \PE_x^t(\T) \ \text{for all } x\geq 0
\end{align} for any time-labelled ternary tree $\T$. Denote the time of the first branching event in $\T(\Xtrunc(t)) = \T(\X(t))$ by $\tau$ (which corresponds to $\tau_\emptyset$ in Definition~\ref{definitionnonmarkov}). 
 Let $\tau^\times\sim \mathit{Exp}(I(\varepsilon)^{-2})$ be the exponential random variable that determines if the ancestral individual in $\T(\boldsymbol{B}_{R^\varepsilon}(t))$ is marked. We proceed by induction on the number of branching events in $\T$. \\
\indent To prove the base case, let $\T_0$ denote the tree with a root and a single leaf. Conditional on $\left\{\T_0 = \Trpre\right\}$, let $B(R^\varepsilon_t)$ be the position of the single individual at time $t$ where $(B_s)_{s\geq0}$ is a standard Brownian motion and $(R^\varepsilon_s)_{s\geq0}$ is an $I(\varepsilon)^2$-truncated $\frac{\alpha}{2}$-stable subordinator. Under the exponentially marked voting procedure, this individual votes $1$ with probability  $\widehat{p}(B(R^\varepsilon_t))$ if she is unmarked (i.e. $\tau^\times\geq \tau$), or she votes $1$ with probability $\tfrac{1}{2}$ if she is marked ($\tau^\times < \tau$). Since the event $\{\Trpre = \T_0\}$ is equivalent to $\{\tau>t\}$, we have, for all $x\geq 0$,
\begin{align}\label{needcoffee3}
\PE_x^t(\T_0) &= \EE^{\varepsilon}_x\left[ \widehat{p}\left(B(R^\varepsilon_t)\right) \mathbbm{1}_{\tau^\times\geq \tau}\conditional \tau>t\right] + \tfrac{1}{2}\PP^{\varepsilon}_x[\tau^\times < \tau\conditional \tau>t]\nonumber \\
&=\EE^{\varepsilon}_x\left[ \widehat{p}\left(B(R^\varepsilon_t)\right)\conditional \tau>t\right] \PP^{\varepsilon}_x[ \tau^\times \geq \tau\conditional \tau>t] + \tfrac{1}{2} \PP^{\varepsilon}_x[\tau^\times < \tau\conditional \tau>t], \end{align} 
where in the second line we have used that, conditional on the event $\{\tau>t\}$, the events $\{B(R^\varepsilon_t)>0\}$ and $\{\tau^\times \geq \tau\}$ are independent. 
We next observe that 
\begin{align}\label{needcoffee2}
  &  \EE^{\varepsilon}_x\left[ \widehat{p}\left(B(R^\varepsilon_t)\right)\conditional \tau>t\right] \PP^{\varepsilon}_x[ \tau^\times \geq \tau\conditional \tau>t] + \tfrac{1}{2} \PP^{\varepsilon}_x[\tau^\times < \tau\conditional \tau>t] \nonumber \\
& \ \ \ \ \ \leq \EE^\varepsilon_x\left[ \widehat{p}\left(B(R^\varepsilon_t)\right)\conditional \tau>t\right]\PP_x^\varepsilon[ \tau^\times \geq t] + \tfrac{1}{2}\PP_x^\varepsilon[\tau^\times < t].
\end{align} To see this, note that $\PP_x^\varepsilon\left[ \tau^\times \geq \tau\conditional \tau>t\right]  \leq \PP_x^\varepsilon\left[ \tau^\times \geq t\right]$ and, since $x\geq 0$, by similar arguments as those used to obtain (\ref{tannn}), we have $\EE^\varepsilon_x\left[ \widehat{p}\left(B(R^\varepsilon_t)\right)\conditional \tau>t\right]\geq \tfrac{1}{2}$. Therefore (\ref{needcoffee2}) holds if \begin{align*}
    \tfrac{1}{2}\left(\PP^{\varepsilon}_x[\tau^\times < \tau\conditional \tau>t] - \PP_x^\varepsilon[\tau^\times < t]\right) \leq \tfrac{1}{2}\left(\PP_x^\varepsilon\left[ \tau^\times \geq t\right]-\PP_x^\varepsilon\left[ \tau^\times \geq \tau\conditional \tau>t\right] \right),
\end{align*} or, equivalently, $$\PP^{\varepsilon}_x[\tau^\times < \tau\conditional \tau>t] +\PP_x^\varepsilon\left[ \tau^\times \geq \tau\conditional \tau>t\right] \leq \PP_x^\varepsilon\left[ \tau^\times \geq t\right]+\PP_x^\varepsilon[\tau^\times < t],$$ which holds trivially. Therefore (\ref{needcoffee2}) holds, and combining (\ref{needcoffee3}) with (\ref{needcoffee2}) we obtain \begin{align}
    \label{needcoffee}
    \PE_x^t(\T_0) \leq \EE^\varepsilon_x\left[ \widehat{p}\left(B(R^\varepsilon_t)\right)\conditional \tau>t\right]\PP_x^\varepsilon[ \tau^\times \geq t] + \tfrac{1}{2}\PP_x^\varepsilon[\tau^\times < t].
\end{align}
\indent Continuing the proof of the base case, consider the leaf in $\X(t)$. Conditional on $\left\{\T(\X(t))=\T_0\right\}$, abuse notation and denote the position of the single individual in $\T(\X(t))$ at time $t$ by $B(R_t)$ for $(B_s)_{s\geq0}$ a standard Brownian motion and $(R_s)_{s\geq0}$ an $\tfrac{\alpha}{2}$-stable subordinator. Define $$\overline{\tau} := \inf\left\{ t\geq 0: |R_t-R_{t-}| >\tfrac{2-\alpha}{\alpha}I(\varepsilon)^2 \right\}$$ which describes the first time that $R_t$ makes a jump of size greater than $\frac{2-\alpha}{\alpha} I(\varepsilon)^2$. Of course, $\overline{\tau} \stackrel{D}{=} \tau^\times$, but $\overline{\tau}$ is defined in terms of the subordinator of the ancestral particle in $\X(t)$. 
 Noting that $\overline{\tau}$ is independent of $\tau$, we have
\begin{align}\label{sosleepy}
\PP_x^t(\T_0) &=\EE^{\varepsilon}_x[p\left(B(R_t)\right) \conditional \overline{\tau}\geq t, \tau>t] \PP_x^\varepsilon[\overline{\tau}\geq t]  +\EE^{\varepsilon}_x[p(B(R_t)) \conditional  \overline{\tau}<t<\tau] \PP_x^\varepsilon[\overline{\tau}<t]\nonumber \\ 
&\geq \EE^{\varepsilon}_x[p(B(R^\varepsilon_t))\conditional \tau>t]\,\PP_x^\varepsilon[\overline{\tau}\geq t] + \tfrac{1}{2}\PP_x^\varepsilon[\overline{\tau}<t] \end{align}
using that, conditional on $\{\overline{\tau} > t\}, B(R_t) \stackrel{D}{=} B(R^\varepsilon_t)$, and $\EE^\varepsilon_x[p(B(R_t)) \conditional \overline{\tau}<t<\tau] \geq \tfrac{1}{2}$ since $x\geq 0$, which can be shown using a similar identity to (\ref{symetrytree1}). By definition of $p$ and $\widehat{p}$, $\EE^\varepsilon_x[p(B(R_t))\conditional \tau>t] \geq \EE^\varepsilon_x[\widehat{p}(B(R_t))\conditional \tau>t]$ for all $x\geq 0$, which, together with (\ref{needcoffee}) and (\ref{sosleepy}) gives us
$$\PP_x^t(\T_0) \geq \widehat{\PP}_x^t(\T_0) $$
for $x\geq 0$, proving the base case.  \\
\indent Now suppose that, for all trees with at most $n-1>1$ branching events, (\ref{teotreeversion}) holds. Let $\T^n$ be a tree with $n$ branching events. Define the first three trees of descent, denoted $\T_1, \T_2,$ and $\T_3$, to be the three subtrees of $\T^n$ generated at time $\tau$. Note that $\T_1, \T_2,$ and $\T_3$ have strictly less than $n$ branching events. Write 
\begin{align*}g\left(\PP^{t-\tau}_{X_\tau}(\T\star)\right):=g\left(\PP_{X_\tau}^{t-\tau}(\T_1),\PP_{X_\tau}^{t-\tau}(\T_2),\PP_{X_\tau}^{t-\tau}(\T_3)\right)\end{align*} and define $g\left(\widehat{\PP}^{t-\tau}_{X_\tau}(\T\star)\right)$ similarly. By  (\ref{symetrytree1}) \begin{align}\label{symmmmmm} g\left(\PP^{t-\tau}_{X_\tau}(\T\star)\right) = 1-g\left(\PP^{t-\tau}_{-X_\tau}(\T\star)\right) \ \text{and} \ g\left(\widehat{\PP}^{t-\tau}_{X_\tau}(\T\star)\right) = 1-g\left(\widehat{\PP}^{t-\tau}_{-X_\tau}(\T\star)\right).\ \end{align}
Let $T_n:=\left\{\Trpre = \TrX = \T^n\right\}$. By almost identical arguments to those used to obtain (\ref{sosleepy}), but now conditioning on the event $\{\overline{\tau} >\tau\}$, we have
\begin{align} &\PP_x^t(\T^n) \nonumber\\ &= \EE_x^\varepsilon\left[g\left(\PP^{t-\tau}_{B(R^\varepsilon_\tau)}(\T\star)\right) \mathbbm{1}_{\overline{\tau} > \tau}\conditional T_n\right] +\EE_x^\varepsilon\left[g\left(\PP^{t-\tau}_{B(R_\tau)}(\T\star)\right) \conditional \overline{\tau} \leq \tau,  T_n\right] \PP_x^\varepsilon[\overline{\tau} \leq \tau\conditional T_n]\nonumber \\
&\geq \EE^\varepsilon_x\left[g\left(\PP^{t-\tau}_{B(R^\varepsilon_\tau)}(\T\star)\right) \mathbbm{1}_{\overline{\tau} > \tau}\conditional T_n\right] + \tfrac{1}{2} \PP_x^\varepsilon[\overline{\tau} \leq \tau\conditional T_n],\label{protein} 
\end{align} using that, for all $x\geq 0$, $ \EE_x^\varepsilon\left[g\left(\PP^{t-\tau}_{B(R_\tau)}(\T\star)\right) \conditional \overline{\tau} \leq \tau, T_n\right] \geq \tfrac{1}{2}$ by a similar symmetry relation to (\ref{tired}). By definition of the exponentially marked voting system, for $\tau^\times$ as above, we also have
\begin{align} &\PE_x^t(\T^n) = \EE_x^\varepsilon\left[g\left(\widehat{\PP}^{t-\tau}_{B(R^\varepsilon_\tau)}(\T\star)\right) \mathbbm{1}_{\tau^\times > \tau}\conditional T_n\right] + \tfrac{1}{2}\PP_x^\varepsilon[\tau^\times \leq \tau\conditional T_n]. \label{protein2} \end{align} Therefore, since $\overline{\tau} \stackrel{D}{=} \tau^\times$, by (\ref{protein}) and (\ref{protein2}), it suffices to show that \begin{align}\label{endresult}
    \EE_x^\varepsilon\left[g\left(\PP^{t-\tau}_{B(R_\tau^\varepsilon)}(\T\star)\right)\mathbbm{1}_{\overline{\tau} > \tau}\conditional T_n\right] \geq \EE_x^\varepsilon\left[g\left(\widehat{\PP}^{t-\tau}_{B(R_\tau^\varepsilon)}(\T\star)\right)  \mathbbm{1}_{\overline{\tau}> \tau}\conditional T_n \right].
\end{align}Now, for all $x\geq0$,\allowdisplaybreaks
\begin{align*}
     &\EE^\varepsilon_x\left[g\left(\PP^{t-\tau}_{B(R_\tau^\varepsilon)}(\T\star)\right)\mathbbm{1}_{\overline{\tau} > \tau}\conditional T_n\right]\\
     &=  \EE^\varepsilon_x\left[g\left(\PP^{t-\tau}_{B(R_\tau^\varepsilon)}(\T\star)\right)\mathbbm{1}_{\overline{\tau} > \tau}\conditional B(R_\tau^\varepsilon) >0, T_n\right]\PP_x^\varepsilon\left[B(R_\tau^\varepsilon) >0\conditional  T_n\right]\nonumber  \\ & \ \ \ +  \EE_x^\varepsilon\left[g\left(\PP^{t-\tau}_{B(R_\tau^\varepsilon)}(\T\star)\right)\mathbbm{1}_{\overline{\tau} > \tau} \conditional B(R_\tau^\varepsilon)\leq 0,  T_n\right]\PP^\varepsilon_x[B(R_\tau^\varepsilon) \leq0\conditional T_n]\\
    &=  \EE_x^\varepsilon\left[g\left(\PP^{t-\tau}_{B(R_\tau^\varepsilon)}(\T\star)\right)\mathbbm{1}_{\overline{\tau} > \tau}\conditional B(R_\tau^\varepsilon) >0,  T_n\right]\PP^\varepsilon_x[B(R_\tau^\varepsilon) >0 \conditional  T_n] \\ & \ \ \ +  \PP_x^\varepsilon[\overline{\tau}>\tau \conditional  T_n]\PP^\varepsilon_x[B(R_\tau^\varepsilon) \leq0\conditional  T_n] \\ & \ \ \ - \EE_x^\varepsilon\left[g\left(\PP^{t-\tau}_{-B(R_\tau^\varepsilon)}(\T\star)\right)\mathbbm{1}_{\overline{\tau}> \tau}\conditional B(R_\tau^\varepsilon) \leq0,  T_n\right]\PP^\varepsilon_x[B(R_\tau^\varepsilon) \leq0 \conditional  T_n]\\
     &=  \EE_x^\varepsilon\left[g\left(\PP^{t-\tau}_{B(R_\tau^\varepsilon)}(\T\star)\right)\mathbbm{1}_{\overline{\tau} > \tau}\conditional B(R_S^\varepsilon) >0,  T_n\right] \left(\PP^\varepsilon_x[B(R_\tau^\varepsilon) >0\conditional  T_n]\vphantom{[g\left(\PP^{t-\tau}_{B(R_\tau^\varepsilon)}(\T\star)\right)} \right. \\ & \ \ \left. \vphantom{[g\left(\PP^{t-\tau}_{B(R_\tau^\varepsilon)}(\T\star)\right)} - \PP^\varepsilon_x[B(R_\tau^\varepsilon) \leq 0\conditional  T_n]\right) +\PP^\varepsilon_x[\overline{\tau}>\tau \conditional  T_n]\PP^\varepsilon_x[B(R_\tau^\varepsilon) \leq0\conditional  T_n]\\
     &\geq   \EE_x^\varepsilon\left[g\left(\widehat{\PP}^{t-\tau}_{B(R_\tau^\varepsilon)}(\T\star)\right)\mathbbm{1}_{\overline{\tau} > \tau}\conditional B(R_\tau^\varepsilon) >0,  T_n\right] \left(\vphantom{[g\left(\PP^{t-\tau}_{B(R_\tau^\varepsilon)}(\T\star)\right)} \PP^\varepsilon_x[B(R_\tau^\varepsilon) >0\conditional  T_n] \right. \\ & \ \  \left.- \PP^\varepsilon_x[B(R_\tau^\varepsilon) \leq 0\conditional  T_n]\vphantom{[g\left(\PP^{t-\tau}_{B(R_\tau^\varepsilon)}(\T\star)\right)} \right)  +\PP^\varepsilon_x[\overline{\tau}>\tau \conditional  T_n]\PP^\varepsilon_x[B(R_\tau^\varepsilon) \leq0 \conditional  T_n]\\
     &=  \EE^\varepsilon_x\left[g\left(\widehat{\PP}^{t-\tau}_{B(R_\tau^\varepsilon)}(\T\star)\right)\mathbbm{1}_{\overline{\tau} > \tau}\conditional  T_n\right]
\end{align*}
where, in the second equality, we have applied the symmetry (\ref{symmmmmm}), and in the second to last line, we used monotonicity of $g$ together with our inductive hypothesis, and that, given $x\geq 0$, the difference $\PP^\varepsilon_x[B(R_\tau^\varepsilon) >0 \conditional T_n] - \PP^\varepsilon_x[B(R_\tau^\varepsilon) \leq 0\conditional T_n]$ is non-negative. The final equality follows by reversing the arguments used above but for $\widehat{\PP}^{t-\tau}_{B(R_\tau^\varepsilon)}(\T\star).$ We conclude that (\ref{endresult}) holds, proving our inductive step.
\end{proof}

\subsubsection{Marked majority voting}\label{sectionnmarked}

\indent \indent In this section, we describe what will be the final voting system in one dimension, denoted $\V^\times$, defined on $\T(\Xtrunc(t))$. This voting system will be carried throughout the one-dimensional proof. In spirit, $\V^\times$ is very similar to $\widehat{\mathbb{V}}$, but no longer relies on knowing the lifetime of particles in order to mark them. Instead, particles are marked (independently) when they are born. In Theorem~\ref{teo:ineqmark2}, it will be shown that our new voting system $\V^\times$ can be coupled to the exponentially marked voting system $\widehat{\V}$, so by Theorem~\ref{teo:ineqmark}, it can also be coupled to the original majority voting system $\V$. \\
\indent Under the marked majority voting system, particles will be marked (independently) with probability $b_\varepsilon$, defined as follows. Recall that, for $i\in M(t)$, $\tau^{\times}_i \sim \mathit{Exp}(I(\varepsilon)^{-2})$ is the first time the subordinator $(R_i(s))_{s\geq 0}$ makes a jump of size larger than $\frac{2-\alpha}{\alpha}I(\varepsilon)^2$. Further recall that $\tau_i~\sim~\mathit{Exp}(\varepsilon^{-2})$, where $\tau_i\wedge t$ is the lifetime of the particle $X_i \stackrel{D}{=}B_i(R_i)$ in $\X(t)$. Define 
\begin{align}\label{bdeltapage} b_\varepsilon := \PP[\tau^{\times}_i < \tau_i] =\frac{I(\varepsilon)^{-2}}{ I(\varepsilon)^{-2} + \varepsilon^{-2}}  \sim \frac{\varepsilon^2}{I(\varepsilon)^2} \end{align}
where $x \sim y$ for some $x, y$ depending on $\varepsilon$ means that there exists constants $c, d>0$ independent of $\varepsilon$ such that $cy < x < dy$. The quantity $b_\varepsilon$ is the probability that the subordinator associated to individual $i$ makes a large jump in its lifetime (if individual $i$ is a leaf, $b_\varepsilon$ gives an upper bound on this probability). By Assumption~\ref{assumptions2}~\ref{assumptions2_B}, $b_\varepsilon \to 0$ as $\varepsilon \to 0$. 
\begin{definition}[$\mathbb{V}^\times_p$]\label{vtimesdef} Let $\varepsilon>0$. For $p:\RR\to [0,1]$, we define a \textit{marked majority} voting procedure on $\Trpre$ as follows. 
\begin{enumerate}[(1)]
\item At each branch point in $\Trpre$, the parent particle $j$ marks each of her three offspring $(j,1), (j,2)$ and $(j,3)$ independently with probability $b_\varepsilon$. Each marked particle (independently) votes $1$ with probability $\tfrac{1}{2}$ and otherwise votes $0$.
    \item Each unmarked leaf $i$ of $\Trpre$, independently, votes $1$ with probability\\ $p(B_i(R^\varepsilon_i)))$ and otherwise votes $0$.
    \item At each branch point in $\Trpre$, if the parent particle $k$ is unmarked, she votes according to the majority vote of her three offspring $(k,1), (k,2)$ and $(k,3)$. 
\end{enumerate}
Observe that the initial ancestor of $\T(\Xtrunc(t))$ is never marked under this procedure.
With the marked majority voting procedure described above, define $\mathbb{V}^\times_p$ to be the vote associated to the root $\emptyset$ of $\Trpre$.
\end{definition} The marked majority voting procedure makes it {more difficult} for the root of $\Trpre$ (rooted at $x>0$) to vote $1$ compared to majority voting. Indeed, even if all three offspring vote $1$ at a branch point, under marked majority voting the parent particle can still vote $0$ with positive probability. This can be viewed as the penalty one must pay for truncating the underlying spatial motion (where this truncation really takes place on the subordinator). In other words, to couple $\V^\times(\boldsymbol{B}_{R^\varepsilon}(t))$ to $\V(\X(t)),$ the voting procedure $\V^\times$ should make it more difficult for individuals to vote $1$, to compensate for the new underlying spatial motion, which makes it easier for individuals to vote $1$ (since, for a tree rooted at $x>0$, $\boldsymbol{B}_{R^\varepsilon}(t)$ is more likely to remain on the right-half line than $\X(t)$).\\
\indent Here, we will use the same initial condition as $\widehat{p}_0$ (\ref{phat}) that we did for exponentially marked voting, and write $\mathbb{V}^\times := \V^\times_{\widehat{p}_0}.$ With this choice of initial condition, the marked majority voting system, $\mathbb{V}^\times$, retains many of the symmetry relations exploited in \cite{etheridge2017branching}, that we have already used here for $\V$ and $\widehat{\V}$. Namely, for all $x_1, x_2\in \RR$ with $x_1\leq x_2$, \begin{align}\label{monotonicity}
    \PP^\varepsilon_{x_1}\left[\Votemarknop=1\right] \leq \PP_{x_2}^\varepsilon \left[\Votemarknop=1\right],\end{align} 
and, for any time-labelled tree $\T$, if we set $$\PPt_{x}(\T) = \PP_x^\varepsilon\left[\Votemarknop=1 \conditional \T(\BB_{R^\varepsilon}(t))=\T\right], $$ then by symmetry of the historical stable process and exchangeability of $0$ and $1$ in the marked voting procedure,
\begin{align}\label{symmetry} \PPt_x(\T)  = 1-\PPt_{-x}(\T) \end{align} for all $\T$, $x\in \RR,$ and $t\geq 0$. Setting $x=0$ in (\ref{symmetry}) gives $\PPt_0(\T)=\frac{1}{2}$, so by monotonicity (\ref{monotonicity}), for any time-labelled ternary tree $\T$,
\begin{align}\label{coco}\PPt_{x}(\T) \geq \tfrac{1}{2} \ \text{ for } x>0, \text{ and } \, \PPt_x(\T) \leq \tfrac{1}{2} \ \text{ for } x<0. \end{align}  
We next introduce notation for our marked majority voting procedure. Recall that $g$ from (\ref{majorityvotingfunction}) is the majority voting function associated to $\mathbb{V}$. Define the \textit{marked majority voting function} $\gee: [0,1]^3 \to [0,1]$ by 
\[\gee(p_1, p_2, p_3):= g\left((1-b_\varepsilon)p_1 + \tfrac{b_\varepsilon}{2},(1-b_\varepsilon)p_2 + \tfrac{b_\varepsilon}{2},(1-b_\varepsilon)p_3 + \tfrac{b_\varepsilon}{2} \right).\]
This is the probability that an unmarked parent particle votes $1$ under $\mathbb{V}^\times$, in the special case when the three offspring are independent and have probabilities $p_1, p_2$ and $p_3$ of voting $1$ if they are unmarked. We abuse notation and write $\gee(q) := \gee(q, q, q)$. It is easy to check using symmetry of the majority voting function (\ref{symmetryofg}) that, for all $q\in [0,1]$, \begin{align}\label{symmetryofmarkedg} g_\times(q) = 1-g_\times(1-q).\end{align}
\indent Let $\{u_-, \tfrac{1}{2}, u_+\}$ be the three solutions to $\gee(q) = q$, satisfying $0<u_-<\tfrac{1}{2}<u_+<1$. The exact derivation of these fixed points can be found in the Proposition~\ref{appendixfixedpointsofg}. It will be useful later to note that these fixed points can be approximated by Taylor expansion as
\begin{align}
u_- &= \frac{1}{2} - \frac{\sqrt{(1 - b_\varepsilon)^3 (1 - 3 b_\varepsilon)}}{2 (1 - b_\varepsilon)^3} = \frac{3}{4}b_\varepsilon^2 +\mathcal{O}(b_\varepsilon^3) \label{asympfixedpoint1}\\
u_+ &= \frac{1}{2} + \frac{\sqrt{(1 - b_\varepsilon)^3 (1 - 3 b_\varepsilon)}}{2 (1 - b_\varepsilon)^3} = 1- \frac{3}{4}b_\varepsilon^2 +\mathcal{O}(b_\varepsilon^3). \label{asympfixedpoint2}
\end{align} Henceforth, these fixed points $u_+$ and $u_-$ will be used to define the initial condition $$\widehat{p}_0(x) = u_+ \mathbbm{1}_{\{x\geq 0\}} + u_- \mathbbm{1}_{\{x\geq 0\}}. $$
We now return to the coupling of $\V^\times$ and the original voting system $\V$ via the intermediate system $\widehat{\mathbb{V}}$. 
\begin{theorem}  \label{teo:ineqmark2} Let $\widehat{\mathbb{V}}$ be the exponentially marked majority voting system (Definition~\ref{definitionnonmarkov}) and $\mathbb{V}^\times$ be the marked majority voting system (Definition~\ref{vtimesdef}), both with initial condition $\widehat{p}_0(x) = u_+\mathbbm{1}_{\{x\geq 0\}} + u_-\mathbbm{1}_{\{x\leq 0\}}$. Then, for all $x\in \RR$ and $t\geq 0$,
\begin{align*}
\PP^\varepsilon_x\left[\Voteprenop=1\right] = \PP^\varepsilon_x\left[\V^\times(\Xtrunc(t))=1\right](1-b_\varepsilon) +\frac{b_\varepsilon}{2}.
\end{align*}
\end{theorem}
\begin{proof}
Recall that, under the voting system $\widehat{\mathbb{V}}$, each particle $i$ is marked if $\tau^{\times}_i<\tau_i$. By definition of $b_\varepsilon$, all particles in $\T(\Xtrunc(t))$ are marked with probability $b_\varepsilon$ under both $\widehat{\V}$ and $\V^\times$, except for ancestral particle, which remains unmarked under $\V^\times$ by definition. Conditioning on the marking of the ancestral particle in $\widehat{\mathbb{V}}$, we obtain 
\begin{align*}
 \PP^\varepsilon_x\left[\widehat{\V}(\Xtrunc(t))=1\right] &= \PP^\varepsilon_x\left[\widehat{\V}(\Xtrunc(t))=1 \conditional \tau^\times > \tau_0\right]\PP^\varepsilon_x[\tau^\times > \tau_0] \\ & \ \ + \PP^\varepsilon_x\left[\widehat{\V}(\Xtrunc(t))=1 \conditional \tau^\times \leq \tau_0\right]\PP^\varepsilon_x[\tau^\times \leq \tau_0] \\ 
&= \PP^\varepsilon_x\left[\mathbb{V}^{\times}(\BB_{R^\varepsilon}(t))=1\right](1-b_\varepsilon) +\frac{b_\varepsilon}{2},
\end{align*} where the last line follows by definition of $b_\varepsilon$.
\end{proof}
\noindent Finally, by Theorem~\ref{teo:ineqmark} and Theorem~\ref{teo:ineqmark2}, to prove our main one-dimensional result, Theorem~\ref{maintheorem1D}, it suffices to show the following. 
\begin{theorem} \label{mainteo1dmarked}
Suppose $I$ satisfies Assumptions~\ref{assumptions2}~\ref{assumptions2_A}-(B). Fix $k\in \NN$ and $T^*\in (0,\infty)$. Let $u_+, u_-$ be as in (\ref{asympfixedpoint1}) and (\ref{asympfixedpoint2}). Then there exist $c_1(\alpha, k), \varepsilon_1(\alpha, k)>0$ such that, for all $t\in [0,T^*]$ and all $\varepsilon \in (0, \varepsilon_1(k)),$
\begin{enumerate}[(1)]
    \item for $x \geq c_1(k)I(\varepsilon)|\log\varepsilon|$, we have $\PP^\varepsilon_x\left[\Votemarknop=1\right] \geq u_+ - \varepsilon^k,$
    \item for $x \leq -c_1(k)I(\varepsilon)|\log\varepsilon|$,  we have $\PP^\varepsilon_x\left[\Votemarknop=1\right] \leq u_- +\varepsilon^k$.
\end{enumerate}
\end{theorem} \noindent Observe that the sharpness of the interface in Theorem~\ref{mainteo1dmarked} is of the same order as the sharpness from the Brownian result, Theorem~\ref{browniantheorem}. This is a result of the truncated subordinated Brownian motion behaving similarly to a Brownian motion (as discussed in Section~\ref{choice_scaling_sec}).
\begin{remark}\label{cool}
By a similar proof to that of Theorem~\ref{votedual}, we see that $$u^\varepsilon(t, x) = \PP^\varepsilon_x\left[\Votemarknop=1\right]$$ is a solution to the equation \begin{align} \nonumber \partial_t u^\varepsilon &= \mathcal{L}^\varepsilon u^\varepsilon + {\varepsilon^{-2}} (g_\times(u^\varepsilon) - u^\varepsilon)\\
&= \mathcal{L}^\varepsilon u^\varepsilon + \mathcal{O}\left({\varepsilon^2}{I(\varepsilon)^{-4}}\right)\left(\tfrac{1}{2}-u^\varepsilon\right)^3 + \mathcal{O}\left({\varepsilon^{-2}}\right)u^\varepsilon(2u^\varepsilon-1)(1-u^\varepsilon)\label{icecream}
\end{align} with initial condition $u^\varepsilon(0, x) = \widehat{p}_0(x)$, where $\mathcal{L}^\varepsilon$ denotes the infinitesimal generator of $(B(R^\varepsilon_t))_{t\geq 0}$. 
Rather remarkably, the work from this section tells us that solutions to (\ref{icecream}) and (\ref{mainequation22}) are related. More precisely, the couplings from Theorem~\ref{teo:ineqmark} and Theorem~\ref{teo:ineqmark2} tell us that solutions to equation~(\ref{icecream}) (after transformation by the function $v\mapsto (1-b_\varepsilon)v+\tfrac{b_\varepsilon}{2}$) are lower and upper bounds to solutions of the scaled fractional Allen--Cahn equation (\ref{mainequation22}) restricted to $x\geq 0$ and $x\leq 0$, respectively. It would be interesting to see if this relationship provides any insights into a PDE-theoretic proof of our main result. 
\end{remark}

\end{section}

\subsection{Proof of Theorem \ref{mainteo1dmarked}}\label{saltchips}
\indent We now prove Theorem~\ref{mainteo1dmarked}. To do so, we will adapt ideas from both \cite{etheridge2017branching, durrett2020motion}. In \cite{etheridge2017branching}, the majority voting function $g$ was used throughout, while \cite{durrett2020motion} builds upon this work and considers more general voting functions. This makes \cite{durrett2020motion} useful when proving results about the marked majority voting function $g_\times$. \\
\indent Throughout this section, we take the initial condition $$\widehat{p}_0(x)=u_+ \mathbbm{1}_{\{x\geq0\}}+u_-\mathbbm{1}_{\{x\leq0\}}.$$ Our next result verifies that the marked majority voting procedure cannot reduce the positive voting bias on the leaves when the root, $x$, is non-negative. Here, we say a leaf has a `positive voting bias' if it has a preference for voting one instead of zero, which is the case when the tree is rooted at a non-negative point $x$. Once we have shown Lemma~\ref{ineq:nobranch}, we will follow the strategy of \cite{etheridge2017branching} to show that, after enough time has passed, with high probability enough rounds of voting have occurred to ensure that the positive voting bias at a leaf is amplified to a large voting bias at the root. \\ \indent By the symmetry (\ref{symmetry}), a similar result to Lemma~\ref{ineq:nobranch} will hold for the negative voting bias on the leaves when $x<0$. In view of this symmetry, we often state results only for positive $x$ when convenient to do so. Recall that $$\PPt_{x}(\T) := \PP_x^\varepsilon\left[\Votemarknop=1 \conditional \T(\BB_{R^\varepsilon}(t))=\T\right].$$
\begin{lemma} \label{ineq:nobranch}
For any time-labelled ternary tree $\T$, time $t>0$, and any $x\geq 0$, \textup{
\begin{align*}
\PPt_x(\T) &\geq u_+ \PP_x[ B(R^\varepsilon_t) \geq 0] + u_-\PP_x[B(R_t^\varepsilon) \leq 0].
\end{align*}}
\end{lemma}
\begin{proof}
This proof follows exactly \cite[Lemma 3.1]{durrett2020motion}, by an inductive argument on the number of branching events in $\T(\BB_{R^\varepsilon}(t))$ together with symmetry of the voting function $g_\times$ and symmetry of the transition density for $(B(R^\varepsilon_t))_{t\geq 0}.$
\end{proof}

Lemma~\ref{ineq:nobranch} partly motivated our definitions of $\widehat{\V}$ and $\V^\times$. Recall that marked individuals in $\V^\times$ and $\widehat{\V}$ vote $1$ or $0$ with equal probability. However, the proof of Theorem~\ref{teo:ineqmark} would have simplified greatly if we had asked marked individuals under $\widehat{\V}$ to vote $0$ with probability $1$. Technically, this version of $\widehat{\V}$ would only satisfy part~(1) of Theorem~\ref{teo:ineqmark}. To obtain part~(2) of Theorem~\ref{teo:ineqmark}, one would need to define $\widehat{\V}$ so that marked individuals vote $1$ with probability $1$. In fact, we will explore these voting systems more in Section~\ref{ch:3}. Unlike $g_\times$, which satisfies (\ref{symmetryofmarkedg}), the voting function corresponding to this other system would not be symmetric. As a result, the proof of \cite[Lemma 3.1]{durrett2020motion} would no longer apply, and we are unsure if Lemma~\ref{ineq:nobranch} would hold at all. The symmetry of $g_\times$ will be used in several other proofs throughout this section as well.\\ 
\indent We now show that the iterative voting procedure amplifies a small positive bias at the leaves to a large positive voting bias at the root. To do this, we define $\geen(q)$ inductively by $$\geeone(q) = \gee(q), \ \geenplusone(q) = \geen(\gee(q)). $$ Noting that the ancestral particle is never marked under $\mathbb{V}^\times$, we see that $\geen(q)$ is the probability of voting $1$ at the root of an $n$-level regular ternary tree under $\mathbb{V}^\times$ if the votes of the \textit{unmarked} leaves are i.i.d.~Bernoulli($q$). In the following result, we consider the rate of convergence of $g_\times$ to its fixed points. Let $I$ be a scaling function satisfying Assumptions~\ref{assumptions2}~\ref{assumptions2_A}-\ref{assumptions2_B} throughout. 

\begin{lemma} \label{lemma:iterative_voting}
Fix $k\in \NN$. There exists $A(k)<\infty$ and $\varepsilon_1(k)>0$ such that, for all $\varepsilon \in (0, \varepsilon_1)$ and $n \geq A(k)|\log \varepsilon|$,
\[  \geen\left( \tfrac{1}{2} + \varepsilon \right)  \geq u_+ - \varepsilon^{k} \ \text{  and  } \ \geen\left( \tfrac{1}{2} - \varepsilon \right)  \leq u_- + \varepsilon^{k}. \]
\end{lemma}
\begin{proof} We follow the proof of \cite[Lemma 3.2]{durrett2020motion}, with some important changes to reflect that our voting function, $g_\times$, depends on the parameter $\varepsilon$. Recall that $$g_\times(p) = g((1-b_\varepsilon)p+\tfrac{b_\varepsilon}{2})$$ where $b_\varepsilon = \mathcal{O}\left(\tfrac{\varepsilon^2}{I(\varepsilon)^2}\right)$. We prove only the first inequality since the second follows by completely symmetric arguments. First, we show that there exists $C_k>0$ and some fixed $q_0>0$ such that, after $n\geq C_k |\log\varepsilon|$ iterations, \begin{align}\label{geqn1}
    g^{(n)}_\times(u_+-q) \geq u_+-\varepsilon^k
\end{align} for all $q\leq q_0$. 
We then show that there exists $D>0$ such that, after $n \geq D|\log\varepsilon|$ iterations, \begin{align}\label{geqn2}
    g_\times^{(n)}\left(\tfrac{1}{2}+\varepsilon\right)\geq u_+-q_0.
\end{align} Combining (\ref{geqn1}) and (\ref{geqn2}) then gives the result. To prove (\ref{geqn1}), choose $\varepsilon_1$ sufficiently small so that $b_\varepsilon \leq \frac{1}{6}$ for all $\varepsilon \in (0,\varepsilon_1)$. Then 
\begin{align}\label{half}
g_\times'\left( \tfrac{1}{2} \right) &= (1-b_\varepsilon) g'\left(\tfrac{1}{2} \right) =\tfrac{3}{2} (1-b_\varepsilon) \geq \tfrac{5}{4} > 1. 
\end{align}
Next, using that $g'(0)=0$ and $g'$ is continuous, together with the estimate (\ref{asympfixedpoint1}), for $\varepsilon_1$  sufficiently small we have \[g'\left(u_-(1-b_\varepsilon)+\tfrac{b_\varepsilon}{2}\right) < \tfrac{1}{4}\] for all $\varepsilon\in (0,\varepsilon_1).$ It follows that, for this choice of $\varepsilon_1$, $$g_\times'(u_-) =(1-b_\varepsilon) g'\left(u_-(1-b_\varepsilon)+ \tfrac{b_\varepsilon}{2}\right) < \tfrac{1}{4}.$$ Since $g_\times'\left(\tfrac{1}{2}\right)>1$ and $g_\times'$ is continuous,
\begin{align*}
q_0 := \inf \left\{ q \geq 0 : g_\times'(u_-+q) \geq \tfrac{1}{2} \right\} > 0.
\end{align*} By the Mean Value Theorem and definition of $q_0$, together with the symmetry of the marked voting function (\ref{symmetryofmarkedg}), for all $q < q_0$ 
\[ u_+ - g_\times(u_+ - q) = g_\times(u_- +q) - u_- \leq q\, g_\times'(u_-+q_0) = \tfrac{q}{2}. \]
Iterating this yields
\[ u_+ - g_\times^{(n)}(u_+-q) \leq  \tfrac{1}{2^n}\left( u_+-q \right) \leq \tfrac{1}{2^n}.\]
It follows that there exists $C_k >0 $ such that if $n \geq C_k |\log \varepsilon|$ and $q \leq q_0$ then 
\[ g_\times^{(n)}(u_+-q) \geq u_+ - \varepsilon^k, \]
thereby proving (\ref{geqn1}). We now prove (\ref{geqn2}). By equation~(\ref{half}), $\varepsilon_1$ is sufficiently small so that, for all $\varepsilon \in (0, \varepsilon_1)$, $g_\times'(\frac{1}{2})>1$. Since $g_\times$ is increasing and $u_-, \tfrac{1}{2}, u_+$ are the only fixed points of $g_\times$, we have
\begin{align*}
g_\times(q) > q \text{ for all } q \in \left(\tfrac{1}{2}, u_+ \right).
\end{align*}
By definition of $q_0$ and since $g_\times'$ is increasing on $(0, \tfrac{1}{2})$, $u_+-q_0-\tfrac{1}{2} = \tfrac{1}{2}-q_0-u_->0$. Therefore 
\begin{align*}
 q_1 := \inf_{\varepsilon\leq q \leq u_+-q_0-\frac{1}{2}} \frac{g_\times\left(\tfrac{1}{2}+q\right)-\left(\frac{1}{2}+q\right)}{q} \geq 0,  
\end{align*}
and so for $q \in \left[\varepsilon,u_+-q_0-\frac{1}{2}\right]$, by definition of $q_1$ we have \begin{align}\label{rowrowrow} g_\times\left(\tfrac{1}{2}+q\right)-\tfrac{1}{2} > (1+q_1)q.\end{align}  Now, if $g_\times\left(\frac{1}{2}+\varepsilon\right) \geq u_+ -q_0$, we are done. If not, we can apply (\ref{rowrowrow}) twice to obtain \begin{align*}
g_\times^{(2)}\left(\tfrac{1}{2}+\varepsilon\right)&=g_\times\left(\tfrac{1}{2} + \left(g_\times(\tfrac{1}{2}+\varepsilon) - \tfrac{1}{2}\right)\right) \\ &\geq (1+q_1)\left[g_\times(\tfrac{1}{2}+\varepsilon)-\tfrac{1}{2}\right] + \tfrac{1}{2}\\ &\geq (1+q_1)^2\varepsilon+\tfrac{1}{2}.
\end{align*} 
Repeating this argument $n-1$ times, we obtain 
\[ g_\times^{(n)}\left( \tfrac{1}{2} + \varepsilon \right) \geq (1+q_1)^n \varepsilon + \tfrac{1}{2}. \]
It follows that, for $D := \tfrac{1}{\log(1+q_1)}$, after $n > D |\log \varepsilon|$ iterations,
\[ g_\times^{(n)}\left( \tfrac{1}{2} + \varepsilon \right) \geq u_+ - q_0. \]
Setting $A = C_k + D$ proves the result.
\end{proof} 

\noindent The following useful inequality for $g_\times$ will be used in the proof of Theorem~\ref{mainteo1dmarked}. 
\begin{lemma} \label{easyboundg}
If $p_1,p_2,p_3 \geq \tfrac{1}{2}$ then,
\[ g_\times(p_1,p_2,p_3) \geq \min (p_1,p_2,p_3,u_+). \]
If $p_1,p_2,p_3 \leq \tfrac{1}{2}$ then,
\[ g_\times(p_1,p_2,p_3) \leq \max(p_1,p_2,p_3,u_-).\]
\end{lemma}
\begin{proof} We prove only the first inequality, since the second one follows by a symmetric argument. Denote
\[ p_{\min} = \min\{ p_1,p_2,p_3,u_+\}. \]
If $p_1, p_2, p_3 \geq \tfrac{1}{2},$ it follows that $\tfrac{1}{2} \leq p_{\min} \leq u_+$. Since $g_\times$ is increasing in each variable, 
\[ g_\times(p_1,p_2,p_3) \geq g_\times(p_{\min}). \]
Therefore it suffices to show that $g_\times(p_{\min}) \geq p_{\min}$. For this recall that $\{u_-,\tfrac{1}{2},u_+\}$ are the only fixed points of $g_\times$. Factorising $g_\times(p_{\min})-p_{\min}$ yields
\[ g_\times(p_{\min}) - p_{\min} = (p_{\min}- u_-)(2p_{\min}-1)(u_+-p_{\min}). \]
Since $\tfrac{1}{2} \leq p_{\min} \leq u_+$, $g_\times(p_{\min})-p_{\min} \geq 0$, as required.
\end{proof}
The following lemma states that, with high probability, by time $t\geq a\varepsilon^2 |\log\varepsilon|$, each ancestral line of descent in $\T(\Xtrunc(t))$ contains at least $\mathcal{O}(|\log \varepsilon|)$ branching events. Let $$\T_n^{reg} = \cup_{k\leq n} \{1, 2, 3\}^k \subset \mathcal{U}$$ denote the $n$-level regular ternary tree, and for $l\in \RR$, let $\T^{reg}_l = \T^{reg}_{\lceil l \rceil}$. For $\T$ a time-labelled ternary tree, we use $\T^{reg}_l\subseteq \T$ to mean that, as subtrees of $\mathcal{U}$, $\T^{reg}_l$ is contained inside $\T$ (ignoring time labels).

\begin{lemma}\label{defnofa} 
Let $k\in \NN$ and $A(k)$ be as in Lemma~\ref{lemma:iterative_voting}. There exists $a_1(\alpha, k) >0$ and $\varepsilon_1(\alpha, k)>0$ such that, for all $\varepsilon \in (0, \varepsilon_1)$ and $t \geq a_1(k)\varepsilon^2 |\log \varepsilon|$,
\begin{align*} \PP^\varepsilon \left[\mathcal{T}\left(\BB_{R^\varepsilon}(t)\right) \supseteq \mathcal{T}^{\text{reg}}_{A(k)|\log\varepsilon|}\right]\geq 1 - \varepsilon^{k}.  \end{align*}
\end{lemma}\begin{proof}
This proof proceeds exactly as that of \cite[Lemma 2.10]{etheridge2017branching}, where the authors estimate the probability that a single leaf of $\mathcal{T}^{\text{reg}}_{A(k)|\log\varepsilon|}$ is not in $\mathcal{T}\left(\BB_{R^\varepsilon}(t)\right)$ and combine this with a union bound summing over all leaves.
\end{proof}

\noindent Next, we control the displacement of leaves from the root of $\BB_{R^\varepsilon}(t)$. 

\begin{lemma}\label{displacement lemma} 
Fix $k\in \NN$ and let $a_1(k)$ be as in Lemma~\ref{defnofa}. There exists $\varepsilon_1(\alpha, k)~>~0$ and $l_1(\alpha, k)>0$ such that, for all $\varepsilon \in (0, \varepsilon_1)$ and $s\leq a_1(k) \varepsilon^2|\log\varepsilon|,$ $$\PP^\varepsilon_x\left[\exists i\in N(s) : |B_i(R^\varepsilon_i(s))-x| \geq l_1(k) I(\varepsilon)|\log \varepsilon| \right] \leq \varepsilon^k.$$ 
\end{lemma}

Lemma~\ref{displacement lemma} highlights the importance of working with a truncated subordinated Brownian motion instead of the original stable process. In the proof of Lemma \ref{displacement lemma}, we control the position of the leaves in $\T(\boldsymbol{B}_{R{^\varepsilon}}(t))$ using a many-to-one lemma. If we were to use the same approach for the stable tree (without any truncation), we could not obtain the polynomial error $\varepsilon^k$ in Lemma~\ref{displacement lemma}, which is crucial to our later proofs.

\reqnomode
\begin{proof}[Proof of Lemma~\ref{displacement lemma}]
First note that for any $m>0$ \begin{align*}
\nonumber &  \PP^\varepsilon_x\left[\exists i\in N(s) : |B_i(R_i^\varepsilon(s))-x| \geq m I(\varepsilon)|\log\varepsilon| \right] \\ & \ \ \ \leq \EE^\varepsilon_x[N(s)]\PP_0\left[|B(R_s^\varepsilon)| \geq m I(\varepsilon)|\log\varepsilon|  \right] \nonumber \\ & \ \ \ = e^{\frac{2 s}{\varepsilon^2}}  \PP_0\left[|B(R_s^\varepsilon)| \geq m I(\varepsilon)|\log\varepsilon|  \right] \nonumber \\ & \ \ \ \leq \varepsilon^{-2a_1} \PP_0\left[|B(R_s^\varepsilon)| \geq m I(\varepsilon)|\log\varepsilon|\right].
\end{align*}
Denote the density of $R_s^\varepsilon$ by $f_{\varepsilon}.$ Then for any $h\geq 0$, partitioning over the event $\{R^\varepsilon_s\leq hI(\varepsilon)^2|\log\varepsilon|\}$ and its complement, we obtain\begin{align*}
    &\PP_0\left[|B(R_s^\varepsilon)| \geq m I(\varepsilon)|\log\varepsilon|\right] \\ &\leq \int_{0}^{hI(\varepsilon)^2|\log \varepsilon|}  \PP_0\left[|B_t| \geq m I(\varepsilon)|\log\varepsilon| \, | \ R_s^\varepsilon = t\right]f_\varepsilon(t) dt + \PP_0[R_s^\varepsilon \geq hI(\varepsilon)^2 |\log\varepsilon|]\\
    &\leq \sup_{0\leq t \leq h I(\varepsilon)^2|\log \varepsilon|}\PP_0[|B_t| \geq m I(\varepsilon)|\log\varepsilon|] + \PP_0[ R_s^\varepsilon \geq hI(\varepsilon)^2 |\log\varepsilon|].
\end{align*} We bound each term separately. First, by a Chernoff bound, for all $t\leq h I(\varepsilon)^2|\log\varepsilon|,$\begin{align*}
    \PP_0\left[|B_t| \geq mI(\varepsilon)|\log \varepsilon|\right] &= \PP_0\left[\sqrt{2t}|Z|\geq m I(\varepsilon) |\log\varepsilon|\right]\\
    &\leq \PP_0\left[\sqrt{2h}|Z|\geq m |\log\varepsilon|^{\frac{1}{2}} \right]\\
    &\leq \exp\left(-\tfrac{1}{4}\tfrac{m^2}{h}|\log \varepsilon|\right)\\
    &= \varepsilon^{\tfrac{m^2}{4h}}.
\end{align*} Fix $h := k+2a_1(k)+1$. By Lemma~\ref{boundonsubordinator}, there exists $\varepsilon_1(k)>0$ such that, for all $\varepsilon \in (0,\varepsilon_1)$ \begin{align*}
    \PP_0\left[ R_s^\varepsilon \geq h I(\varepsilon)^2 |\log\varepsilon|\right]
    &\leq \varepsilon^{k+2a_1(k)}.
\end{align*} Putting this together, we obtain \begin{align*}
    \PP^\varepsilon_x\left[\exists i\in N(s) : |B_i(R_i^\varepsilon(s))-x| \geq m I(\varepsilon)|\log\varepsilon| \right] \leq \varepsilon^k + \varepsilon^{\tfrac{m^2}{4h}-2a_1(k)}
\end{align*} so the result holds by choosing $m = l_1(k)$ sufficiently large. 
\end{proof}
\noindent In the following proof of Theorem~\ref{mainteo1dmarked}, we suppose $x\geq 2l_1(k)I(\varepsilon)|\log\varepsilon|$, $s^*:=a_1(k)\varepsilon^2|\log\varepsilon|$ and consider the cases $t\leq s^*$ and $t\geq s^*$ separately. First, if $t\leq s^*$, by Lemma~\ref{displacement lemma}, with high probability none of the particles in $\T(\X(t))$ have moved a distance further than $l_1(k)I(\varepsilon)|\log\varepsilon|,$ and the result follows easily. When $t\geq s^*$, we use Lemma~\ref{ineq:nobranch} to show that the leaves of $\BB_{R^\varepsilon}(t)$ have a positive voting bias, which, by Lemma \ref{defnofa} is magnified by $\mathcal{O}(|\log\varepsilon|)$ rounds of voting, so  Lemma \ref{lemma:iterative_voting} applies and gives the result. 
\begin{proof}[Proof of Theorem \ref{mainteo1dmarked}] Our approach of truncating the stable subordinator now allows us to follow the strategy of proof of \cite[Theorem 2.6]{etheridge2017branching}. We suppress the superscript $\varepsilon$ of $\PP^\varepsilon_x$ throughout. Fix $k\in \NN$ and $T^*\in (0,\infty).$
Let $\varepsilon<\tfrac{1}{2}$ and define $z_\varepsilon$ implicitly by the relation \begin{align}\label{iloverowing} \PP_{z_\varepsilon}[B(R_{T^*}^\varepsilon) \geq 0] = \tfrac{1}{2}+ (u_+-u_-)^{-1}\varepsilon,\end{align} 
and note that $z_\varepsilon \sim \varepsilon \sqrt{4\pi R^\varepsilon_{T^*}}$ as $\varepsilon \to 0$. Moreover, with arbitrarily high probability, $R_{T^*}^\varepsilon$ is bounded above by a constant (depending on $T^*$). 
Comparing the above estimate with the definition of $z_\varepsilon$, there exists a constant $C(T^*)$ such that $z_\varepsilon \leq C(T^*) \varepsilon$ as $\varepsilon \to 0$ (with arbitrarily high probability). Suppose $\varepsilon_1(k) < \tfrac{1}{2}$ is sufficiently small so that Lemmas~\ref{defnofa} and  \ref{displacement lemma}  hold for all $\varepsilon \in (0, \varepsilon_1)$. 
Let $c_1(k)=2l_1(k)$, for $l_1(k)$ as in Lemma~\ref{displacement lemma}. By Assumption~\ref{assumptions2}~\ref{assumptions2_B}, $\varepsilon I(\varepsilon)^{-1} \to 0$ as $\varepsilon \to 0$, so we can choose $\varepsilon_1$ sufficiently small so that
\begin{align*}
l_1(k) I(\varepsilon)|\log\varepsilon| + z_\varepsilon \leq c_1(k) I(\varepsilon)|\log\varepsilon|
\end{align*} for all $\varepsilon\in (0,\varepsilon_1)$ with arbitrarily high probability. Let $a_1$ be as in Lemma~\ref{defnofa} and define $$s^*(\varepsilon)= a_1(k)\varepsilon^2|\log\varepsilon|.$$ \allowdisplaybreaks
Let $t\in (0, s^*)$ and $z\geq c_1(k) I(\varepsilon)|\log\varepsilon|$. Note that $g_\times(u_+) = u_+$, so if the initial condition is constant with $p(x) \equiv u_+$, then \begin{align}\label{con}\PP_z\left[\V^\times_{u_+}(\BB_{R^\varepsilon}(t)) = 1\right] = u_+ \text{ for all } t>0, z\in \RR. \end{align} Recall that the initial condition for marked majority voting is chosen to be $\widehat{p}_0 = u_+ \mathbbm{1}_{\{x\geq 0\}} + u_- \mathbbm{1}_{\{x\leq 0\}}$, so \allowdisplaybreaks  \begin{align*}
&\PP_z\left[\V^\times(\BB_{R^\varepsilon}(t)) = 0\right] \leq \PP_z\left[\exists \, i\in N(t) :\vphantom{\V^\times} |B_i(R_i^\varepsilon(t))-z| \geq c_1 I(\varepsilon)|\log\varepsilon| \right]  \\ 
 & \ \ \ \ \ \ \  + \PP_z\left[\{\V^\times(\BB_{R^\varepsilon}(t)) = 0\} \cap \{\nexists\,  i\in N(t) : |B_i(R_i^\varepsilon(t))-z| \geq c_1I(\varepsilon)|\log\varepsilon|\} \right]  \\
 &\ \ \ \ \ \leq \varepsilon^k + 1-u_+ \\
 &\ \ \ \ \ = \varepsilon^k + u_-
 \end{align*} 
 where we have used (\ref{con}) in the second inequality. This proves the result when $t<s^*.$ 
  Now suppose $t \in [s^*, T^*]$ and $z\geq c_1(k) I(\varepsilon)|\log\varepsilon|.$ Let $\T_{s^*} = \T(\BB_{R^\varepsilon}(s^*))$ be the time-labelled tree of the branching stable process at time $s^*$. Define 
  $$q_{t-s^*}(z) = \PP_z[\V^\times(\BB_{R^\varepsilon}(t-s^*))=1]$$ for all $z\in \mathbb{R}$. Write $\{\BB_{R^\varepsilon}(s^*)>z_\varepsilon\}$ for the event $B_i(R^\varepsilon_i(s^*))>z_\varepsilon$ for all $i\in N(s^*)$. Then
\begin{align}\label{fries}
    \PP_z[\V^\times(\BB_{R^\varepsilon}(t))=1] &=  \PP_z\left[\mathbb{V}^\times_{q_{t-s^*}(\cdot)}(\BB_{R^\varepsilon}(s^*))=1\right] \nonumber \\
    &\geq \PP_z \left[\left\{\mathbb{V}^\times_{q_{t-s^*}(z_\varepsilon)}(\BB_{R^\varepsilon}(s^*))=1\right\} \cap \{\BB_{R^\varepsilon}(s^*)>z_\varepsilon\}\right]\nonumber \\
    &\geq \PP_z\left[\mathbb{V}^\times_{q_{t-s^*}(z_\varepsilon)}(\BB_{R^\varepsilon}(s^*))=1\right] - \varepsilon^k,
\end{align} 
 where the first line follows by the Markov property of $\BB_{R^\varepsilon}$ at time $s^*$, the second line follows by monotonicity (\ref{monotonicity}), and the third line follows by the Lemma~\ref{displacement lemma} and our assumption $z\geq c_1 I(\varepsilon)|\log\varepsilon|$. Now, by Lemma~\ref{ineq:nobranch} and the definition of $z_\varepsilon$ (\ref{iloverowing}) (noting that $t-s^* \leq T^*$), we have
 \begin{align}\label{burgers} q_{t-s^*}(z_\varepsilon) &\geq \PP_{z_\varepsilon}[B(R^\varepsilon_{t-s^*}) \geq 0] u_+ +  \PP_{z_\varepsilon}[B(R^\varepsilon_{t-s^*}) \leq 0] u_- \nonumber\\ &\geq u_+\left( \tfrac{1}{2} + (u_+-u_-)^{-1}\varepsilon  \right) + u_-\left( \tfrac{1}{2} - (u_+-u_-)^{-1}\varepsilon  \right)\nonumber\\ &= \tfrac{1}{2} + \varepsilon.\end{align}
Substituting (\ref{burgers}) into (\ref{fries}), we obtain \begin{align}\label{shake} \PP_z\left[\V^\times(\BB_{R^\varepsilon}(t))=1\right] \geq \PP_z\left[\mathbb{V}^\times_{\frac{1}{2}+\varepsilon}(\BB_{R^\varepsilon}(s^*))=1\right]-\varepsilon^k.\end{align}

 Note that, if $p_i \geq \tfrac{1}{2} + \varepsilon$ for $i=1, 2, 3$, then $ g_\times(p_1, p_2, p_3) \geq \min(p_1, p_2, p_3,u_+)$ from Lemma~\ref{easyboundg}. Therefore, if each leaf of $\T(\boldsymbol{B}_{R^\varepsilon}(s^*))$ votes $1$ independently with probability greater than $\tfrac{1}{2} + \varepsilon$, and $\T(\boldsymbol{B}_{R^\varepsilon}(s^*)) \supseteq \T^{reg}_{A|\log\varepsilon|}$, then each leaf in $\T^{reg}_{A|\log\varepsilon|}$ also votes $1$ with probability greater than $\tfrac{1}{2} + \varepsilon$. By Lemma~\ref{defnofa}, $\T(\boldsymbol{B}_{R^\varepsilon}(s^*))\supseteq \T^{reg}_{A|\log\varepsilon|}$ with probability at least $1-\varepsilon^k$, so by  Lemma~\ref{lemma:iterative_voting} \begin{align*} \PP_z\left[\mathbb{V}^\times_{\frac{1}{2}+\varepsilon}(\BB_{R^\varepsilon}(s^*))=1\right] &\geq (1-\varepsilon^k)g_\times^{(\lceil A|\log\varepsilon|\rceil)}\left(\tfrac{1}{2}+ \varepsilon\right)\\
&\geq (1-\varepsilon^k)(u_+-\varepsilon^k) \\
&\geq u_+-2\varepsilon^k.
\end{align*} 
Substituting this into (\ref{shake}) yields \begin{align} \label{conven}\PP_z\left[\V^\times(\BB_{R^\varepsilon}(t))=1\right] \geq u_+-3\varepsilon^k,\end{align} thereby proving part (1) of Theorem~\ref{mainteo1dmarked}. Part (2) of Theorem~\ref{mainteo1dmarked} follows by completely symmetric arguments.
\end{proof}
\begin{remark}\label{coeffignore} Observe that (\ref{conven}) contains a coefficient in front of the polynomial error term, $\varepsilon^k$, that is not mentioned in the statement of Theorem~\ref{mainteo1dmarked}. Our convention here and in the following sections is that coefficients of polynomial error terms will not be stated in theorems and propositions (this convention was also used in \cite{etheridge2017branching}).\end{remark}
\subsection{Slope of the interface}
Just as in \cite{etheridge2017branching}, to prove the multidimensional result, Theorem~\ref{maintheorem2}, we make use of a lower bound on the `slope' of the interface in dimension $\dd=1$. For a time-labelled ternary tree $\T$, recall that $$\PPt_{x}(\T):= \PP^\varepsilon_x\left[\Votemarknop =1) \conditional \T(\BB_{R^\varepsilon}(t))=\T\right].$$ 
\begin{proposition}
Suppose $x\geq 0$ and $\eta>0$. Then for any time-labelled ternary tree $\T$ and any time $t$, \begin{align}\label{yoga} \textup{$ \PPt_{x}(\T)-\PPt_{x-\eta}(\T) \geq \PPt_{x+\eta}(\T)-\PPt_x(\T).$ } \end{align}
\end{proposition}
\begin{proof}
We follow the strategy of \cite{etheridge2017branching}, adapted to take account of our different choice of voting function $g_\times$. We proceed by induction on the number of branching events. Let $\T_0$ denote a time-labelled tree with a root and a single leaf. Recall that, under $\V^\times$, the initial ancestor is never marked. Denote the transition density of $B(R^\varepsilon_t)$ started at $z\in \RR$ by $p_{z,t}(\cdot)$. Then for $x\geq 0$ and $\eta>0$, $$\PPt_{x} (\T_0)-\PPt_{x-\eta}(\T_0) = \int_{x-\eta}^x p_{0,t}(z)dz\geq \int_x^{x+\eta} p_{0,t}(z)dz = \PPt_{x+\eta}(\T_0)-\PPt_{x}(\T_0).$$ Now assume the inequality holds for all time-labelled ternary trees with at most $n$ branch points. Let $\T$ be a time-labelled ternary tree with $n+1$ internal vertices, and denote the time of the first branching event in $\T$ by $\tau$. Let $\T_1, \T_2$ and $\T_3$ denote the three trees of descent from the first branching event in $\T$. Write \begin{align*} g_\times\left(\PPttau_{x}(\T\star)\right):=g_\times\left(\PPttau_{x}(\T_1),\PPttau_{x}(\T_2),\PPttau_{x}(\T_3)\right).\end{align*} Then \begin{align*}
 \PPt_x(\T)-\PPt_{x-\eta}(\T) 
   &=\EE^\varepsilon_x\left[g_\times\left(\PPttau_{B(R_\tau^\varepsilon)}(\T\star ) \right)\right] -\EE^\varepsilon_{x-\eta}\left[g_\times\left(\PPttau_{B(R_\tau^\varepsilon)}(\T\star ) \right)\right]  \\
   &= \int_{-\infty}^\infty \left[g_\times\left(\PPttau_{y}(\T\star ) \right) - g_\times\left(\PPttau_{y-\eta}(\T\star ) \right) \right]p_{x,\tau}(y)dy\\
   &= \int_0^\infty  \left[g_\times(\PPttau_{y}(\T\star ) ) - g_\times\left(\PPttau_{y-\eta}(\T\star ) \right) \right] (p_{x,\tau}(y)-p_{x,\tau}( -y))dy,
\end{align*}
{w}here the final line follows from the symmetry relations (\ref{symmetry}) and (\ref{symmetryofmarkedg}). \allowdisplaybreaks By identical arguments applied to the right hand side of (\ref{yoga}), we obtain \begin{align*}\allowdisplaybreaks
  &\left(  \PPt_{x}(\T)-\PPt_{x-\eta}(\T) \right) - \left(\PPt_{x+\eta}(\T)-\PPt_{x-\eta}(\T) \right) \\
  &= \int_0^\infty  \left(g_\times\left(\PPttau_{y}(\T\star ) \right) - g_\times\left(\PPttau_{y-\eta}(\T\star ) \right)\right)(p_{x,\tau}(y)-p_{x,\tau}(-y))dy\\& \ \ - \int_0^\infty  \left(g_\times\left(\PPttau_{y+\eta}(\T\star ) \right) - g_\times\left(\PPttau_{y}(\T\star ) \right)\right)   (p_{x,\tau}(y)-p_{x,\tau}(-y))dy.
\end{align*} Since $x\geq 0$, $p_{x,\tau}(y)-p_{x,\tau}(-y)\geq 0$ for $y\geq 0$, and it suffices to show that, for $y\geq 0$, \begin{align}\label{chocolate}  \left(g_\times\left(\PPttau_{y}(\T\star ) \right) - g_\times\left(\PPttau_{y-\eta}(\T\star ) \right)\right) - \left(g_\times\left(\PPttau_{y+\eta}(\T\star ) \right) - g_\times\left(\PPttau_{y}(\T\star ) \right)\right) \geq 0.\end{align} By the inductive hypothesis, for each $i = 1, 2, 3$
\begin{align}\label{caramel}\left(\PPttau_y(\T_i)-\PPttau_{y-\eta}(\T_i)\right)-\left(\PPttau_{y+\eta}(\T_i)-\PPttau_y(\T_i)\right)\geq 0\end{align} for $y\geq 0$. By monotonicity of $g_\times$ and (\ref{caramel}) \begin{align}\label{taylor}
    g_\times\left(\PPttau_{y-\eta}(\T\star )\right) \leq g_\times\left(2\PPttau_y(\T\star)+\PPttau_{y+\eta}(\T\star)\right).
\end{align}
Substituting (\ref{taylor}) into (\ref{chocolate}), it suffices to show \begin{align}\label{vanilla}
\hspace*{-1cm} g_\times\left(\PPttau_{y+\eta}(\T\star)\right)-2g_\times\left(\PPttau_{y}(\T\star)\right)+g_\times\left(2\PPttau_{y}(\T\star)-\PPttau_{y+\eta}(\T\star)\right)\leq 0.
\end{align}  By definition of $g_\times$, (\ref{vanilla}) is equivalent to 
\begin{align}\label{pistachio}
g\left((1-b_\varepsilon)\PPttau_{y+\eta}(\T\star)+\tfrac{b_\varepsilon}{2}\right)-2g\left((1-b_\varepsilon)\PPttau_{y}(\T\star)+\tfrac{b_\varepsilon}{2}\right)\nonumber \\
+g\left((1-b_\varepsilon)(2\PPttau_{y}(\T\star)-\PPttau_{y+\eta}(\T\star))+\tfrac{b_\varepsilon}{2}\right) \leq 0.
\end{align} To see that (\ref{pistachio}) holds, note that \begin{align*}
    &g(p_1+q_1, p_2 +q_2, p_3 +q_3) - 2g(p_1, p_2, p_3) + g(p_1-q_1, p_2-q_2, p_3-q_3) \\
    & \ \ \ = 2q_1 q_2 (1-2p_3) + 2q_2 q_3 (1-2p_1) + 2q_3 q_1(1-2p_2).
\end{align*}
Setting $p_i = (1-b_\varepsilon)\PPttau_{y}(\T_i) + \tfrac{b_\varepsilon}{2}$ and $q_i = (1-b_\varepsilon)\left(\PPttau_{y+\eta}(\T_i) - \PPttau_{y}(\T_i)\right)$, we see that (\ref{pistachio}) will hold if $p_i\geq \tfrac{1}{2}$, or equivalently, $\PPttau_{y}(\T_i)\geq \frac{1}{2},$ which  holds by (\ref{coco}) since $y\geq 0$.
\end{proof}
\noindent With this, we can prove the slope of the interface result.

\begin{corollary} \label{cor:slope}
Let $\varepsilon_1(\alpha)$ and $c_1(\alpha)$ be as in Theorem~\ref{mainteo1dmarked}. Let $\varepsilon < \min(\varepsilon_1, \frac{1}{24}).$ 
Suppose that for some $t\in [0, T^*]$ and $z\in \RR$, $$\left|\PP^\varepsilon_z[\Votemarknop =1]-\tfrac{1}{2}\right| \leq \tfrac{5}{12},$$ and let $w\in \RR$ with $|z-w|\leq c_1( \alpha)I(\varepsilon)|\log \varepsilon|$. Then $$ \left|\PP^\varepsilon_z[\Votemarknop =1]-\PP^\varepsilon_w[\Votemarknop =1]\right| \geq \frac{|z-w|}{48c_1( \alpha)I(\varepsilon)|\log\varepsilon|}.$$  \end{corollary}
\begin{proof}
This follows exactly that of \cite[Corollary 2.12]{etheridge2017branching}, replacing the interface width $\varepsilon|\log\varepsilon|$ with $I(\varepsilon)|\log\varepsilon|.$
\end{proof}

\subsection{Coupling one-dimensional and \texorpdfstring{$\dd$}{d}-dimensional processes}\label{acouplingargument} 
\indent In this section, we will construct a coupling of the the one-dimensional and multidimensional voting systems, so that the results of Section~\ref{sectionmajorityvotinginonedimension} can be used to prove our multidimensional result in the next section. To accomplish this, we require the following regularity properties also used in \cite{etheridge2017branching}, which follow from Assumptions~\ref{assumptions1} by \cite{chen}. Recall that the sets $(\mathbf{\Gamma}_t)_{0\leq t <\mathscr{T}}$ denote the mean curvature flow of $\boldsymbol{\Gamma}_0$ defined in (\ref{gammadefinition}).
\begin{enumerate}[(1)]
    \item There exists $c_0>0$ such that for all $t\in [0,T^*]$ and $x\in\{y  : |d(y,t)|\leq c_0\}$  \begin{align}\label{A1}
        |\nabla d(x,t) | = 1.
    \end{align}
    Moreover, $d$ is a $C^{a, \frac{a}{2}}$ function in $\{(x,t) : |d(x,t)|\leq c_0, t\leq T^*\}$ for $a$ as in Assumption~\ref{assumptions1} (A).
    \item Viewing $\mathbf{n} := \nabla d$ as the positive normal direction, for $x\in \mathbf{\Gamma}_t$, the normal velocity of $\mathbf{\Gamma}_t$ at $x$ is $-d(x,t)$, and the curvature of $\mathbf{\Gamma}_t$ at $x$ is $-\Delta d(x,t)$. Thus, the equation defining mean curvature flow, equation~(\ref{mcfeqn}), becomes \begin{align*}
         \dot{d}(x,t)= \Delta d(x,t)
    \end{align*} for all $x$ such that $d(x,t)=0$.
    \item There exists $C_0>0$ such that for all $t\in [0,T^*]$ and $x$ such that $|d(x,t)|\leq c_0$ for $c_0$ as in the first assumption,\begin{align}
        \label{A3} \left|\nabla \left(\dot{d}(x,t) - \Delta d(x,t)\right)\right|\leq C_0.
    \end{align}
    \item There exists $v_0, V_0>0$ such that for all $t\in [0, T^*-v_0]$ and all $s\in [t, t+v_0]$, \begin{align}
        \label{A4} |d(x,t) - d(x,s)| \leq V_0(s-t).
    \end{align}
\end{enumerate}
The condition (\ref{A1}) ensures that, for all $t\geq 0,$ the region $\{x : d(x,t)\leq c_0\}$ is not self intersecting. That is, for any $x$ with $d(x,t)\leq c_0$, the closed ball centred at $x$ of radius $d(x,t)$ intersects $\mathbf{\Gamma}_t$ at precisely one point.\\
\indent Before explaining our result, let us briefly recall the coupling in \cite{etheridge2017branching} that compares $d(W_s, t-s)$, the distance from a $\dd$-dimensional Brownian motion to $\mathbf{\Gamma}_{t-s}$, to a one-dimensional Brownian motion.
\begin{proposition}[\cite{etheridge2017branching}]\label{sar}
Let $(W_s)_{s\geq 0}$ denote a $\dd$-dimensional Brownian motion started at $x\in \RR^\dd$. Suppose that $t\leq T^*, \beta \leq c_0$ and let $$T_\beta = \inf (\{s\in [0,t):|d(W_s, t-s)|\geq \beta\} \cup \{t\}).$$ Then we can couple $(W_s)_{s\geq 0}$ with a one-dimensional Brownian motion $(B_s)_{s\geq 0}$ started from $z=d(x,t)$ in such a way that for $s\leq T_\beta$, \begin{align}\label{sarahscoupling} B_s - C_0\beta s \leq d(W_s, t-s) \leq B_s + C_0\beta s.\end{align}
\end{proposition}
This result was a key ingredient in the proofs of \cite[Proposition 2.17, Lemma 2.18]{etheridge2017branching} that gave a comparison between the multidimensional and one-dimensional results. It turns out that \cite[Proposition 2.17, Lemma 2.18]{etheridge2017branching} are extremely sensitive to any change in the coupling (\ref{sarahscoupling}). Indeed, if there is any additional drift term in the left and right bounds of (\ref{sarahscoupling}) (that is \textit{not} of the form $f(\varepsilon)s$ for some $f$ satisfying $\lim_{\varepsilon\to 0}f(\varepsilon)=0$), then this error propagates, and the strategy of proof in \cite{etheridge2017branching} no longer works. This provides us with a major hurdle, since, if we mimic the proof of the Brownian coupling but with subordinated Brownian motions, the drift term in (\ref{sarahscoupling}) changes drastically. To overcome this, we employ not one but \textit{two} coupling results. The first, Theorem~\ref{teo:subestimate}, is a straightforward adaptation of Proposition~\ref{sar} to our setting. Our second (and final) coupling will then follow by replacing the multidimensional subordinated Brownian motion in the previous coupling result by one that is \textit{shifted} along an appropriately chosen outward facing normal vector of $\mathbf{\Gamma}_{t-s}$ (this is the content of Theorem~\ref{couplingtheoremforz}).  
\begin{theorem} \label{teo:subestimate}
Let $k\in \NN$. Let $(W_t)_{t \geq 0}$ be a $\mathbbm{d}$-dimensional standard Brownian motion started at $x\in \RR^\dd$, and $(R_t^\varepsilon)_{t \geq 0}$ be an $I(\varepsilon)^2$-truncated $\frac{\alpha}{2}$-stable subordinator satisfying Assumption~\ref{assumption3}. Fix $t\leq T^*$ and $\beta < c_0$ for $c_0$ as in (\ref{A3}). Define the stopping time
\begin{align*}
T_\beta &= \inf(\{s \in [0, (k+1)\varepsilon^2|\log\varepsilon|): |d(W_{s},t-s)|>\beta \}\cup \{t\}). \end{align*}
Fix $s\geq 0$. If $R^\varepsilon_s < T_\beta \wedge t$, then there exists a one-dimensional standard Brownian motion $(B_t)_{t\geq 0}$ started at $d(x,t)$, constants $C_0, D_0>0$ and $\varepsilon_1(k)>0$ such that, for all $\varepsilon\in (0,\varepsilon_1(k))$, with probability at least $1-\varepsilon^{k+1}$,
\begin{align}\label{thisisthecoupling}
|d(W(R^\varepsilon_s),t-s) - B(R^\varepsilon_s)| \leq C_0 \beta s + D_0(k+2)I(\varepsilon)^2|\log\varepsilon|.
\end{align}
\end{theorem}

\begin{proof}
We first rewrite $d(W(R^\varepsilon_s),t-s)$ as
\begin{align} \label{eq:browniandistance}
d(W({R^\varepsilon_s}),t-R^\varepsilon_s) + \left[d(W({R_s^\varepsilon}),t-s) - d(W({R^\varepsilon_s}),t-R^\varepsilon_s)\right].
\end{align}
By the Proposition~\ref{sar}, since $R^\varepsilon_s \leq T_\beta\wedge t$ by assumption, there exists a one-dimensional Brownian motion $(B_t)_{t\geq 0}$ started from $d(x,t)$ and $C_0>0$ such that
\begin{align}\label{purp} B({R^\varepsilon_s}) - C_0 \beta R_s^\varepsilon \leq d(W(R_s^\varepsilon),t-R^\varepsilon_s) 
&\leq  B({R^\varepsilon_s}) + C_0 \beta R_s^\varepsilon. \end{align} By (\ref{A4}), the second term in (\ref{eq:browniandistance}) is bounded as
\begin{align}\label{sparkle} |d(W({R^\varepsilon_s}),t-s) - d(W({R^\varepsilon_s}),t-R^\varepsilon_s)| \leq V_0 |R^\varepsilon_s -s|.\end{align} Combining (\ref{eq:browniandistance}), (\ref{purp}) and (\ref{sparkle}), \begin{align*}
   |d(W(R^\varepsilon_s),t-s) - B(R^\varepsilon_s)| &\leq C_0 \beta R^\varepsilon_s  + V_0|R^\varepsilon_s - s|\\
   &\leq  C_0 \beta s  + D_0 |R^\varepsilon_s -s|
\end{align*} where $D_0 := V_0 + C_0c_0$ for $c_0$ as in (\ref{A1}). By Lemma~\ref{boundonsubordinator}, there exists $\varepsilon_1(k)>0$ such that, for all $\varepsilon \in (0,\varepsilon_1(k)),$ \[\PP(|R^\varepsilon_s - s| > (k+2)I(\varepsilon)^2|\log\varepsilon|)\leq \varepsilon^{k+1} \] and the result follows.\qedhere
\end{proof}
We now define the shifted subordinated Brownian motions that will ultimately move the unwanted drift term in (\ref{thisisthecoupling}) into the multidimensional spatial motion.
 \begin{definition}[$Z^+_s, Z^-_s$] \label{definitionofZ}
 Let $(W(R^\varepsilon_t))_{t \geq 0}$ be a $\mathbbm{d}$-dimensional subordinated Brownian motion. Fix $0< t, T < \mathscr{T}$, $l>0$ and $\beta < c_0$ for $c_0$ as in (\ref{A1}).  Let $x_s \in \mathbf{\Gamma}_{t-s}$ be the unique point on $\mathbf{\Gamma}_{t-s}$ that is the shortest distance from $W(R^\varepsilon_s)$, and  $\mathbf{v}_s$ be the outward facing unit vector perpendicular to the tangent hypersurface of $\mathbf{\Gamma}_{t-s}$ at $x_s$. Then we define the processes $(Z^+_s)_{0 \leq s \leq T}$ and $(Z^-_s)_{0 \leq s \leq T}$ by \begin{align}\label{zplusdef} Z_s^+ &=
 \begin{cases}
      W(R^\varepsilon_s) +l I(\varepsilon)^2|\log \varepsilon|\mathbf{v}_s & \text{ if  } |d(W(R^\varepsilon_s), t-s)| \leq \beta \\
     W(R^\varepsilon_s) & \text{ otherwise.}
 \end{cases}\\
 \label{zminusdef} Z_s^- &=
 \begin{cases}
      W(R^\varepsilon_s) -l I(\varepsilon)^2|\log \varepsilon|\mathbf{v}_s & \text{ if  } |d(W(R^\varepsilon_s), t-s)| \leq \beta \\
     W(R^\varepsilon_s) & \text{ otherwise.}
 \end{cases}\end{align} \end{definition}

\noindent Observe that we may choose $\varepsilon$ sufficiently small so that any point $x$ on the line segment between $W(R_s^\varepsilon)$ and $Z^+_s$ (or $Z^-_s$) satisfies $d(x,t-s) \leq c_0$. Then by (\ref{A1}) $|\nabla d(x, t-s)|=1$ and $\{z: d(z,t)\leq c_0\}$ is not self intersecting. This implies that $\mathbf{\Gamma}_{t-s}$ is sufficiently `flat' near $x$ to ensure that there is a unique point $y\in \mathbf{\Gamma}_{t-s}$ that is the closest point on $\mathbf{\Gamma}_{t-s}$ to both $Z^+_s$ and $W(R^\varepsilon_s)$ (and similarly for $Z^-_s$ and $W(R^\varepsilon_s)$). Therefore, provided $\varepsilon$ is sufficiently small, we have \begin{align}
\label{ineqZminus} & d(Z^+_s,t-s) = d(W(R^\varepsilon_s), t-s)+lI(\varepsilon)^2|\log \varepsilon|\\ 
    &d(Z^-_s,t-s) = d(W(R^\varepsilon_s), t-s)-l I(\varepsilon)^2|\log \varepsilon|.\label{ineqZplus} 
    \end{align}
Consequently, we obtain the following important restatement of Theorem~\ref{teo:subestimate}. \begin{theorem}\label{couplingtheoremforz}
Let $k\in \NN$. For $\alpha \in (1,2)$ and $D_0$ as in Theorem~\ref{teo:subestimate}, let $Z^+_t$ and $Z^-_t$ be as in Definition~\ref{definitionofZ} for $l := D_0 (k+2),$ started at $x\in \RR^\dd$. Let $(R_t^\varepsilon)_{t \geq 0}$ be an $I(\varepsilon)^2$-truncated $\frac{\alpha}{2}$-stable subordinator satisfying Assumption~\ref{assumption3}. Fix $t\leq T^*$, $\beta < c_0$ for $c_0$ as in (\ref{A3}). Define the stopping time
\begin{align*}
T_\beta &= \inf(\{s \in [0, (k+1)\varepsilon^2|\log\varepsilon|): |d(W_{s},t-s)|>\beta \}\cup \{t\}). \end{align*}
Fix $s\geq 0$. If $R^\varepsilon_s < T_\beta \wedge t$, then there exists a one-dimensional standard Brownian motion $(B_t)_{t\geq 0}$ started at $d(x,t)$ and $C_0>0$ such that, with probability at least $1-\varepsilon^{k+1}$,
\begin{align}\label{thisisthecoupling:copiedlabel}
d(Z^+_s,t-s) &\geq B(R^\varepsilon_s)-C_0 \beta s \\
d(Z^-_s,t-s) &\leq B(R^\varepsilon_s)+ C_0 \beta s.\label{oxford}
\end{align}
\end{theorem}
\begin{proof}
This follows from Theorem~\ref{teo:subestimate} and equations (\ref{ineqZminus}), (\ref{ineqZplus}). \end{proof}
    \begin{notation}
    As we have seen in Theorem~\ref{couplingtheoremforz}, $Z^+$ and $Z^-$ satisfy (\ref{thisisthecoupling:copiedlabel}) and (\ref{oxford}) when $l:= D_0(k+2)$ in (\ref{zplusdef}) and (\ref{zminusdef}). For the remainder of this work, we shall take $l:= D_0(k+2)$ in the definition of $Z^+$ and $Z^-$, where the choice of $k$ will be clear in the given context. 
    \end{notation}
\indent Equipped with Theorem~\ref{couplingtheoremforz}, we will be able to use our one-dimensional result, but for the processes ${Z}^+, {Z}^-$ instead of a $\dd$-dimensional stable process. To translate this back to a result in terms of stable processes, we use the following comparison between root votes.\\
\indent Let $\boldsymbol{Z}^+$ be the ternary branching process (with branching rate $\varepsilon^{-2}$) in which individuals independently travel according to $(Z^+_s)_{0\leq s\leq \mathscr{T}}$. Define $\boldsymbol{Z}^-$ similarly. Denote the historical ternary branching process associated to the $R^\varepsilon_t$-subordinated $\dd$-dimensional Brownian motion by $\Ytrunc$.
\begin{proposition} \label{gronwallforZ}
Let $0 < t < \mathscr{T}$, $0 \leq \beta < c_0$, $k \in \mathbb{N}$, $p:\RR^\dd\to [0,1]$ and $F$ be as in (\ref{defnofF}). Let $\boldsymbol{Z}^+(t)$ and $\boldsymbol{Z}^-(t)$ be the historical path of branching processes defined above. Then, for any $x \in \mathbb{R}^\mathbbm{d}$, there exists $m_1, m_2>0$ such that
\begin{align} & \left| \PP_x^\varepsilon[\mathbb{V}^\times_p(\boldsymbol{Z}^-(t))=1)] - \PP_x^\varepsilon[\mathbb{V}^\times_p(\boldsymbol{W}_{\hspace{-.1cm} R^\varepsilon}(t))=1)] \right| \leq m_1e^{-\frac{t}{\varepsilon^2}}+m_2F(\varepsilon) \label{zminuss} \\
 & \left| \PP_x^\varepsilon[\mathbb{V}^\times_p(\boldsymbol{Z}^+(t))=1)] - \PP_x^\varepsilon[\mathbb{V}^\times_p(\boldsymbol{W}_{\hspace{-.1cm} R^\varepsilon}(t))=1)] \right|\leq m_1e^{-\frac{t}{\varepsilon^2}}+m_2F(\varepsilon). \label{zplus} \end{align} 

\end{proposition}\reqnomode
Proposition~\ref{gronwallforZ} is integral to the proof of the main multidimensional result. While intuitively the root votes in (\ref{zminuss}) and (\ref{zplus}) should be close (since the spatial motions are) to obtain the precise bound above requires lengthy calculations which are not illuminating. So as to not disrupt our flow, we defer the proof of Proposition~\ref{gronwallforZ} to Section~\ref{gronwallappendix} of the appendix. 

\begin{remark}
Observe that Proposition \ref{gronwallforZ} marks the first appearance of the term $F(\varepsilon)$ that will later contribute to the sharpness of the interface in Theorem~\ref{maintheorem}. It is also the first time we require that $\alpha \in (1,2)$ and that Assumption~\ref{assumptions2}~\ref{assumptions2_C} holds, to ensure that $F(\varepsilon) \to 0$.
\end{remark}
\section{Majority voting in dimension \texorpdfstring{$\dd \geq 2$}{dg2}\label{ch:3}}
In this section, we will use the one-dimensional result to prove the main multidimensional result, Theorem~\ref{maintheorem2}. To begin, in Section~\ref{sectioncoupvotingd}, we will prove a series of couplings. This will allow us to restate Theorem~\ref{maintheorem2} in terms of the processes $\boldsymbol{Z}^+(t)$ and $\boldsymbol{Z}^-(t)$ in  Theorem~\ref{new_multi_d_theorem}. We go on to prove Theorem~\ref{new_multi_d_theorem} in Section~\ref{generationoftheinterfacesection} and Section~\ref{propinterfacesection} following similar arguments to those in \cite{etheridge2017branching}. The proof of a technical lemma will make up the content of Section~\ref{sectionproofofuglylemma}.\\ 
\indent Let us briefly recall the notation introduced in Section~\ref{ch:2}. We write $X_t$ for the one-dimensional $\alpha$-stable process (with associated historical ternary branching process $\X(t)$), and $Y_t$ for the $\dd$-dimensional $\alpha$-stable process (with associated historical ternary branching process $\Y(t)$). The one-dimensional $R^\varepsilon_t$-subordinated Brownian motion is denoted $B(R_t^\varepsilon)$ (with associated historical ternary branching process $\Xtrunc(t)$), and the $\dd$-dimensional $R^\varepsilon_t$-subordinated Brownian motion is denoted $W(R_t^\varepsilon)$ (with associated historical ternary branching process $\Ytrunc(t)$). As ever, all stable processes and subordinators are assumed to satisfy Assumption~\ref{assumption3}.
\subsection{A coupling of voting systems in higher dimensions}\label{sectioncoupvotingd}
\indent Recall that ${Z}^+(t)$ and ${Z}^-(t)$ satisfy the coupling result Theorem~\ref{couplingtheoremforz}. This is almost identical to the coupling result from the Brownian setting, Proposition~\ref{sar}. Using this and our one-dimensional result (Theorem~\ref{mainteo1dmarked}) it will be straightforward to prove an analogue of Theorem~\ref{maintheorem2} for the processes $\boldsymbol{Z}^+(t)$ and $\boldsymbol{Z}^-(t)$ by adapting the techniques of \cite{etheridge2017branching}. In this section, we will show that this analogue of Theorem~\ref{maintheorem2} for the processes $\boldsymbol{Z}^+(t)$ and $\boldsymbol{Z}^-(t)$ (stated in Theorem~\ref{new_multi_d_theorem}) will imply Theorem~\ref{maintheorem}. To do this, we construct the following couplings: \begin{align*}
    \PP^\varepsilon_x\left[\VV(\boldsymbol{Y}(t))=1\right] &\stackrel{Prop \ \ref{votingcoup2}}{\approx} \PP^\varepsilon_x\left[\VV^+(\boldsymbol{W}_{\hspace{-.1cm} R^\varepsilon}(t))=1\right]\\ &\stackrel{Prop \ \ref{votingcoup1}}{\approx} \PP^\varepsilon_x\left[\VV^\times(\boldsymbol{W}_{\hspace{-.1cm} R^\varepsilon}(t))=1\right] \\
   & \stackrel{Prop \  \ref{gronwallforZ}}{\approx} \PP^\varepsilon_x\left[\VV^\times(\boldsymbol{Z}^-(t))=1\right]
\end{align*} where the voting system $\VV^+$ will be defined in Definition~\ref{abel}. As we will see, a similar series of couplings also relate $\PP^\varepsilon_x\left[\VV(\boldsymbol{Y}(t))=1\right]$ to $\PP^\varepsilon_x\left[\VV^\times(\boldsymbol{Z}^+(t))=1\right]$. Note that the final coupling of $\VV^\times(\boldsymbol{Z}^-(t))$ to $\VV^\times (\boldsymbol{W}_{R^\varepsilon}(t))$ has already been developed in Section~\ref{ch:2}. \\ 
\indent We now proceed to construct a coupling of $\VV(\Y(t))$ to $\VV^\times(\Ytrunc(t))$. To begin, we introduce the positively and negatively biased asymmetric marked voting procedures.

\begin{definition}[$\VV^+, \VV^-$]\label{abel}Let $\varepsilon>0$ and $t\geq0$. Let $b_\varepsilon$ be as in (\ref{bdeltapage}). For a fixed function $p:\RR^{\dd}\to [0,1]$, we define the \textit{positively biased} (resp. \textit{negatively biased}) asymmetric marked voting procedures on $\T(\Ytrunc(t))$ as follows.
\begin{enumerate}[(1)]
\item At each branch point in $\T(\Ytrunc(t))$, the parent particle $j$ marks each of their three offspring $(j,1), (j,2)$ and $(j,3)$ independently with probability $b_\varepsilon$. All marked particles vote $1$ with probability $1$ (resp. $0$ for the negatively biased procedure).
    \item Each unmarked leaf $i$ of $\T(\Ytrunc(t))$, independently, votes $1$ with probability\\ $p(W_i(R^\varepsilon_i(t))$  and otherwise votes $0$.
    \item At each branch point in $\T(\Ytrunc(t))$, if the parent particle $k$ is unmarked, she votes according to the majority vote of her three offspring $(k,1), (k,2)$ and $(k,3)$. 
\end{enumerate}
Define $\VV^+$ (resp. $\VV^-$) to be the vote associated to the root $\emptyset$ of the ternary branching truncated stable tree under the positively biased (negatively biased) asymmetric marked voting procedure described above.
\end{definition}
\noindent We can now prove the first coupling result.
\begin{proposition}\label{votingcoup2}
For all $\varepsilon>0$, $x \in \mathbb{R}^\mathbbm{d}$, $t \geq 0$ and $p:\RR^\dd\to [0,1]$ we have
\begin{align*} (1-b_\varepsilon)\PP^\varepsilon_x[\mathbb{V}_p^-(\Ytrunc(t))=1] \leq \PP^\varepsilon_x[\mathbb{V}_p(\Y(t))=1] \leq (1-b_\varepsilon)\PP^\varepsilon_x[\mathbb{V}_p^+(\Ytrunc(t))=1]+b_\varepsilon. \end{align*}
\end{proposition}
\begin{proof}
This proof proceeds almost identically to the proof of the one-dimensional coupling of voting systems, Theorems~\ref{teo:ineqmark} and \ref{teo:ineqmark2} using an intermediate asymmetric exponentially marked voting system. More specifically, define the negatively biased exponential marked voting procedure $\widehat{\V}^-_p$ like the exponential marked voting procedure from Definition~\ref{definitionnonmarkov}, except that marked individuals vote $1$ with probability $0$. Similarly define the positively biased exponential marked voting procedure $\widehat{\V}^+_p$ where marked individuals vote $1$ with probability $1$. Then, mimicking the proof of Theorem~\ref{teo:ineqmark}, we can show that, for all $x\in \RR^\dd$, \begin{align} \label{comparison_multi_mark} \PP_x^\varepsilon\left[\widehat{\V}^-_p(\Ytrunc(t))=1\right]\leq \PP^\varepsilon_x[\mathbb{V}_p(\Y(t))=1] \leq \PP_x^\varepsilon\left[\widehat{\V}^+_p(\Ytrunc(t))=1\right].\end{align} Finally, by conditioning on whether or not the initial ancestor in $\Ytrunc(t)$ is marked (just as we did in the proof of Theorem~\ref{teo:ineqmark2}) we obtain $$\PP_x^\varepsilon\left[\widehat{\V}^-_p(\Ytrunc(t))=1\right]\geq (1-b_\varepsilon)\PP^\varepsilon_x[\mathbb{V}_p^-(\Ytrunc(t))=1] $$ and $$\PP_x^\varepsilon\left[\widehat{\V}^+_p(\Ytrunc(t))=1\right]\leq (1-b_\varepsilon)\PP^\varepsilon_x[\mathbb{V}_p^+(\Ytrunc(t))=1]+b_\varepsilon,$$ proving the result. 
\end{proof}
In the one-dimensional analogue of Proposition~\ref{votingcoup2} (Theorems~\ref{teo:ineqmark}, \ref{teo:ineqmark2}), we were able to couple $\V(\X(t))$ and $\V^\times(\Xtrunc(t))$ using the (symmetric) exponentially marked voting procedure $\widehat{\V}.$ However, we could not adapt this proof to couple $\VV(\Y(t))$ to $\VV^\times(\Ytrunc(t))$ (having instead to use the auxiliary voting procedures $\VV^+$ and $\VV^-$).  This is because the initial condition $p$ is no longer assumed to be symmetric.\\
\indent To better understand this, let us revisit the proof of Theorem~\ref{teo:ineqmark}, where we showed that, for $x\geq 0$, \begin{align}\label{pepper_steak}
    \PP_x^\varepsilon[\V(\X(t))=1]\geq \PP_x^\varepsilon\left[\widehat{\V}(\boldsymbol{B}_{R^\varepsilon}(t))=1\right]
\end{align} with the reverse inequality holding when $x<0$. We then obtained a coupling of $\V(\X(t))$ to $\V^\times(\boldsymbol{B}_{R^\varepsilon}(t))$ by showing in Theorem~\ref{teo:ineqmark2} that $$\PP_x^\varepsilon\left[\widehat{\V}(\boldsymbol{B}_{R^\varepsilon}(t))=1\right] = (1-b_\varepsilon)\PP_x^\varepsilon[\V^\times(\boldsymbol{B}_{R^\varepsilon}(t))=1]+\frac{b_\varepsilon}{2}.$$ To prove Theorem~\ref{teo:ineqmark}, we used an inductive argument on the number of branching events in $\T(\boldsymbol{B}_{R^\varepsilon}(t))$ and $\T(\X(t)).$ In the base case, when $\T(\X(t)) = \T(\boldsymbol{B}_{R^\varepsilon}(t))=\T_0$, the tree with a single leaf, we considered the single individual in $\T(\X(t)),$ $X(t)\stackrel{D}{=} B(R_t)$ for $B(t)$ a standard one-dimensional Brownian motion and $R(t)$ an $\frac{\alpha}{2}$-stable subordinator. We saw in equation~(\ref{sosleepy}) that, if $\overline{\tau}$ was the first time that $R_t$ made a large jump, and $\tau$ was the time of the first branching event in $\T(\X(t)),$ then if $x\geq 0$ and $p_0(x)=\mathbbm{1}_{\{x\geq 0\}}$ \begin{align}\label{ahalf}\EE_x^\varepsilon[p_0(B(R_t)) \, |\,  \overline{\tau}< t < \tau]\geq \tfrac{1}{2},\end{align} from which it followed that \begin{align}\label{sosleepy2}
\PP_x^t(\T_0) &=\EE^{\varepsilon}_x[p_0\left(B(R_t)\right) \conditional\overline{\tau}\geq t, \tau>t] \PP_x^\varepsilon[\overline{\tau}\geq t]  +\EE^{\varepsilon}_x[p_0(B(R_t)) \conditional \overline{\tau}<t<\tau] \PP_x^\varepsilon[\overline{\tau}<t]\nonumber \\ 
&\geq \EE^{\varepsilon}_x[p_0(B(R^\varepsilon_t))\conditional\tau>t]\,\PP_x^\varepsilon[\overline{\tau}\geq t] + \tfrac{1}{2}\PP_x^\varepsilon[\overline{\tau}<t]. \end{align}
The quantity in (\ref{sosleepy2}) is an upper bound for $\widehat{\PP}_x^t(\T_0),$ so by induction we obtained (\ref{pepper_steak}).\\ 
\indent In the multidimensional setting, for a general initial condition $p$, this argument does not hold. More specifically, (\ref{ahalf}) need not hold since $p$ may not be symmetric; instead, we only have the trivial inequality $\EE_x^\varepsilon[p(W(R_t))\, |\, \overline{\tau}<t<\tau]\geq 0$. Using this, the multidimensional analogue of (\ref{sosleepy2}) becomes $$\PP_x^t(\T_0) \geq \EE^{\varepsilon}_x[p(W(R^\varepsilon_t))\conditional\tau>t]\,\PP_x^\varepsilon[\overline{\tau}\geq t],$$ where the right hand side is an upper bound for the probability that, conditional on $\T(\boldsymbol{W}_{R^\varepsilon}(t))=\T_0$, the single individual votes $1$ under the negatively biased exponentially marked voting procedure $\widehat{\V}_p^-$ defined in the proof of Proposition~\ref{votingcoup2}. Similarly, we can use the trivial bound $\EE_x^\varepsilon[p(W(R_t))\, |\, \overline{\tau}<t<\tau]\leq 1$ to obtain $$\PP_x^t(\T_0) \leq \EE^{\varepsilon}_x[p(W(R^\varepsilon_t))\conditional\tau>t]\,\PP_x^\varepsilon[\overline{\tau}\geq t]+\PP_x^\varepsilon[\overline{\tau}\leq t]$$ where the right hand side is equal to the probability that, conditional on $\T(\boldsymbol{W}_{R^\varepsilon}(t))=\T_0$, the single individual in votes $1$ under the positively biased exponentially marked voting procedure, $\widehat{\V}_p^+.$ These bounds, together with an inductive argument, can be used to prove equation~(\ref{comparison_multi_mark}) from the proof of Proposition~\ref{votingcoup2}.

\begin{remark}
Just as in Remark~\ref{cool}, we can write down the partial differential equation solved by $\PP^\varepsilon_x[\mathbb{V}_p^+(\Ytrunc(t))=1]$ and $\PP^\varepsilon_x[\mathbb{V}_p^-(\Ytrunc(t))=1].$ Denote the infinitesimal generator of $(W(R^\varepsilon_t))_{t\geq 0}$ by $\mathcal{L}^\varepsilon$. Then it is straightforward to verify, using similar arguments to those in the proof of Theorem~\ref{votedual}, that $v_+^\varepsilon(t,x) := \PP^\varepsilon_x[\mathbb{V}_p^+(\Ytrunc(t))=1]$ solves \begin{align*}
    \partial_t v^\varepsilon_+ = \mathcal{L}^\varepsilon v^\varepsilon_+ + {\varepsilon^{-2}}f_+(v^\varepsilon_+), \ \ v_+^\varepsilon(0,x)=p(x)
\end{align*}  and $v_-^\varepsilon(t,x) := \PP^\varepsilon_x[\mathbb{V}_p^-(\Ytrunc(t))=1]$ solves \begin{align*}
    \partial_t v^\varepsilon_- = \mathcal{L}^\varepsilon v^\varepsilon_+ + {\varepsilon^{-2}}f_-(v^\varepsilon_-), \ \ v_-^\varepsilon(0,x)=p(x)
\end{align*} where the nonlinearities $f_+$ and $f_-$ are given by $$f_+(x) := x(1-x)(2x-1) -2b_\varepsilon^3(1-x)^3-3b_\varepsilon^2(1-x)^2(2x-1)+6b_\varepsilon x(1-x)^2 $$ and $$f_-(x) := x(1-x)(2x-1)+2b_\varepsilon^3x^3- 3b_\varepsilon^2x^2(2x-1)-6b_\varepsilon x^2(1-x).$$ Proposition~\ref{votingcoup2} relates solutions to these equations to the original (scaled) fractional Allen--Cahn equation, equation~(\ref{mainequation22}).
\end{remark}
\noindent The positively and negatively biased voting systems can be compared to our (symmetric) marked system as follows. Combining Propositions~\ref{votingcoup2} and \ref{votingcoup1} will give us the desired comparison between $\PP^\varepsilon_x[\mathbb{V}_p(\Y(t))=1]$ and $\PP_x^\varepsilon[\mathbb{V}_p^\times(\Ytrunc(t))=1]$.
\begin{proposition}\label{votingcoup1}
There exists $C>0$ such that, for all $\varepsilon>0$, $x \in \mathbb{R}^\mathbbm{d}$, $t \geq 0$ and $p:\RR^\dd\to [0,1]$, 
\begin{align}
 \sup_{x \in \mathbb{R}^\mathbbm{d}} \left(\PP^\varepsilon_x[\mathbb{V}_p^+(\Ytrunc(t))=1]-\PP^\varepsilon_x[\mathbb{V}_p^\times(\Ytrunc(t))=1]\right) &\leq  C b_\varepsilon  \label{poto1} \\
 \sup_{x \in \mathbb{R}^\mathbbm{d}} \left(\PP_x^\varepsilon[\mathbb{V}_p^\times(\Ytrunc(t))=1]-\PP^\varepsilon_x[\mathbb{V}_p^-(\Ytrunc(t))=1]\right) &\leq C b_\varepsilon. \label{poto2}
 \end{align}
\end{proposition}

\begin{proof}
We prove only (\ref{poto1}), noting that (\ref{poto2}) follows by symmetric arguments. Define $g_+:[0,1] \to [0,1]$ by $g_+(q)=g((1-b_\varepsilon)q+b_\varepsilon)$ where $g$ is the ordinary majority voting function. This is the probability that an unmarked parent particle votes $1$ under $\VV^+$, in the special case when the three offspring are independent and each have probability $q$ of voting $1$ if they are unmarked. Write $\tau$ for the time of the first branching event in $\Ytrunc(\cdot)$. To ease notation, set $$u_\times^\varepsilon(t, x) = \PP_x^\varepsilon[\VV^\times(\Ytrunc(t))=1] \ \text{ and } \  u_+^\varepsilon(t,x)=\PP_x^\varepsilon[\VV^+(\Ytrunc(t))=1].$$ Then, by the Markov property at time $t \wedge \tau$ and definition of $\VV^\times$ and $\VV^+$ (noting that the initial ancestor is never marked in both schemes) we have

\begin{align*}
u_\times^\varepsilon(t, x) &= \EE_x^\varepsilon\left[g_\times(u_\times^\varepsilon(t-\tau,W(R^\varepsilon_{\tau})))\mathbbm{1}_{\tau \leq t}\right] + \EE_x^\varepsilon\left[p(W(R_t^\varepsilon))\mathbbm{1}_{\tau>t}\right] \\
u_+^\varepsilon(t,x) &= \EE_x^\varepsilon\left[g_+\left(u_+^\varepsilon\left(t-\tau,W(R^\varepsilon_{\tau})\right)\right)\mathbbm{1}_{\tau \leq t}] + \EE_x^\varepsilon[p(W(R_t^\varepsilon))\mathbbm{1}_{\tau>t}\right].
\end{align*}
It follows that
\begin{align*}
|u_\times^\varepsilon(t,x) - u_+^\varepsilon(t,x)| & \leq \EE_x^\varepsilon\left[\left|g_\times(u_\times^\varepsilon(t-\tau, W(R^\varepsilon_{\tau}))) - g_+(u_+^\varepsilon(t-\tau, W(R^\varepsilon_{\tau})))\right| \mathbbm{1}_{\tau \leq t} \right].
\end{align*}
By definition of $g_\times$ and $g_+$, and since $g$ is Lipschitz with constant $\tfrac{3}{2}$, we have
\begin{align*}
&|u_\times^\varepsilon(t,x) - u_+^\varepsilon(t,x)| \\ & \leq \tfrac{3}{2}\EE_x^\varepsilon\left[\left| (1-b_\varepsilon)(u_\times^\varepsilon(t-\tau,W(R^\varepsilon_{\tau})) - u_+^\varepsilon(t-\tau, W(R^\varepsilon_{\tau})))- \tfrac{b_\varepsilon}{2} \right|  \mathbbm{1}_{\{\tau \leq t\}} \right] \\ & \leq \tfrac{3}{4} b_\varepsilon + \tfrac{3}{2}(1-b_\varepsilon) \EE_x^\varepsilon\left[\left|u_\times^\varepsilon(t-\tau, W(R^\varepsilon_{\tau}))-u_+^\varepsilon(t-\tau, W(R^\varepsilon_{\tau}))\right|\mathbbm{1}_{\{\tau \leq t\}}\right] \\ 
&=  \tfrac{3}{4} b_\varepsilon + \tfrac{3}{2}(1-b_\varepsilon) \int_0^t \frac{e^{-\rho \varepsilon^{-2}}}{\varepsilon^2} \EE^\varepsilon_x\left[|u_\times^\varepsilon(t-\rho, W(R^\varepsilon_u))-u_+^\varepsilon(t-\rho, W(R^\varepsilon_u))|\right] d\rho \\ &\leq \tfrac{3}{4} b_\varepsilon + \tfrac{3}{2}(1-b_\varepsilon) \int_0^t \frac{e^{-\rho \varepsilon^{-2}}}{\varepsilon^2} \lVert u_\times^\varepsilon(t-\rho, \cdot)-u_+^\varepsilon(t-\rho, \cdot)\rVert_\infty d\rho,
\end{align*}
where $\lVert \cdot \rVert_\infty$ denotes the uniform norm, and we have used that $\tau \sim \mathit{Exp}(\varepsilon^{-2})$ is independent of the spatial motion. Noting that the above inequality holds for all $x\in \RR^\dd$, and applying the change of variables $\rho \mapsto t - \rho$, we obtain
\begin{align*}
\lVert u_\times^\varepsilon(t, \cdot) - u_+^\varepsilon(t, \cdot) \rVert_\infty \leq \tfrac{3}{4}b_\varepsilon + \tfrac{3}{2}e^{-t \varepsilon^{-2}}\int_{0}^t e^{\rho\varepsilon^{-2}} \varepsilon^{-2} \lVert u_\times^\varepsilon(\rho, v) - u_+^\varepsilon(\rho, \cdot)\rVert_\infty d\rho.
\end{align*}
By an adaptation of Gr\"onwall's inequality, available, for instance, in \cite[Theorem 15]{dragomir2002some},
\begin{align*}
\lVert u_\times^\varepsilon(t, \cdot) - u_+^\varepsilon(t, \cdot) \rVert_\infty &\leq \tfrac{3}{4}b_\varepsilon \exp \left(\tfrac{3}{2}\left(\int_0^t \exp\left( - s \varepsilon^{-2} \right) \varepsilon^{-2}\right) ds \right) \\ &= \tfrac{3}{4}b_\varepsilon \exp\left(\tfrac{3}{2} \mathbb{P}[\tau \leq t] \right) \\ & \leq \tfrac{3}{4}b_\varepsilon \exp\left( \tfrac{3}{2} \right).
\end{align*}
Setting $C:= \tfrac{3}{4}\exp\left({\tfrac{3}{2}}\right)$ gives the result. \end{proof}

Now, using the coupling result Proposition~\ref{gronwallforZ} from Section~\ref{acouplingargument}, we obtain our main coupling result of this section.
\begin{corollary} \label{new_corollary} Let $\varepsilon\in (0,1)$, $x\in \RR^\dd$ and $p:\RR^\dd \to [0,1]$. Let $F$ be as in (\ref{defnofF}). Then there exists $a_\dd(\alpha)>0$ and $m>0$ such that, for all $t\geq a_\dd \varepsilon^2|\log \varepsilon|,$ $$\PP^\varepsilon_x[\mathbb{V}_p(\Y(t))=1]\leq \PP^\varepsilon_x[\mathbb{V}_p^\times(\boldsymbol{Z}^-(t))=1] + mF(\varepsilon) + mb_\varepsilon$$ and 
$$\PP^\varepsilon_x[\mathbb{V}_p(\Y(t))=1]\geq (1-b_\varepsilon)\PP^\varepsilon_x[\mathbb{V}_p^\times(\boldsymbol{Z}^+(t))=1] - mF(\varepsilon) - m b_\varepsilon.$$ 
\end{corollary} 
\begin{proof}
   We prove only the first inequality, noting that the second follows by similar arguments. By Propositions~\ref{votingcoup2}, \ref{votingcoup1}, and \ref{gronwallforZ} there exists $m_1, m_2, C>0$ such that \begin{align*} \PP^\varepsilon_x[\mathbb{V}_p(\Y(t))=1]&\leq (1-b_\varepsilon)\PP^\varepsilon_x[\mathbb{V}_p^\times(\Ytrunc(t))=1]+(C+1)b_\varepsilon \\ & \leq (1-b_\varepsilon)\left(\PP^\varepsilon_x[\mathbb{V}_p^\times(\boldsymbol{Z}^-(t))=1]+ m_1e^{-\frac{t}{\varepsilon^2}}+m_2F(\varepsilon)\right)+(C+1)b_\varepsilon.\end{align*} Choose $a_\dd$ sufficiently large so that, for $t\geq a_\dd \varepsilon^2|\log \varepsilon|,$ $e^{-\frac{t}{\varepsilon^2}}\leq F(\varepsilon).$ Choosing $m$ sufficiently large then gives the upper bound. 
\end{proof}

Next, we will state our main theorem for $\boldsymbol{Z}^+(t)$ and $\boldsymbol{Z}^-(t)$ and show using Corollary~\ref{new_corollary} that it implies Theorem~\ref{maintheorem}. Recall that $u_-= \frac{3}{4}b_\varepsilon^2 + \mathcal{O}(b_\varepsilon^3)$ and $u_+= 1-\frac{3}{4}b_\varepsilon^2 + \mathcal{O}(b_\varepsilon^3).$
 \begin{theorem}\label{new_multi_d_theorem} Fix $I$ satisfying Assumptions~\ref{assumptions2} and $k\in \NN$. Suppose the initial condition $p$ satisfies Assumptions~\ref{assumptions1}. Let $\mathscr{T}$ and $d(x, t)$ be as in Section~\ref{secone}, $F$ be as in (\ref{defnofF}) and fix $T^*\in (0, \mathscr{T})$. Let $u_+, u_-$ be as in (\ref{asympfixedpoint1}) and (\ref{asympfixedpoint2}). Then there exists $\varepsilon_\dd(\alpha, k), a_\dd(\alpha, k), c_\dd(\alpha, k)>0$ such that, for $\varepsilon \in (0, \varepsilon_\dd)$ and $a_\dd \varepsilon^2 |\log\varepsilon|\leq t\leq T^*$,
\begin{enumerate}[(1)]
    \item for $x$ with $d(x,t) \geq c_\dd I(\varepsilon)|\log \varepsilon|$, $\PP_x^\varepsilon\left[\VV^\times(\boldsymbol{Z}^+(t)) = 1\right]\geq u_+-\varepsilon^k,$
    \item for $x$ with $d(x,t) \leq -c_\dd I(\varepsilon)|\log \varepsilon|,$ $\PP_x^\varepsilon\left[\VV^\times(\boldsymbol{Z}^-(t)) = 1\right]\leq u_-+\varepsilon^k.$
\end{enumerate}
\end{theorem} 
\noindent To see that this implies Theorem~\ref{maintheorem},  let $k\in \NN$ and suppose $$\PP_x^\varepsilon\left[\VV^\times(\boldsymbol{Z}^+(t)) = 1\right]\geq u_+-\varepsilon^k.$$ By Corollary~\ref{new_corollary}, this implies $$\PP^\varepsilon_x[\mathbb{V}_p(\Y(t))=1] \geq (1-b_\varepsilon)(u_+-\varepsilon^k)-mF(\varepsilon)-mb_\varepsilon$$ for some $m>0$. Since $u_+\geq 1-b_\varepsilon$, it is straightforward to see that, for $\varepsilon>0$ sufficiently small, we may increase $m$ as necessary so that $$\PP^\varepsilon_x[\mathbb{V}_p(\Y(t))=1] \geq 1-mF(\varepsilon)-m \frac{\varepsilon^2}{I(\varepsilon)^2}.$$ Similar arguments using Theorem~\ref{new_multi_d_theorem}~(2) prove the lower bound in Theorem~\ref{maintheorem}.

\subsection{Generation of the interface}\label{generationoftheinterfacesection}
We now show that in a time $\mathcal{O}(\varepsilon^2|\log\varepsilon|)$, an interface of width $\mathcal{O}(I(\varepsilon)|\log\varepsilon|)$ is created. Here, we refer to the solution interface associated to the partial differential equation solved by $\PP_x^\varepsilon[\VV^\times(\boldsymbol{Z}^-(t))=1]$ with initial condition $p$. We will make use of the following one-dimensional result, where we recall that $\V^\times = \V^\times_{\widehat{p}_0}$ is the marked majority voting system with initial condition
$$\widehat{p}_0(x) = u_+\mathbbm{1}_{\{x\geq 0\}}+u_-\mathbbm{1}_{\{x< 0\}}. $$
\begin{proposition}\label{boundsofvote} 
Let $a_1$ be as in Lemma~\ref{defnofa} and fix $k\in \NN$. Then there exists $\varepsilon_\dd(k) >0$ such that, for all $\varepsilon\in (0,\varepsilon_\dd)$ , $t\geq a_1(k)\varepsilon^2|\log\varepsilon|$ and $x\in \RR$, \begin{align}\label{greenwallet} u_- - \varepsilon^k \leq \PP^\varepsilon_x[\V^\times(\boldsymbol{B}_{R^\varepsilon}(t))=1]\leq u_+ + \varepsilon^k.\end{align} 
\end{proposition}
\begin{proof}
We prove the right hand inequality in (\ref{greenwallet}) and note that the left hand inequality follows by very similar arguments. It is easy to verify that $\delta:= g_\times'(u_+) = \mathcal{O}(b_\varepsilon)<1$, and by the Mean Value Theorem, since $g_\times'$ is decreasing on $[u_+, 1],$ for all $q\in (0, 1-u_+]$, $$g_\times(u_++q) - g_\times(u_+) \leq \delta q.$$ Since $u_+$ is a fixed point of $g_\times,$ for $\varepsilon$ sufficiently small, $g_\times(q)<q$ for $q \in (u_+, 1]$ so iterating the above inequality as in the proof of Lemma~\ref{lemma:iterative_voting} gives us $$g_\times^{(n)}(u_++q)-u_+ \leq \delta^n q $$ for all $q\in (0, 1-u_+]$. In particular, $$g^{(n)}_\times(u_+ + (1-u_+)) - u_+ \leq \delta^n(1-u_+)\leq \varepsilon^k $$ after $n\geq C(k)|\log \varepsilon|$ iterations, for some $C(k)>0$. That is, $g^{(n)}_\times(1)\leq u_+ +\varepsilon^k$ if $n\geq C(k)|\log\varepsilon|$. We note that, since $g_\times$ is increasing on $[0,1]$, the largest value of the iterates of $g_\times(x)$ will be when $x = 1$. Finally, by Lemma~\ref{defnofa}, for $t\geq a_1(k)\varepsilon^2|\log\varepsilon|,$ $$\PP^\varepsilon_x\left[\T(\boldsymbol{B}_{R^\varepsilon}(t))\supset \T^{reg}_{A(k)|\log\varepsilon|}\right]\geq 1-\varepsilon^k. $$ Therefore when $t\geq a_1(k)\varepsilon^2 |\log\varepsilon|, \PP^\varepsilon_x[\V^\times(\boldsymbol{B}_{R^\varepsilon}(t))=1] \leq u_+ + 2\varepsilon^k.$
\end{proof}
It is straightforward to adapt the proof of Lemma~\ref{defnofa} to show that, for all $k\in \NN$ and $A(k)$ as in Lemma~\ref{lemma:iterative_voting}, there exists $\rho_\dd^+(k)>0$ and $\varepsilon_\dd>0$ such that, for all $\varepsilon\in (0, \varepsilon_\dd)$, $x\in \RR^\dd$ and $t\geq \rho_\dd^+(k)\varepsilon^2|\log \varepsilon|$, 
$$\PP^\varepsilon_x\left[\T(\boldsymbol{Z}^+(t))\supset \T^{reg}_{A(k)|\log\varepsilon|}\right]\geq 1-\varepsilon^k. $$ Similarly, there exists $\rho_\dd^-(k)$ and $\varepsilon_\dd'>0$ such that, for all $\varepsilon\in (0, \varepsilon_\dd')$ and $t\geq \rho_\dd^-\varepsilon^2|\log\varepsilon|,$ $$\PP^\varepsilon_x\left[\T(\boldsymbol{Z}^-(t))\supset \T^{reg}_{A(k)|\log\varepsilon|}\right]\geq 1-\varepsilon^k. $$ With this, we can adapt the proof of Proposition~\ref{boundsofvote} to show that, for any $k\in \NN$ and $\varepsilon$ sufficiently small, if $t\geq \rho_\dd^+(k)\varepsilon^2|\log\varepsilon|$,\begin{align}\label{rhoplus} u_- - \varepsilon^k \leq \PP^\varepsilon_x[\V^\times_p(\boldsymbol{Z}^+(t))=1]\leq u_+ + \varepsilon^k\end{align} for any initial condition $p$, and if $t\geq \rho_\dd^-(k)\varepsilon^2|\log\varepsilon|$ \begin{align}\label{rhominus} u_- - \varepsilon^k \leq \PP^\varepsilon_x[\V^\times_p(\boldsymbol{Z}^-(t))=1]\leq u_+ + \varepsilon^k.\end{align}   
In our later proofs, we will want (\ref{greenwallet}), (\ref{rhoplus}) and (\ref{rhominus}) to hold simultaneously, so it will be useful to define \begin{align}\label{defrho}
    \rho_\dd(k) := \rho_\dd^-(k)\vee\rho_\dd^+(k)\vee a_1(k).
\end{align}

\begin{proposition} \label{generationoftheinterface}
Let $k\in \NN$ and $\rho_\dd(k)$ be as in equation~(\ref{defrho}). Fix $I$ satisfying Assumptions~\ref{assumptions2} and let $u_+, u_-$ be as in (\ref{asympfixedpoint1}), (\ref{asympfixedpoint2}). Then there exists $\varepsilon_\dd(\alpha, k), b_\dd(\alpha, k) >0$ such that, for all $\varepsilon \in (0,\varepsilon_\dd),$ if \begin{align*} t_\dd(k, \varepsilon)&:= \rho_\dd(k)\varepsilon^2|\log\varepsilon|, \\  t'_\dd(k,\varepsilon)&:= (2 \rho_\dd(k)+k+1)\varepsilon^2|\log\varepsilon|,\ \end{align*}then for $t\in [t_\dd, t'_\dd],$
\begin{enumerate}[(1)]
    \item for $d(x,t) \geq b_\dd(k) I(\varepsilon)|\log\varepsilon|$, we have $\PP^\varepsilon_x[\VM(\Zbf^-(t))=1] \geq u_+-\varepsilon^k,$
    \item for $d(x, t) \leq -b_\dd(k) I(\varepsilon)|\log\varepsilon|$,  we have $\PP^\varepsilon_x[\VM(\Zbf^-(t))=1] \leq u_-+\varepsilon^k$.
\end{enumerate}
\end{proposition}
\begin{remark}
By almost identical arguments, Proposition~\ref{generationoftheinterface} holds when $\Zbf^-$ is replaced with $\Zbf^+$. Note that our choice of $t_\dd$ and $t_\dd'$ are stricter than needed for this result alone, but it will be convenient to define them in this way for use in later proofs.
\end{remark}

\begin{proof}
We follow the proof of \cite[Proposition 2.16]{etheridge2017branching} closely, and consider the multidimensional analogues of Lemmas~\ref{defnofa} and \ref{displacement lemma}. First, by choice of $\rho_\dd^-$, there exists $\varepsilon_\dd(\alpha, k)>0$ such that, for all $\varepsilon \in (0,\varepsilon_\dd)$, $x\in \RR^\dd$ and $t\geq \rho_\dd^-(k) \varepsilon^2 |\log\varepsilon|,$
\begin{align}\label{multidtree}
    \PP^\varepsilon_x \left[\mathcal{T}(\Zbf^-(t)) \supseteq \mathcal{T}^{\text{reg}}_{A(k)|\log\varepsilon|}\right]\geq 1 - \varepsilon^{k}
\end{align} for $A(k)$ as in Lemma~\ref{lemma:iterative_voting}. By standard estimates for the multidimensional standard normal variable, the proof of Lemma~\ref{displacement lemma} can be adapted to show that there exists $h_\dd(k)>0$ and $\varepsilon_\dd(k)>0$ such that, for all $\varepsilon \in (0,\varepsilon_\dd)$ and $t\in [t_\dd,t_\dd']$, \begin{align}\label{displacementmultid}
    \PP_x^\varepsilon\left[\exists i\in N(s) : | W_i(R^\varepsilon_i(s))-x| \geq h_\dd(k) I(\varepsilon)|\log\varepsilon| \right] \leq \varepsilon^k.
\end{align} 
By definition of $Z_s^-$, $| Z^-_s| \leq | W(R^\varepsilon_s)| + D_0(k+2)I(\varepsilon)^2|\log\varepsilon|,$ 
for $D_0$ the constant from Theorem~\ref{teo:subestimate}. Therefore
(\ref{displacementmultid}) implies that, for $\varepsilon$ sufficiently small, there exists $l_\dd(k) > h_\dd(k)$ for which \begin{align*}
    \PP_x^\varepsilon\left[\exists i\in N(s) : | Z^-_i(s)-x| \geq l_\dd(k) I(\varepsilon)|\log\varepsilon| \right] \leq \varepsilon^k.
\end{align*} 
Set $b_\dd = 2l_\dd$. Recall that $d(x,t)$ is the signed distance between $x\in \RR^\dd$ and $\mathbf{\Gamma}_{t}$. By the regularity assumption on $\mathbf{\Gamma}_{t}$ (\ref{A4}), there exist $v_0, V_0>0$ such that, for $t\leq v_0$ and $x\in \RR^\dd$, $|d(x,0)-d(x,t)|\leq V_0t$. Reduce $\varepsilon_\dd$ if necessary so that $t_\dd'\leq v_0$ for all $\varepsilon\in (0,\varepsilon_\dd).$ Let $\varepsilon\in (0,\varepsilon_\dd)$, $t\in [t_\dd, t_\dd']$ and $x$ be such that $d(x,t)\geq b_\dd I(\varepsilon)|\log\varepsilon|$ and $| Z^-_i(t)-x|\leq l_\dd I(\varepsilon)|\log\varepsilon|$. It follows by the triangle inequality and Lipschitz continuity of $d(\cdot, t)$ that 
\begin{align*}
    d(Z^-_i(t),0) &\geq d(x,t) - \left|d(x,t)-d(Z^-_i(t),t)\right|-\left| d(Z^-_i(t),t)-d(Z^-_i(t),0)\right|\\
    &\geq b_\dd I(\varepsilon)|\log\varepsilon| - l_\dd I(\varepsilon)|\log\varepsilon|- V_0t_\dd'\\
    &=\tfrac{1}{2}b_\dd I(\varepsilon)|\log\varepsilon| - V_0(2\rho_\dd +k+1)\varepsilon^2|\log\varepsilon|.
\end{align*}
By Assumption~\ref{assumptions2}~\ref{assumptions2_B} we may reduce $\varepsilon_\dd$ if necessary so that $$d(Z^-_i(t),0)\geq \tfrac{1}{4}b_\dd I(\varepsilon)|\log\varepsilon|$$ for all $\varepsilon\in (0,\varepsilon_\dd).$ Finally, by Assumptions~\ref{assumptions1}~(B)-(C), and reducing $\varepsilon_\dd$ if necessary, \begin{align}\label{12ineq}
    p(Z^-_i(t)) &\geq \tfrac{1}{2}+ \gamma\left(\tfrac{1}{4}b_\dd I(\varepsilon)|\log\varepsilon| \wedge r\right)\nonumber \\
    &\geq \tfrac{1}{2}+\varepsilon
\end{align} for all $\varepsilon\in (0,\varepsilon_\dd)$. We then combine (\ref{multidtree}), (\ref{displacementmultid}) and (\ref{12ineq}) exactly as the proof of Theorem~\ref{mainteo1dmarked}, to obtain that, for $\varepsilon\in (0,\varepsilon_\dd)$, $t\in [t_\dd, t_\dd']$ and $x$ such that $d(x,t)\geq b_\dd I(\varepsilon)|\log\varepsilon|,$ $$\PP_x^\varepsilon[\VV^\times(\Zbf^-(t))=1] \geq u_+-3\varepsilon^k.$$
The upper bound is obtained using the same approach.
\end{proof}

\subsection{Propagation of the interface}\label{propinterfacesection}
In this section, we will compare $\VV^\times(\boldsymbol{Z}^-(t))$ to $\V^\times(\Xtrunc(t))$, and use this to show that the interface propagates. Throughout this section, define \begin{align}\label{defn_gamma} \gamma(t):= K_1e^{K_2t}I(\varepsilon)|\log\varepsilon|\end{align}  where the choice of $K_1, K_2$ and $\varepsilon$ will be clear in the given context.

\begin{proposition}\label{propogationofinterface}
Let $l \in \mathbb{N}$ with $l\geq 4$ and fix $I$ satisfying Assumptions~\ref{assumptions2}. Let $t_\dd$ be as in Proposition~\ref{generationoftheinterface}. There exist $K_1, K_2>0$ such that for $\gamma(\cdot)$ as in (\ref{defn_gamma}), $\varepsilon\in (0,\varepsilon_\dd)$ and $t\in [t_\dd(l), T^*]$ we have \begin{align}\label{firstpropagation}
\sup_{x\in\RR^\dd} \left(\PP^\varepsilon_x[\VM(\boldsymbol{Z}^-(t)) =1]- \PP^\varepsilon_{d(x,t)+\gamma(t)}[\V^\times(\Xtrunc (t))=1]\right) \leq \varepsilon^l \ \ \ \ \ \end{align} 
and
\begin{align} \label{Secondpropagation}
    \sup_{x\in \RR^\dd} \left(\PP^\varepsilon_x[\VM(\boldsymbol{Z}^+(t))=0]-\allowdisplaybreaks \PP^\varepsilon_{d(x,t)-\gamma(t)}[\V^\times(\Xtrunc (t))=0]\right)\leq \varepsilon^l. \ \ \ \ \
\end{align}
\end{proposition}

\noindent Throughout this section, we will extend the domain of $g_\times:[0,1]\to[0,1]$ to all of $\RR$. Namely, we set $$g_\times(p) = \begin{cases} g_\times(0)& \text{if } p<0 \\
g_\times(p) &\text{if } p\in [0,1]\\
g_\times(1)  & \text{if } p>1.
\end{cases}$$
\indent Key to the proof of Proposition~\ref{propogationofinterface} will be Lemma~\ref{biguglylemma}, which parallels \cite[Lemma 2.18]{etheridge2017branching}. The proof of Theorem~\ref{new_multi_d_theorem} will then follow easily. We defer the lengthy proof of Lemma~\ref{biguglylemma} to Section~\ref{sectionproofofuglylemma}.

\begin{lemma}\label{biguglylemma} Let $K_1>0$, $l \in \mathbb{N}$ with $l\geq 4$, and fix $I$ satisfying Assumptions~\ref{assumptions2}. Let $t_\dd'$ be as in Proposition~\ref{generationoftheinterface}. Then there exists $K_2 = K_2(K_1, l)>0$ and $\varepsilon_\dd(K_1, K_2, l)>0$ such that, for $\gamma(\cdot)$ as in (\ref{defn_gamma}), $\varepsilon\in (0, \varepsilon_\dd)$, $x\in \RR^\dd$, $s\in [0, (l+1) \varepsilon^2|\log\varepsilon|]$ and $t\in [t'_{\mathbbm{d}}(l), T^*]$, 
\begin{align}\label{2.18eqn1} &\EE_x\left[g_\times\left(\PP^\varepsilon_{d(Z^-_s, t-s) + \gamma(t-s)}[\mathbb{V}^\times(\Xtrunc(t-s))=1] + \varepsilon^l \right) \right] \nonumber \\
& \ \  \leq \frac{3}{4} \varepsilon^l+ \EE_{d(x,t)}\left[g_\times\left(\mathbb{P}^{\varepsilon}_{B(R_s^\varepsilon)  +\gamma(t)}[\mathbb{V}^\times(\Xtrunc (t-s))=1] \right) \right] + \mathbbm{1}_{\{s \leq \varepsilon^3\}}\varepsilon^l \end{align} 
and  \begin{align}\label{2.18eqn2} & \EE_x\left[g_\times\left(\PP^\varepsilon_{d(Z^+_s, t-s) - \gamma(t-s)}[\mathbb{V}^\times(\Xtrunc(t-s))=0] + \varepsilon^l \right)\right] \nonumber \\
& \ \   \leq \frac{3}{4} \varepsilon^l +\EE_{d(x,t)}\left[g_\times\left(\mathbb{P}^{\varepsilon}_{B(R_s^\varepsilon)-\gamma(t)}[\mathbb{V}^\times(\Xtrunc (t-s))=0]\right)\right]+  \mathbbm{1}_{\{s \leq \varepsilon^3\}}\varepsilon^l. \end{align}
\end{lemma}

\begin{proof}[Proof of Proposition~\ref{propogationofinterface}] We only prove (\ref{firstpropagation}), since (\ref{Secondpropagation}) follows by completely symmetric arguments. 
Set $K_1 = b_\dd(l) +c_1(l)$ for $b_\dd$ as in Proposition~\ref{generationoftheinterface} and $c_1$ as in Theorem~\ref{mainteo1dmarked}. Take $\varepsilon_\dd>0$ sufficiently small so that Theorem~\ref{mainteo1dmarked}, Proposition~\ref{generationoftheinterface}, Proposition~\ref{boundsofvote} and Lemma~\ref{biguglylemma} hold for all $\varepsilon \in (0, \varepsilon_\dd)$. We first observe that, for $\varepsilon \in (0,\varepsilon_\dd)$, $t\in [t_\dd(l), t'_\dd(l)]$ (for $t_\dd$ and $t_\dd'$ as in Proposition~\ref{generationoftheinterface}) and $x\in \RR^\dd$, 
\begin{align}\label{koala}
  \PP^\varepsilon_x[\VM(\Zbf^-(t)) =1] -  \PP^\varepsilon_{d(x,t)+\gamma(t)}[\V^\times(\Xtrunc(t))=1] \leq \varepsilon^l.
\end{align} 
To see this, first suppose that $d(x, t) \leq -b_\dd(l) I(\varepsilon)|\log\varepsilon|$. Now, reducing $\varepsilon_\dd$ if necessary, by Proposition~\ref{generationoftheinterface} 
\begin{align*}
\PP^\varepsilon_x[\VM(\Zbf^- (t)) =1] \leq u_- + \varepsilon^l.
\end{align*}
Also, by Proposition~\ref{boundsofvote},
\begin{align*}
\PP^\varepsilon_{d(x,t)+\gamma(t)}[\V^\times(\Xtrunc(t))=1] &\geq u_- -\varepsilon^l,
\end{align*} 
hence (\ref{koala}) holds. Here, we continue to ignore coefficients in front of polynomial error terms following Remark~\ref{coeffignore}. If we added a coefficient to the error term in (\ref{koala}), it would appear in all polynomial error terms that follow, but would not affect our proof.\\
\indent Now suppose $d(x,t) \geq -b_\dd(l) I(\varepsilon)|\log\varepsilon|$. Then, reducing $\varepsilon$ if necessary, $$d(x,t) + \gamma(t)\geq c_1(l) I(\varepsilon)|\log\varepsilon|,$$ so by Theorem~\ref{mainteo1dmarked}, $\PP^\varepsilon_{d(x,t)+\gamma(t)}[\V^\times(\Xtrunc(t))=1] \geq u_+-\varepsilon^l$. By (\ref{rhominus}), $$\PP^\varepsilon_x[\VM(\Zbf^-(t))=1]\leq u_+ +\varepsilon^l,$$ and (\ref{koala}) holds.\\ 
\indent It remains to verify (\ref{koala}) for $t\in [t_\dd', T^*]$. Assume for the purpose of a contradiction that there exists $t\in [t_\dd', T^*]$ such that, for some $x\in \RR^\dd$, 
$$ \PP^\varepsilon_x[\VM(\Zbf^-(t))=1]- \PP^\varepsilon_{d(x,t)+\gamma(t)}[\V^\times(\Xtrunc (t))=1] >\varepsilon^l. $$
Let $T'$ be the infimum of the set of such $t$, and choose
$$T\in [T', \text{min}(T' +\varepsilon^{l+3}, T^*)], $$
which is in the set of such $t$. So there exists some $x \in \RR^\dd$ such that 
\begin{align*} \PP^\varepsilon_x[\VM(\Zbf^-(T))=1]-\PP^\varepsilon_{d(x,T)+\gamma(T)}[\V^\times(\Xtrunc (T))=1]> \varepsilon^l.\ \ \ \ \ \ \ \end{align*}
We will show \begin{align}\label{cake2}   \PP^\varepsilon_x[\VM(\Zbf^-(T))=1] \leq \tfrac{7}{8}\varepsilon^l+\PP^\varepsilon_{d(x,T)+\gamma(T)}[\V^\times(\Xtrunc (T))=1]. \ \ \ \ \ \  \end{align}
Let $\tau$ be the time of the first branching event in $\Zbf^-(T)$ and $Z_\tau^-$ be the position of the initial ancestor particle at that time. Then, by the Strong Markov Property at time $\tau \wedge (T-t_\dd),$ \begin{align}\label{markovpropeqn}
    \PP_x^\varepsilon[\VV^\times(\Zbf^-(t))=1]= \EE_x^\varepsilon\left[g_\times(\PP_{Z^-_\tau}^\varepsilon [\VV^\times(\Zbf^-(T-\tau))=1]) \mathbbm{1}_{\tau \leq T-t_\dd}\right]\nonumber\\ 
    + \EE_x^\varepsilon\left[\PP^\varepsilon_{Z^-_{T-t_\dd}}[\VV^\times(\Zbf^-(t_\dd))=1]\mathbbm{1}_{\tau\geq T-t_\dd}\right].
\end{align}
Since $T-t_\dd \geq t_\dd'- t_\dd > (l+1)\varepsilon^2|\log\varepsilon|$ and $\tau\sim \mathit{Exp}(\varepsilon^{-2}),$ the second term on the right side of (\ref{markovpropeqn}) is bounded by 
\begin{align} \nonumber \EE_x^\varepsilon\left[\PP^\varepsilon_{Z^-_{T-t_\dd}}[\VV(\Zbf^+(t_\dd))=1]\mathbbm{1}_{\tau\geq T-t_\dd}\right]&\leq \PP\left[\tau\geq (l+1)\varepsilon^2|\log\varepsilon|\right]\\ &=\varepsilon^{l+1}.\label{2.49}\end{align}

To bound the first term on the right hand side of (\ref{markovpropeqn}), we partition over the event $\{\tau\leq \varepsilon^{3+l}\}$ (which has probability $\leq \varepsilon^{l+1}$) and its complement to obtain
\begin{flalign}\label{2.50}
 & \EE_x^\varepsilon\left[g_\times(\PP_{Z^-_\tau}^\varepsilon [\VV^\times(\Zbf^-(T-\tau))=1]) \mathbbm{1}_{\tau \leq T-t_\dd}\right] \nonumber \\
 & \leq \EE_x^\varepsilon\left[g_\times(\PP_{Z_\tau^-}^\varepsilon [\VV^\times(\Zbf^+(T-\tau))=1]) \mathbbm{1}_{ \varepsilon^{l+1}\leq \tau \leq T-t_\dd}\right] + \varepsilon^{l+1} \nonumber  \\
   &  \leq  \EE_x^\varepsilon\left[g_\times\left(\PP^\varepsilon_{d(Z^-_\tau,T-\tau)+ \gamma(T-\tau)}[\mathbb{V}^\times(\Xtrunc(T-\tau))=1] + \varepsilon^l \right) \mathbbm{1}_{\tau \leq T-t_\dd} \right] + \varepsilon^{l+1},
\end{flalign}
where the last line follows by minimality of $T'$, since $\varepsilon^{l+3} \leq \tau \leq T-t_\dd$, so $T-\tau \in [t_\dd, T')$, and by monotoncity of $g_\times$. Then, conditioning on the value of $\tau$, and noting that the path of the ancestral particle $(B(R^\varepsilon_\cdot))$ is independent of $\tau$,\allowdisplaybreaks \begin{align}\label{2.51}
   &\EE_x^\varepsilon\left[g_\times\left(\PP^\varepsilon_{d(Z^-_\tau,T-\tau)+\gamma(T-\tau)}[\mathbb{V}^\times(\Xtrunc(T-\tau))=1] + \varepsilon^l \right) \mathbbm{1}_{ \tau \leq T-t_\dd} \right]\nonumber \\
   &\leq \int_0^{(l+1)\varepsilon^2|\log\varepsilon|}\frac{e^{-\varepsilon^{-2}s}}{\varepsilon^2} \EE_x\left[g_\times\left(\PP^\varepsilon_{d(Z^-_s,T-s)+ \gamma(T-s)}[\mathbb{V}^\times(\Xtrunc(T-s))=1] + \varepsilon^l \right) \right]ds\nonumber \\
   & \ \ \ \ + \PP[\tau\geq (l+1)\varepsilon^2|\log\varepsilon]]\nonumber\\
   &\leq \int_0^{(l+1)\varepsilon^2|\log\varepsilon|}\frac{e^{-\varepsilon^{-2}s}}{\varepsilon^2}  \EE_{d(x,T)}\left[g_\times\left(\PP^\varepsilon_{B(R_s^\varepsilon)+\gamma(T)}[\mathbb{V}^\times(\Xtrunc(T-s))=1]\right) \right]ds \nonumber \\
   & \ \ \ \ + \varepsilon^{l+1} + \varepsilon^{l}\left(\tfrac{3}{4} + \PP[\tau \leq \varepsilon^3]\right)\nonumber\\
   &\leq \EE_{d(x,T)}^\varepsilon\left[g_\times \left(\PP^\varepsilon_{B(R^\varepsilon_{\tau'})+\gamma(T)}[\mathbb{V}^\times(\Xtrunc(T-\tau'))=1]\right)\mathbbm{1}_{\tau'\leq T-t_\dd} \right]\nonumber\\ &\ \ \ \ + \tfrac{3}{4}\varepsilon^l + 2\varepsilon^{l+1},
\end{align}
where the second inequality follows by Lemma~\ref{biguglylemma}. Here $\tau'$ denotes the time of the first branching event in $\Xtrunc$, which has the same distribution as $\tau$. The final inequality holds since $T\geq t_\dd'$, so $T-t_\dd\geq (l+1)\varepsilon^2|\log\varepsilon|$. Putting  (\ref{2.49}), (\ref{2.50}) and (\ref{2.51}) into (\ref{markovpropeqn}), we obtain
\begin{align*}
    \PP_x^\varepsilon[\VV^\times(\boldsymbol{Z}^-(T))=1] &\leq\EE_{d(x,T)}^\varepsilon\left[g_\times\left(\PP^\varepsilon_{B(R^\varepsilon_{\tau'})+\gamma(T)}[\mathbb{V}^\times(\Xtrunc(T-\tau'))=1]\right)\mathbbm{1}_{\tau'\leq T-t_\dd} \right]\\ & \ \ + \tfrac{3}{4} \varepsilon^l + 4 \varepsilon^{l+1}\\
   & \ \ \leq  \PP^\varepsilon_{d(x,T)+\gamma(T)}[\mathbb{V}^\times(\Xtrunc(T))=1] + \tfrac{3}{4} \varepsilon^l + 4 \varepsilon^{l+1} ,
\end{align*}
where the last line follows by the Strong Markov Property for $(\Xtrunc(\cdot))$ at time $\tau'\wedge (T-t_\dd)$. We can reduce $\varepsilon_\dd$ if necessary so that $4 \varepsilon^{l+1} + \tfrac{3}{4} \varepsilon^l \leq \tfrac{7}{8} \varepsilon^l$ for all $\varepsilon \in (0,\varepsilon_\dd)$. This gives (\ref{cake2}), thereby proving (\ref{firstpropagation}). The inequality (\ref{Secondpropagation}) follows by a similar argument, using (\ref{2.18eqn2}).
\end{proof}

\noindent \noindent With this, we can now prove Theorem~\ref{new_multi_d_theorem}.
\begin{proof}[Proof of Theorem \ref{new_multi_d_theorem}]
Set $c_\dd(l) := c_1(l) +K_1e^{K_2T^*}$. Then, for any $x\in \RR^\dd$ and $t\in [t_\dd, T^*]$ such that $d(x,t)\leq -c_\dd(l)I(\varepsilon)|\log\varepsilon|$, we have
$$d(x,t)+K_1e^{K_2 t}I(\varepsilon)|\log\varepsilon| \leq - c_1(l)I(\varepsilon)|\log\varepsilon|.$$ Then, by Theorem~\ref{mainteo1dmarked}, reducing $\varepsilon_\dd$ if necessary so that $\varepsilon_\dd<\varepsilon_1(l),$ $$\PP_x^\varepsilon[\VV^\times(\boldsymbol{Z}^-(t))=1)]\leq u_- + 2\varepsilon^l.$$ Similarly, for $x$ and $t$ such that $d(x,t)\geq c_\dd(l)I(\varepsilon)|\log\varepsilon|$, by Theorem~\ref{mainteo1dmarked} and (\ref{Secondpropagation}), $\PP_x^\varepsilon[\VV^\times(\boldsymbol{Z}^+(t))=1)]\geq u_+-2\varepsilon^l$. Theorem~\ref{new_multi_d_theorem} then holds by setting $a_\dd:=\rho_\dd$.
\end{proof}
\begin{proof}[Proof of Theorem \ref{maintheorem}]
This follows immediately by combining Theorem~\ref{new_multi_d_theorem} and Corollary~\ref{new_corollary}.
\end{proof}

\subsection{Proof of Lemma \ref{biguglylemma}}\label{sectionproofofuglylemma}
 To prove Lemma~\ref{biguglylemma}, we follow the proof of \cite[Lemma 2.18]{etheridge2017branching} and consider separately the cases $|d(x,t)| \leq DI(\varepsilon)|\log\varepsilon|$ and $|d(x,t)| \geq DI(\varepsilon)|\log\varepsilon|$, for some large $D>0$. Since neither the one-dimensional process $B(R^\varepsilon_s)$ nor the multidimensional process $Z^+$ (or $Z^-$) travel further than a distance $\mathcal{O}(I(\varepsilon)|\log\varepsilon|)$ in time $s = \mathcal{O}(\varepsilon^2|\log(\varepsilon)|)$ with high probability, if $D$ is sufficiently large and $|d(x,t)|\leq D I(\varepsilon)|\log\varepsilon|$,  we will see that the result follows from the main one-dimensional result for $\VV^\times(\Xtrunc)$, Theorem~\ref{mainteo1dmarked}.  When $|d(x,t)|\leq D I(\varepsilon)|\log\varepsilon|,$ we apply Theorem~\ref{couplingtheoremforz} so that, with probability at least $1-\varepsilon^{l+1},$ $$d(Z^-_s, t-s)\leq B(R_s^\varepsilon) + \mathcal{O}(I(\varepsilon)|\log\varepsilon|)s.$$
 Using this and monotonicity of $g_\times$ we can bound the left hand side of (\ref{2.18eqn1}) by
 \begin{align}\label{roses} \EE_{d(x,t)}\left[g_\times\left(\mathbb{P}^{\varepsilon}_{B(R_s^\varepsilon) +\gamma(t-s)+\mathcal{O}(s)I(\varepsilon)|\log\varepsilon|}[\V^\times(\Xtrunc(t))=1]+\varepsilon^l\right)\right] + \mathcal{O}(\varepsilon^{l+1}). \end{align} 
 We then control (\ref{roses}) by considering two cases: when the argument of $g_\times$ in (\ref{roses}) is bounded away from $\tfrac{1}{2}$, and when it is close to $\tfrac{1}{2}$. In the former case, we use that $|g_\times'(y)| < \tfrac{2}{3}$ when $y$ is bounded far enough away from $\tfrac{1}{2}$, together with monotonicty of $g_\times$, to obtain (\ref{2.18eqn1}). In the second case, we apply the slope of the interface result, Corollary~\ref{cor:slope}, to bound the difference between the two expectations appearing in the inequality (\ref{2.18eqn1}) directly. 
\begin{proof}[Proof of Lemma \ref{biguglylemma}]
Fix $l\geq 4$. For all $u\geq 0$ and $z \in \RR$, let 
$$\mathbb{Q}_z^{\varepsilon, u} = \PP_z^\varepsilon[\mathbb{V}^\times(\BB_{R^\varepsilon}(u))=1].$$ Let $C_0$ be as in (\ref{A3}) and $c_1$ be defined as in Theorem~\ref{mainteo1dmarked}. Let
\begin{align} \label{defnofR} R:= 2c_1(l)+4(l+1) \dd+1\end{align}
and fix $K_2$ such that
\begin{align}\label{K2} K_1(K_2-C_0)-C_0R-C_1= c_1(1). \end{align}
\noindent To start we let $\varepsilon_\dd(l) = \varepsilon_1(l)$ where $\varepsilon_1(l)$ is defined in Theorem~\ref{mainteo1dmarked}. \\
\indent Following the proof of \cite[Lemma 2.18]{etheridge2017branching}, we begin by estimating the probability that a $\dd$-dimensional subordinated Brownian motion moves further than a distance $I(\varepsilon)|\log \varepsilon|$ in time $s\leq (l+1)\varepsilon^2 |\log\varepsilon|$. Define the event 
$$A_x = \left\{\sup_{u\in [0,R^\varepsilon_s]} | W_u-x |\leq 2 (l+1) I(\varepsilon)|\log\varepsilon|\right\}. $$
Then, bounding $| W_{u}-x |$ by the sum of the moduli of $\dd$ one-dimensional Brownian motions, and by Lemma~\ref{boundonsubordinator}, which bounds the displacement of the subordinator $R^\varepsilon_s$ for small times, we obtain
\begin{align}\label{Acomp}
    \PP_x[A_x^c] &\leq 2\dd \PP_0\left[\sup_{u\in[0,R_s^\varepsilon]} B_u> 2(l+1) I(\varepsilon)|\log\varepsilon|\right]\nonumber \\
    &\leq 2\dd \PP_0\left[\sup_{u\in[0,(l+2)I(\varepsilon)^2|\log\varepsilon|]} B_u> 2(l+1) I(\varepsilon)|\log\varepsilon|\right] + 2\dd \varepsilon^{l+1}\nonumber \\
    &\leq 4\dd \PP_0\left[B_1 > 2((l+1)|\log\varepsilon|)^{1/2}\right] + 2\dd \varepsilon^{l+1}\nonumber \\
    &\leq 6\dd\varepsilon^{l+1}
\end{align}
where the second inequality follows by the reflection principle, and the final inequality follows by identical arguments to those in the proof of Lemma~\ref{displacement lemma}. Now consider the cases
\begin{enumerate}[(i)]
    \item $d(x, t) \leq -(2c_1(l) +2(l+1) \dd + K_1e^{K_2(t-s)})I(\varepsilon)|\log\varepsilon|$
    \item $d(x, t) \geq (2c_1(l) +2(l+1) \dd + K_1e^{K_2(t-s)})I(\varepsilon)|\log\varepsilon|$
    \item $|d(x, t)| \leq (2c_1(l) +2(l+1) \dd + K_1e^{K_2(t-s)})I(\varepsilon)|\log\varepsilon|$.\end{enumerate}
    \textbf{Case (i):} By (\ref{A3}), there exist $v_0, V_0>0$ such that, if $s\leq v_0$ and $x\in \RR^\dd$, then $$|d(x,t)-d(x,t-s)|\leq V_0s. $$ Reduce $
    \varepsilon_\dd$ if necessary to ensure that, for all $\varepsilon \in (0,\varepsilon_{\dd}),$ $(l+1)\varepsilon^2|\log\varepsilon| \leq v_0.$ Then if $A_x$ occurs, \begin{align*}
        &d(W(R^\varepsilon_s), t-s) + K_1e^{K_2(t-s)}I(\varepsilon)|\log\varepsilon| \\&\leq -(2c_1(l) + 2(l+1) \dd) I(\varepsilon)|\log\varepsilon| + |d(W(R^\varepsilon_s), t-s) - d(x,t)|\nonumber\\
        &\leq -(2c_1(l) + 2(l+1) \dd)I(\varepsilon)|\log\varepsilon| + |d(x,t) - d(x,t-s)| + |W(R^\varepsilon_s) - x| \nonumber\\
        &\leq -2c_1(l)I(\varepsilon)|\log\varepsilon| +V_0(l+1)\varepsilon^2|\log\varepsilon|.
    \end{align*} 
    By Assumption~\ref{assumptions2}~\ref{assumptions2_B}, we may reduce $\varepsilon_\dd$ if necessary so that 
    $$d(W(R^\varepsilon_s), t-s) + K_1e^{K_2(t-s)}I(\varepsilon)|\log\varepsilon| \leq -c_1( l)I(\varepsilon)|\log\varepsilon|, $$
    for all $\varepsilon \in (0,\varepsilon_\dd).$ Then, since $d(Z^-_s, t-s) \leq d(W(R_s^\varepsilon), t-s),$
    $$d(Z^-_s, t-s) + K_1e^{K_2(t-s)}I(\varepsilon)|\log\varepsilon| \leq -c_1(l)I(\varepsilon)|\log\varepsilon|.$$ 
    So, by Theorem~\ref{mainteo1dmarked} and definition of $g_\times$, \begin{align*}
        \EE_x\left[g_\times\left(\mathbb{Q}^{\varepsilon,t-s}_{d(Z^-_s, t-s)+\gamma(t-s)}+\varepsilon^l\right)\right]&\leq \EE_x[g_\times(u_- + 2\varepsilon^l)\mathbbm{1}_{A_x}] + \PP_x\left[A_x^c\right]\\
        &\leq u_- + 6\dd\varepsilon^{l+1}+12\varepsilon^{l}b_\varepsilon
    \end{align*} where the last line follows by calculating $g_\times(u_- + 2\varepsilon^l)$ explicitly and reducing $\varepsilon_\dd$ if necessary. \\
    \indent  Next, recall that $g_\times(y) = g((1-b_\varepsilon)y+\tfrac{b_\varepsilon}{2})$ for $y\in [0,1].$ So 
     \[ g_\times'(y) = 6(1-b_\varepsilon)\left((1-b_\varepsilon)y+\tfrac{b_\varepsilon}{2}\right)\left(1-\left((1-b_\varepsilon)y+\tfrac{b_\varepsilon}{2}\right)\right).\]
     Hence, if \begin{align}\label{tay} (1-b_\varepsilon)(y+\delta)+\tfrac{b_\varepsilon}{2}\leq \tfrac{1}{9} \ \ \text{or} \ \ (1-b_\varepsilon)y+\tfrac{b_\varepsilon}{2}\geq \tfrac{8}{9}\end{align} then \begin{align}\label{geqnp}
    g_\times(y+\delta)\leq g_\times(y)+\tfrac{2}{3}\delta.
\end{align}
    From Proposition~\ref{boundsofvote} (since $t-s\geq \rho_\dd(l)\varepsilon^2|\log\varepsilon|$) and (\ref{geqnp}), reducing $\varepsilon$ if necessary,
    \[ \EE_{d(x,t)}\left[g_\times\left(\mathbb{Q}^{\varepsilon,t-s}_{B(R^\varepsilon_s) +\gamma(t)} \right)\right] \geq u_- - \tfrac{2}{3} \varepsilon^l . \] 
    By choosing $\varepsilon$ small enough so that $6 \dd \varepsilon^{l+1} + 12\varepsilon^{l}b_\varepsilon + \tfrac{2}{3} \varepsilon^l \leq \tfrac{3}{4}\varepsilon^{l}$ (\ref{2.18eqn1}) holds in this case. \\ \text{}\\ 

\noindent \textbf{Case (ii):} Suppose now that $d(x, t) \geq (2c_1(l) +2(l+1) \dd + K_1e^{K_2(t-s)})I(\varepsilon)|\log\varepsilon|$. Using this, together with a similar argument to that used to obtain  (\ref{Acomp}), we have \begin{align*}
    \PP_{d(x,t)}[|B(R^\varepsilon_s)| \geq c_1(l)I(\varepsilon)|\log\varepsilon|] \leq \varepsilon^{l+1}.
\end{align*}
It follows that \begin{align*}
    \EE_{d(x,t)}\left[g_\times \left(\QQ^{\varepsilon, t-s}_{B(R_s^\varepsilon)+\gamma(t)}\right)\right] &\geq \EE_{d(x,t)}\left[g_\times \left(\QQ^{\varepsilon, t-s}_{B(R_s^\varepsilon) +\gamma(t)}\right)\mathbbm{1}_{\{B(R_s^\varepsilon)\geq c_1(l)I(\varepsilon)|\log\varepsilon|\}} \right]\\
    &\geq g_\times(u_+-\varepsilon^l) -\varepsilon^{l+1}\\
    &\geq u_+-\varepsilon^{l+1}-12\varepsilon^{l}b_\varepsilon
\end{align*} where the second inequality follows by Theorem~\ref{mainteo1dmarked}, and in the third inequality we expand $g_\times(u_+ - \varepsilon^l)$ and reduced $\varepsilon_\dd$ if necessary. From Proposition \ref{boundsofvote} we can get that, for $\varepsilon$ small enough, 
\begin{align*}
\EE_x\left[g_\times\left(\mathbb{Q}^{\varepsilon,t-s}_{d(Z^-_s, t-s)+\gamma(t-s)}+\varepsilon^l\right)\right] \leq u_+ + \tfrac{2}{3}\varepsilon^l.
\end{align*}
Hence, reducing $\varepsilon$ if necessary (\ref{2.18eqn1}) holds in this case. \text{}\\

\noindent \textbf{Case (iii):} Finally, suppose $|d(x, t)| \leq (2c_1(l) +2(l+1) \dd + K_1e^{K_2(t-s)})I(\varepsilon)|\log(\varepsilon)|$.
 If $A_x$ occurs and $u\in [0,(l+2)I(\varepsilon)^2|\log\varepsilon|]$, 
\begin{align*}
 &   |d(W({R_u^\varepsilon}), t-u)| \\ &\leq |W({R^\varepsilon_u}) - x| + |d(x, t)| + |d(x, t)-d(x, t-u)| \\
    &\leq (2c_1(l) + 4(l+1) \dd +K_1e^{K_2(t-s)})I(\varepsilon)|\log\varepsilon| + V_0(l+2)I(\varepsilon)^2|\log(\varepsilon)|.   \end{align*}
    Therefore, reducing $\varepsilon$ if necessary, with probability at least $1-\varepsilon^{l+1}$, 
    $$|d(W(R^\varepsilon_s),t-s)| \leq (R+K_1e^{K_2(t-s)})I(\varepsilon)|\log\varepsilon|$$ for $R$ as in (\ref{defnofR}). Now set \begin{align}
\beta = (R+K_1e^{K_2(t-s)})I(\varepsilon)|\log\varepsilon|.
\label{defnbeta}\end{align}
\noindent Reduce $\varepsilon_\dd$ if necessary so that, for all $\varepsilon \in (0, \varepsilon_\dd)$, $\beta \leq c_0/2$, for $c_0$ as in (\ref{A1}). Recall that
\[T_\beta = \inf\left(\{s \in [0, (l+1)\varepsilon^2|\log\varepsilon|): |d(W_{s},t-s)|>\beta \}\cup \{t\}\right).  \]
    Note that $\PP[R^\varepsilon_s >T_\beta] \leq 2\varepsilon^{l+1}:$ by the above calculation, if $A_x$ occurs, then $T_\beta > R^\varepsilon_s$ with probability at least $\geq 1-\varepsilon^{l+1}$. Therefore, by Theorem~\ref{couplingtheoremforz}, and reducing $\varepsilon$ if necessary so that $T_\beta <t,$
\begin{align}
d(Z^-_s, t-s) &\leq B(R_s^\varepsilon) + C_0\beta s \label{boundexpE} \end{align} 
with probability at least $1-2\varepsilon^{l+1}$. Then, by monotonicity of $g_\times$ and (\ref{boundexpE}), partitioning over $\{R^\varepsilon_s > T_\beta\}$ and $A_x$, we obtain
\begin{align}\label{qeq}
&\EE_x\left[g_\times\left(\mathbb{Q}^{\varepsilon,t-s}_{d(Z^-_s, t-s)+\gamma(t-s)}+\varepsilon^l \right)\right] \\ \nonumber &\leq \EE_{d(x,t)}\left[g_\times\left(\mathbb{Q}^{\varepsilon,t-s}_{B(R^\varepsilon_s) + C_0\beta s + \gamma(t-s)}+\varepsilon^l\right)\right] + (2+6\dd) \varepsilon^{l+1}. 
\end{align}
Let $$D := \left\{\left|\mathbb{Q}^{\varepsilon,t-s}_{B(R^\varepsilon_s) + C_0\beta s + \gamma(t-s)}-\tfrac{1}{2} \right|\leq \tfrac{5}{12}\right\}. $$ 
\noindent We consider $D$ and $D^c$ separately to bound the right hand side of (\ref{qeq}). First suppose the event $D$ occurs. Then, by definition of $\beta$ (\ref{defnbeta}), \begin{align}\label{Keqn}
    &K_1e^{K_2t}I(\varepsilon)|\log\varepsilon| - \left(C_0\beta s+ K_1e^{K_2(t-s)}I(\varepsilon)|\log\varepsilon| \right)\nonumber \\
    &= \left(K_1e^{K_2(t-s)}\left(e^{K_2s}-1-C_0s\right)-C_0Rs\right) I(\varepsilon)|\log\varepsilon| \nonumber \\
    &\geq \left(K_1(K_2-C_0)-C_0R\right)sI(\varepsilon)|\log\varepsilon| \nonumber \\
    &=c_1(1)sI(\varepsilon)|\log\varepsilon| 
\end{align} 
where the final equality follows by (\ref{K2}). Reducing $\varepsilon_\dd$ if necessary so that $\varepsilon_\dd < \min(\varepsilon_1(1), \tfrac{1}{24}),$ for $\varepsilon \in (0, \varepsilon_\dd)$ we may apply Corollary \ref{cor:slope} with $$z= B(R_s^\varepsilon)+C_0\beta s+K_1e^{K_2(t-s)}I(\varepsilon)|\log\varepsilon|$$ and $$w=z+ c_1(1)sI(\varepsilon)|\log\varepsilon| \leq B(R_s^\varepsilon) + \gamma(t)$$ to give \begin{align}\label{2.63}
    \QQ^{\varepsilon,t-s}_{B(R_s^\varepsilon) + C_0\beta s + \gamma(t-s)}\mathbbm{1}_D \leq  \left(\QQ^{\varepsilon,t-s}_{B(R_s^\varepsilon) +  \gamma(t)}-\tfrac{1}{48}s\right)\mathbbm{1}_D.
\end{align} 
\noindent Now suppose the event $D^c$ occurs. Reduce $\varepsilon_\dd$ if necessary so that $$\tfrac{1}{12}<\tfrac{1}{9}-\varepsilon^l(1-b_\varepsilon)-\tfrac{b_\varepsilon}{2},$$ which implies (\ref{tay}) for $\delta = \varepsilon^l$. Thus, for $\varepsilon \in (0,\varepsilon_\dd)$, we have
\begin{align}\label{2.65}
    g_\times\left(  \QQ^{\varepsilon,t-s}_{B(R_s^\varepsilon) + C_0\beta s + \gamma(t-s)} + \varepsilon^l \right)\mathbbm{1}_{D^c} &\leq \left(g_\times\left(  \QQ^{\varepsilon,t-s}_{B(R_s^\varepsilon)+C_0\beta s +\gamma(t-s)}\right) + \tfrac{2}{3}\varepsilon^l \right)\mathbbm{1}_{D^c} \nonumber \\
    &\leq \left(g_\times\left(  \QQ^{\varepsilon,t-s}_{B(R_s^\varepsilon)+\gamma(t)}\right) + \tfrac{2}{3}\varepsilon^l \right)\mathbbm{1}_{D^c}
\end{align} where the first inequality follows by (\ref{geqnp}) and the second inequality by (\ref{Keqn}) and monotonicity of $g_\times$. Putting (\ref{2.63}) and (\ref{2.65}) into (\ref{qeq}), and since $2+6\dd \leq 8\dd$ we obtain
\begin{align*}
    & \EE_x\left[g_\times\left(\mathbb{Q}^{\varepsilon,t-s}_{d(Z^-_s, t-s)+\gamma(t-s)}+\varepsilon^l\right)\right]\\
    & \leq  \EE_{d(x,t)}\left[g_\times\left(\mathbb{Q}^{\varepsilon,t-s}_{B(R_s^\varepsilon)+\gamma(t)}-\tfrac{1}{48}s+\varepsilon^l\right)\mathbbm{1}_D\right]+  \EE_{d(x,t)}\left[\left(g_\times\left(\mathbb{Q}^{\varepsilon,t-s}_{B(R_s^\varepsilon)+\gamma(t)}\right)+\tfrac{2}{3}\varepsilon^l\right)\mathbbm{1}_{D^c}\right] \\ & 
 \ \ +8\dd\varepsilon^{l+1}\\
     &\leq \EE_{d(x,t)}\left[g_\times\left(\mathbb{Q}^{\varepsilon,t-s}_{B(R_s^\varepsilon)+\gamma(t)}\right)\right]+\tfrac{2}{3}\varepsilon^l + \varepsilon^l\mathbbm{1}_{\{\tiny\frac{1}{48}s\leq \varepsilon^l\}} +8\dd\varepsilon^{l+1}
\end{align*}
where in the final inequality, we use that $g_\times'(y)\leq \tfrac{3}{2}$ for all $y\in [0,1]$. Further reducing $\varepsilon_\dd$ if necessary so that $8\dd\varepsilon^{l+1}\leq \tfrac{1}{12}\varepsilon^l$ and $48 \varepsilon^l \leq \varepsilon^3$ for all $\varepsilon \in (0,\varepsilon_\dd)$ gives the result.
\end{proof} 
\appendix
\section{Appendix}
In this appendix, we will calculate the fixed points of $g_\times$ (Section~\ref{app_g}), prove Proposition~\ref{gronwallforZ} (Section~\ref{gronwallappendix}) and provide several supplementary calculations for the truncated subordinator $R^\varepsilon_s$ (Section~\ref{appsub}).
\subsection{Fixed points of \texorpdfstring{$g_\times$}{g}}\label{app_g}

\begin{proposition}\label{appendixfixedpointsofg}
The function $g_\times$ has fixed points $u_-, \tfrac{1}{2},$ and $u_+$, where \reqnomode \begin{align*}
    u_- &= \tfrac{1}{2}- \tfrac{\sqrt{(1-b_\varepsilon)^3(1-3b_\varepsilon)}}{2(1-b_\varepsilon)^3}, \ \  u_+ = \tfrac{1}{2}+ \tfrac{\sqrt{(1-b_\varepsilon)^3(1-3b_\varepsilon)}}{2(1-b_\varepsilon)^3}.
\end{align*}
\end{proposition}
\begin{proof}
We aim to find $a$ such that $g_\times\left(\tfrac{1}{2} + a\right) = \tfrac{1}{2}+a$. Now, \begin{align*}
    g_\times\left(\tfrac{1}{2} + a\right) &= 3\left((1-b_\varepsilon)(\tfrac{1}{2} + a) + \tfrac{b_\varepsilon}{2}\right)^2-2\left((1-b_\varepsilon)(\tfrac{1}{2} + a) + \tfrac{b_\varepsilon}{2}\right)^3\\
    &= 2(b_\varepsilon-1)^3a^3 + \tfrac{3}{2}(1-b_\varepsilon)a + \tfrac{1}{2}.
\end{align*} Setting this equal to $\tfrac{1}{2}+a$ we obtain the quadratic equation 
$$2(1-b_\varepsilon)^3a^2 +\tfrac{3}{2}(1-b_\varepsilon) = 0,$$ 
for which $u_- - \tfrac{1}{2}$ and $u_+ - \tfrac{1}{2}$ are clearly solutions. To see that $g_\times(\tfrac{1}{2}) = \tfrac{1}{2},$ note that \[ g_\times\left(\tfrac{1}{2}\right) = g\left((1-b_\varepsilon)\tfrac{1}{2}+ \tfrac{b_\varepsilon}{2}\right) = g\left(\tfrac{1}{2}\right) = \tfrac{1}{2}. \qedhere \]
\end{proof}
\subsection{Proof of Proposition~\ref{gronwallforZ}}\label{gronwallappendix}
\noindent Before proving Proposition~\ref{gronwallforZ}, we will need the following result.
\begin{proposition} \label{regofdensity}
Let $h_t(x,\cdot)$ denote the transition density of a $\dd$-dimensional Brownian motion $W_t$ started at $x$. There exists a constant $C>0$ such that, for all $r\geq 0,$
\[ |h_r(x,y) - h_r(x,y+z) | \leq C r^{-\frac{\mathbbm{d}+1}{2}} |  z | \]
for all $x,y, z \in \mathbb{R}^\mathbbm{d}$.
\end{proposition}
\begin{proof}
Fix $r\geq 0$. Since
\[ h_r(x,y) = \frac{1}{(4 \pi r)^{\frac{\mathbbm{d}}{2}}} \exp\left(- \tfrac{| x - y |^2}{4 r} \right),\]
by the Mean Value Theorem
\begin{align*}
|h_r(x,y) - h_r(x,y+z) | \leq \frac{1}{(4 \pi r)^{\frac{\mathbbm{d}}{2}}} | z | \left| \nabla \exp\left(-\tfrac{| \xi | ^2}{4r}\right) \right|
\end{align*}
for some $\xi$ on the line segment between $y-x$ and $y+z-x$. Now,
\begin{align*}
\left| \nabla \exp\left(-\tfrac{|\xi|^2}{4r}\right) \right| = \frac{| \xi |}{2r} \exp\left(-\tfrac{| \xi |^2}{4r} \right) \leq {r^{-\frac{1}{2}}} 
\end{align*}
 since $x e^{-x^2}\leq 1$ for all $x\in \RR$. The result follows by setting $C:= (4\pi)^{-\frac{\dd}{2}}.$
\end{proof} 
\noindent We now prove Proposition~\ref{gronwallforZ}.\begin{proof}[Proof of Proposition~\ref{gronwallforZ}]\reqnomode
We prove (\ref{zplus}) and note that (\ref{zminuss}) follows by identical arguments. Denote the standard Euclidean distance in $\RR^\dd$ as $|\cdot|$ throughout. To begin, we will construct a coupling of $\boldsymbol{Z}^+(t)$ to a historical ternary branching subordinated Brownian motion, $\Ytrunc(t)$. Define $$ u(t,x):=  \PP_x[\mathbb{V}^\times_p(\boldsymbol{Z}^+(t))=1] \ \text{ and } \ v(t,x):= \PP_x[\mathbb{V}^\times_p(\boldsymbol{W}_{R^\varepsilon}(t))=1].$$
Abuse notation and let $\tau$ denote the time of the first branching event in both $\boldsymbol{Z}^+(t)$ and $\Ytrunc(t)$. By the Markov property at time $\tau$, $u$ and $v$ can be written as 
\begin{align}\label{cafe}u(t,x) &= \EE_x[g_\times(u(t-\tau,Z^+_\tau)) \mathbbm{1}_{\tau \leq t}] + \EE_x[p(Z^+_t)\mathbbm{1}_{\tau > t}]\\
\label{hotchoccy} v(t,x) &= \EE_x[g_\times(v(t-\tau, W(R_\tau^\varepsilon))) \mathbbm{1}_{\tau \leq t}] + \EE_x[p(W(R_t^\varepsilon)) \mathbbm{1}_{\tau > t}].\end{align}
To bound the difference of $u$ and $v$, we control the difference of the first terms and second terms in (\ref{cafe}) and (\ref{hotchoccy}) separately. First, since $\tau\sim \mathit{Exp}(\varepsilon^{-2}),$
\begin{align}\label{red}
\EE_x[p(Z^+_t)\mathbbm{1}_{\tau > t}] - \EE_x[p(W(R_t^\varepsilon)) \mathbbm{1}_{\tau > t}] \leq  \PP[t \leq \tau] \leq   e^{-\frac{t}{\varepsilon^2}}.
\end{align}
To bound the difference of the first terms in (\ref{cafe}) and (\ref{hotchoccy}), 
set $t^* := k\varepsilon^2|\log(\varepsilon)|$ and $\delta_\varepsilon:=D_0(k+2)I^2(\varepsilon)|\log(\varepsilon)|$ for $D_0$ as in Theorem~\ref{teo:subestimate}. Denote the transition density of $W(R^\varepsilon_t)$ started at $x$ by $f_t(x,\cdot)$. 
Then, since $g_\times$ is bounded above by one and, by definition of $Z^+_t$, \allowdisplaybreaks
\begin{align}
 \nonumber &\EE_x\left[g_\times(u(t-\tau,Z^+_\tau))\mathbbm{1}_{\tau \leq t}\right] \\ \nonumber & \leq\EE_{x}\left[g_\times(u(t-\tau, Z^+_\tau))\mathbbm{1}_{\tau \leq t \wedge t^*}\right] + \PP[\tau>t^*] \\ \nonumber & \leq \sup_{\substack{| w | \leq \delta_\varepsilon}} \EE_{x}\left[g_\times(u(t-\tau, W(R^\varepsilon_\tau)+w))\mathbbm{1}_{\tau \leq t \wedge t^*}\right]  + \varepsilon^k \\ \nonumber & = \sup_{\substack{| w | \leq \delta_\varepsilon}} \mathbb{E}\left[\int_{\RR^\dd} g_\times(u(t-\tau,z+w)) f_\tau(x,z) dz \mathbbm{1}_{\tau \leq t \wedge t^*}\right] + \varepsilon^k \\ \nonumber &=  \sup_{\substack{| w | \leq \delta_\varepsilon}} \mathbb{E}\left[\int_{\RR^\dd} g_\times(u(t-\tau,z)) f_\tau(x,z- w) dz \mathbbm{1}_{\tau \leq t \wedge t^*}\right]  + \varepsilon^k  \\ \nonumber & \leq \mathbb{E}\left[\int_{\RR^\dd} g_\times(u(t-\tau,z)) f_\tau(x,z) dz \mathbbm{1}_{\tau \leq t \wedge t^*}\right] + \varepsilon^k \\ \label{last}& \ \ \  +  \sup_{\substack{| w | \leq \delta_\varepsilon}} \EE\left[\int_{\RR^\dd} g_\times(u(t-\tau,z))[f_\tau(x,z-w)-f_\tau(x,z)] \mathbbm{1}_{\tau \leq t \wedge t^*} dz\right]. 
\end{align} 
Consider the $\dd$-dimensional ball $\B_x\left(\tau^{{\frac{1}{\alpha}}}+\delta_\varepsilon\right) :=\left\{ z \in \mathbb{R}^\mathbbm{d} : | z-x | \leq \tau^{\frac{1}{\alpha}} +\delta_\varepsilon\right\}$. To ease notation, let $\B_x := \B_x\left(\tau^{{\frac{1}{\alpha}}}+\delta_\varepsilon\right)$. 
Here, $\B_x$ has been chosen so that, for $z \in \B_x$, the difference $|f_\tau(x, z-w) - f_\tau(x,z)|$ is sufficiently small, and for $z\in \B_x^{\mathsf{c}}$ (the complement of $\B_x$ in $\RR^\dd$), the probability of the subordinated Brownian motion jumping from $x$ to $z$ by time $\tau$ is sufficiently small. Then, to bound the last term in (\ref{last}) we use that $g(u(t-\tau, z))$ is bounded by $1$ to get
\begin{align}
\nonumber & \sup_{\substack{| w | \leq \delta_\varepsilon}} \EE\left[\int_{\RR^\dd} g(u(t-\tau, z))[f_\tau(x,z-w)-f_\tau(x,z)] \mathbbm{1}_{\tau \leq t \wedge t^*} dz\right] \\ \nonumber & \leq \sup_{\substack{| w | \leq \delta_\varepsilon}} \EE\left[\int_{\RR^\dd} |f_\tau(x,z-w)-f_\tau(x,z)| \mathbbm{1}_{\tau \leq t \wedge t^*}  dz \right] \\ \nonumber & \leq \sup_{\substack{| w | \leq \delta_\varepsilon}} \EE\left[\int_{\B_x} |f_\tau(x,z-w)-f_\tau(x,z)| dz \mathbbm{1}_{\tau \leq t \wedge t^*} \right] \\ & \label{twointegrals}\ \ \ + \sup_{| w| \leq \delta_\varepsilon} \EE\left[\int_{\B_x^{\mathsf{c}}} |f_\tau(x,z-w)-f_\tau(x,z)| dz  \mathbbm{1}_{\tau \leq t \wedge t^*} \right].
\end{align}
To bound the first term of (\ref{twointegrals}), first note that
 \[ \PP[\tau < \delta_\varepsilon^{\alpha}] \leq 1-\exp\left( -{\delta_\varepsilon^\alpha}{\varepsilon^{-2}} \right) \leq {I(\varepsilon)^{2\alpha}}{\varepsilon^{-2}} |\log(\varepsilon)|^\alpha. \]
 By Proposition~\ref{regofdensity} (noting that $f_\tau(x, \cdot) \equiv h_{R^\varepsilon_\tau}(x,\cdot)$ for $h$ the transition density of a $\dd$-dimensional Brownian motion) using that the first integral is bounded above by two, and allowing the constant $C$ to change from line to line,
\allowdisplaybreaks  \begin{align}\label{chillip}
\nonumber & \sup_{| w | \leq \delta_\varepsilon} \EE\left[\int_{\B_x} |f_\tau(x,z-w)-f_\tau(x,z)|  dz \mathbbm{1}_{\tau \leq t \wedge t^*} \right] \\ \nonumber & \leq \sup_{| w | \leq \delta_\varepsilon} C \EE\left[ \int_{\B_x} | w | (R_\tau^\varepsilon)^{-\frac{\dd+1}{2}}  dz\mathbbm{1}_{\tau \leq t \wedge t^*} \mathbbm{1}_{\tau \geq \delta^\alpha_\varepsilon}\right] + 2{I(\varepsilon)^{2\alpha}}{\varepsilon^{-2}} |\log(\varepsilon)|^\alpha \\ \nonumber &\leq C  \delta_\varepsilon  \EE\left[V(\B_x)  (R_\tau^\varepsilon)^{-\frac{\dd+1}{2}} \mathbbm{1}_{\tau \geq \delta^\alpha_\varepsilon} \right] + 2{I(\varepsilon)^{2\alpha}}{\varepsilon^{-2}} |\log(\varepsilon)|^\alpha\\  & \leq C  \delta_\varepsilon \mathbb{E}\left[ (2\tau)^{\frac{\dd}{\alpha}} (R_\tau^\varepsilon)^{-\frac{\dd+1}{2}} \mathbbm{1}_{\tau \geq \delta^\alpha_\varepsilon}\right] +  2{I(\varepsilon)^{2\alpha}}{\varepsilon^{-2}} |\log(\varepsilon)|^\alpha \end{align} where $V(\B_x)$ denotes the volume of $\B_x$ which is proportional to $(\tau^{\frac{1}{\alpha}}+\delta_\varepsilon)^{\mathbbm{d}}$, and, conditional on $\tau \geq \delta^{\alpha}_\varepsilon$, is bounded above by $(2\tau)^{\frac{\dd}{\alpha}}$. In Lemma~\ref{newLemma:NegMomentsForR}, we will bound the $-p$-th moments of $(R^\varepsilon_s)_{s\geq0}$. Using this, and letting $D_1, D_2$ change from line to line, we obtain \begin{align*} \EE\left[\tau^{\frac{\dd}{\alpha}}(R_\tau^\varepsilon)^{-\frac{\dd+1}{2}} \mathbbm{1}_{\tau \geq \delta^\alpha_\varepsilon} \right] &\leq \EE\left[\tau^{\frac{\dd}{\alpha}}(R_{\tau}^\varepsilon)^{-\frac{\dd+1}{2}} \right]\\
&\leq D_1 \EE\left[\tau^{\frac{\dd}{\alpha}}e^\tau\right] + D_2\EE\left[\tau^{-\frac{1}{\alpha}}e^\tau\right]\\
&\leq D_1 \varepsilon^{\frac{2\dd}{\alpha}} + D_2 \varepsilon^{-\frac{2}{\alpha}}.\end{align*} Substituting this back into (\ref{chillip}), and again letting the constants change from line to line, we obtain \begin{align} \nonumber &\sup_{| w | \leq \delta_\varepsilon} \EE\left[\int_{\B_x} |f_\tau(x,z-w)-f_\tau(x,z)|  dz \mathbbm{1}_{\tau\leq t \wedge t^*} \right] \\  &  \leq C \delta_\varepsilon\left(\varepsilon^{\frac{2\dd}{\alpha}} + \varepsilon^{-\frac{2}{\alpha}}\right)+ 2{I(\varepsilon)^{2\alpha}}{\varepsilon^{-2}} |\log(\varepsilon)|^\alpha \nonumber \\ 
&= C_1\varepsilon^{\frac{2\dd}{\alpha}}I(\varepsilon)^{2}|\log(\varepsilon)| + C_2 {I(\varepsilon)^{2}}{\varepsilon^{-\frac{2}{\alpha}}}|\log\varepsilon| + 2{I(\varepsilon)^{2\alpha}}{\varepsilon^{-2}} |\log(\varepsilon)|^\alpha\nonumber \\
& \leq C_2 {I(\varepsilon)^{2}}{\varepsilon^{-\frac{2}{\alpha}}}|\log\varepsilon|,\label{lasteqn}
\end{align} where the final inequality follows by Assumptions~\ref{assumptions2}~(B)-(C) and we allow $C_2$ to change from line to line. By Assumption~\ref{assumptions2}~(C), (\ref{lasteqn}) goes to $0$ with $\varepsilon$.

To bound the second term in (\ref{twointegrals}) we use \cite[Theorem 1.1]{chen:2003}, which provides upper bounds on the transition density of subordinated Brownian motion, with truncated stable subordinator that has truncation level independent of $\varepsilon$. To apply this result, we rewrite $W(R^\varepsilon_s)$ in terms of a $1$-truncated subordinated Brownian motion as follows.
Let $(U_{s}^{a})_{s \geq 0}$ denote an $\frac{\alpha}{2}$-stable subordinator with truncation level $a$ (and no speed change). Let $\stackrel{D}{=}$ denote equality in distribution. Then $$R_s^\varepsilon \stackrel{D}{=} U_{s I(\varepsilon)^{\alpha-2}}^{I(\varepsilon)^2} \stackrel{D}{=}I(\varepsilon)^2 U_{s I(\varepsilon)^{-2}}^{1}$$ where the first equality follows by definition and the second equality follows by showing that, if $\Psi$ and $\Psi'$ are the characteristic exponents of $(U^{I(\varepsilon)^2}_s)_{s\geq 0}$ and $(U^1_s)_{s\geq 0}$ respectively, then $I(\varepsilon)^\alpha \Psi(\theta) = \Psi'(\theta I(\varepsilon)^2),$ which follows easily from the L\'evy-Khintchine formula.
Then, by the scaling property of the Brownian motion, \reqnomode
\begin{align} W(R_s^\varepsilon) \stackrel{D}{=} I(\varepsilon) W(U_{s I(\varepsilon)^{-2}}^1). \label{scalingtruncBM} \end{align}
Denote the transition density of $(W(U^{1}_t))_{t \geq 0}$ by $\hat{f}_t(x,y)$. By (\ref{scalingtruncBM}), $\hat{f}_t(x,\cdot)$ is related to the transition density of $W(R^\varepsilon_s)$ by
\[ f_t(x,y) = \hat{f}_{I(\varepsilon)^{-2}t}(I(\varepsilon)^{-1}x,I(\varepsilon)^{-1}y). \]
Therefore the second term in (\ref{twointegrals}) can be rewritten as
\[ \sup_{| w| \leq \delta_\varepsilon} \EE\left[\int_{\B_x^{\mathsf{c}}} \left|\hat{f}_{I(\varepsilon)^{-2}\tau }(I(\varepsilon)^{-1}x,I(\varepsilon)^{-1}(z-w))-\hat{f}_{I(\varepsilon)^{-2}\tau }(I(\varepsilon)^{-1}x,I(\varepsilon)^{-1}z)\right| dz  \mathbbm{1}_{\tau \leq t \wedge t^*}\right] \]
which, by \cite[Theorem 1.1]{chen:2003}, is bounded above by 

\begin{align}
\nonumber & \sup_{| w| \leq \delta_\varepsilon} \EE\left[\int_{\B_x^{\mathsf{c}}} \frac{D \tau I(\varepsilon)^{-2}}{((z - w) I(\varepsilon)^{-1} )^{\mathbbm{d}+\alpha}} dz \right] +\EE\left[\int_{\B_x^{\mathsf{c}}} \frac{D\tau I(\varepsilon)^{-2}}{(z I(\varepsilon)^{-1} )^{\mathbbm{d}+\alpha}} dz \right] \\ \nonumber &  =  \sup_{| w| \leq \delta_\varepsilon} \EE\left[\int_{\B_x^{\mathsf{c}}} \frac{D\tau I(\varepsilon)^{\alpha-1}}{(z-w)^{\mathbbm{d}+\alpha}} dz \right]+ \EE\left[\int_{\B_x^{\mathsf{c}}} \frac{D \tau I(\varepsilon)^{\alpha-1}}{z^{\mathbbm{d}+\alpha}} dz \right]   \\ 
\nonumber & \leq D I(\varepsilon)^{\alpha-1} \mathbb{E}\left[ \tau \int_{\tau^{\frac{1}{\alpha}}}^\infty \frac{1}{z^{\alpha+1}} dz \right]  \\ 
\nonumber & \leq  D I(\varepsilon)^{\alpha-1} \mathbb{E}[\tau (\tau^{-1})] \\  &\label{anotherbound} = D I(\varepsilon)^{\alpha-1}, \end{align} 
for some $D>0$ that we allow to change from line to line. Here we have applied 
the change of variables $z-w\mapsto z$ in the third line. Note that, since $\alpha>1$ the last quantity goes to $0$ with $\varepsilon$. Recall from (\ref{defnofF}) that $$F(\varepsilon) = {I(\varepsilon)^{2}}{\varepsilon^{-\frac{2}{\alpha}}}|\log\varepsilon| + I(\varepsilon)^{\alpha-1}. $$ Then, choosing $m>0$ sufficiently large and $\varepsilon$ sufficiently small so that $$mF(\varepsilon) \geq  \varepsilon^k  + DI(\varepsilon)^{\alpha-1} + C_2 {I(\varepsilon)^{2}}{\varepsilon^{-\frac{2}{\alpha}}}|\log\varepsilon|,$$ 
by (\ref{anotherbound}) and (\ref{lasteqn}) we can bound (\ref{twointegrals}) to obtain
\begin{align} &\EE_x\left[g_\times(u(t-\tau, Z^+_\tau))\mathbbm{1}_{\tau \leq t}\right]\nonumber\\ &\leq \mathbb{E}\left[\int_{\RR^\dd} g_\times(u(t-\tau, z)) f_\tau(x,z) dz \mathbbm{1}_{\tau \leq t \wedge t^*}\right] + mF(\varepsilon). \label{pink}\end{align}
Finally, we note that \begin{align*}\EE_x\left[g_\times(u(t-\tau, W(R^\varepsilon_\tau))\mathbbm{1}_{\tau\leq t}\right] &\geq \EE_x\left[g_\times(u(t-\tau, W(R^\varepsilon_\tau))\mathbbm{1}_{\tau\leq t\wedge t^*}\right] \\ &= \EE\left[\int_{\RR^\dd} g_\times(u(t-\tau, z))f_\tau(x,z)dz \mathbbm{1}_{\tau\leq t\wedge t^*}\right],\end{align*} which, together with (\ref{pink}) and (\ref{red}) gives 
\begin{align*}
&u(t,x)-v(t,x) \\ &\leq \mathbb{E}\left[\int g_\times(u(t-\tau, z)) f_\tau(x,z) dz \mathbbm{1}_{\tau \leq t \wedge t^*}\right]-\mathbb{E}\left[\int g_\times(v(t-\tau, z)) f_\tau(x,z) dz \mathbbm{1}_{\tau \leq t \wedge t^*}\right] \\ & \ \  +e^{-\frac{t}{\varepsilon^2}}+ mF(\varepsilon).
\end{align*}

\noindent Using the same approach we can obtain the lower bound on $u(t,x) - v(t,x).$ Namely, by analogy with (\ref{last}), \begin{align*}
    &\EE_x\left[g_\times(u(t-\tau,Z^+_t))\mathbbm{1}_{\tau \leq t}\right] \\ &\geq  \EE\left[\int_{\RR^\dd} g(u(t-\tau, z))[f_\tau(x,{z}-{w})-f_\tau(x,{z})]\, \mathbbm{1}_{\tau \leq t \wedge t^*} dz\right]\\  & \ \ +\mathbb{E}\left[g(u(t-\tau, W(R^{\varepsilon}_\tau))) \mathbbm{1}_{\tau \leq t \wedge t^*}\right] \\
    &\geq  - \sup_{\substack{| w | \leq \delta_\varepsilon}}\left| \EE\left[\int_{\RR^\dd} g(u(t-\tau, z))[f_\tau(x,{z}-{w})-f_\tau(x,{z})]\, \mathbbm{1}_{\tau \leq t \wedge t^*} dz\right]\right|\\ & \ \ +\mathbb{E}\left[g(u(t-\tau, W(R^{\varepsilon}_\tau))) \mathbbm{1}_{\tau \leq t \wedge t^*}\right]
\end{align*} and the final term can be bounded identically as before to give us that
\begin{align*}
&u(t,x)-v(t,x) \\&\geq \mathbb{E}\left[\int g_\times(u(t-\tau, z)) f_\tau(x,z) dz \mathbbm{1}_{\tau \leq t \wedge t^*}\right]-\mathbb{E}\left[\int g_\times(v(t-\tau, z)) f_\tau(x,z) dz \mathbbm{1}_{\tau \leq t \wedge t^*}\right] \\ & \ \ -\left( e^{-\frac{t}{\varepsilon^2}}+mF(\varepsilon) \right).
\end{align*}
Therefore, using that $g_\times$ is Lipschitz with constant $\tfrac{3}{2}$, we obtain
\begin{align*} 
&|u(t,x)-v(t,x)| \\ &\leq \EE\left[\left|\int (g(u(t-\tau, z))-g(v(t-\tau,z)))f_\tau(x,z) dz \right|\mathbbm{1}_{\tau \leq t}\right]+ e^{-\frac{t}{\varepsilon^2}}+mF(\varepsilon)\\ & \leq \tfrac{3}{2} \EE\left[|u(t-\tau,W(R^\varepsilon_\tau))-v(t-\tau, W(R^\varepsilon_\tau)) |\mathbbm{1}_{\tau \leq t}\right]  +  e^{-\frac{t}{\varepsilon^2}}+mF(\varepsilon) \\ & \leq \tfrac{3}{2} \int_0^t \lVert u(\rho, \cdot)-v(\rho,\cdot) \rVert_{\infty} e^{-(t-\rho) \varepsilon^{-2}} \varepsilon^{-2} d\rho +  e^{-\frac{t}{\varepsilon^2}}+mF(\varepsilon).
\end{align*}
Finally, using the Gronwall's inequality \cite{movljankulov1972ob} (see also \cite[Theorem 15]{dragomir2002some}) we deduce that
\begin{align*}
\lVert u(t,\cdot) - v(t,\cdot) \rVert_\infty &\leq \left( e^{-\frac{t}{\varepsilon^2}}+mF(\varepsilon)\right) \exp\left( \tfrac{3}{2} \int_0^t \exp(-u \varepsilon^{-2}) \varepsilon^{-2} du \right), \\ & \leq \left( e^{-\frac{t}{\varepsilon^2}}+ mF(\varepsilon)\right) \exp\left( \tfrac{3}{2} \right),
\end{align*}
which gives the desired bound by choosing an appropriate $m_1, m_2>0$. 
\end{proof}

\subsection{Truncated subordinator calculations}\label{appsub} Throughout this section, let $(R^\varepsilon_s)_{s\geq 0}$ be the $I(\varepsilon)^2$-truncated $\tfrac{\alpha}{2}$-stable subordinator with L\'evy measure given by \begin{align*}\tfrac{\alpha}{2}\left(\tfrac{2-\alpha}{\alpha}\right)^{\frac{\alpha}{2}} I(\varepsilon)^{\alpha-2}{y^{-1-\frac{\alpha}{2}}} \mathbbm{1}_{\left\{0\leq y\leq \frac{2-\alpha}{\alpha} I(\varepsilon)^2\right\}}dy.\end{align*} 
Denote by $\PP$ the probability measure under which $(R^\varepsilon_s)_{s\geq0}$, started at $R^\varepsilon_0 = 0$ has this distribution, and let $\EE$ denote the corresponding expectation. For all $s, \lambda \geq 0$ the Laplace transform $\phi(\lambda):= \EE\left[\exp\left(-\lambda R^\varepsilon_s\right)\right]$ of $R^\varepsilon_s$ is \begin{align}\label{laplace_transform_R}
      \phi(\lambda) &= \exp\left(\tfrac{\alpha}{2}\left(\tfrac{2-\alpha}{\alpha}\right)^{\frac{\alpha}{2}} I(\varepsilon)^{\alpha-2}s\int_0^{\frac{2-\alpha}{\alpha}I(\varepsilon)^2} \frac{e^{-\lambda y}-1}{y^{\frac{\alpha}{2}+1}}dy\right).
\end{align} The following lemma will enable us to show, in Proposition~\ref{boundonsubordinator}, that $R^\varepsilon_s$ is close to $s$ for small times $s$, which is a crucial component of our proof of Theorem~\ref{teo:subestimate}.
\begin{lemma}\label{expectedvalueofR}
Let $(R^\varepsilon_s)_{s\geq 0}$ be as above. For all $s\geq 0$, $\EE[R_s^\varepsilon] = s. $
\end{lemma}
\begin{proof}\reqnomode
The expected value of $R^\varepsilon_s$ can be calculated explicitly by considering the derivative of its Laplace transform, namely $\EE[R_s^\varepsilon] = -\frac{d}{d\lambda}\phi(\lambda)|_{\lambda =0}$. Denote $I:= I(\varepsilon)$ throughout. Fix $\lambda, s \geq 0$. Using integration by parts, \begin{align*}
    &\int_0^{\frac{2-\alpha}{\alpha}I^2} \frac{e^{-\lambda y}-1}{y^{\frac{\alpha}{2}+1}}dy = -\tfrac{2}{\alpha}\left(\tfrac{2-\alpha}{\alpha}\right)^{-\frac{\alpha}{2}}I^{-\alpha}(e^{-\frac{2-\alpha}{\alpha}\lambda I^2}-1) - \tfrac{2}{\alpha}\lambda^{\frac{\alpha}{2}} \gamma\left(1-\tfrac{\alpha}{2}, \tfrac{2-\alpha}{\alpha}\lambda I^2\right)
\end{align*} where $\gamma(s,x) := \int_0^x t^{s-1}e^{-t}dt$ is the lower incomplete gamma function. So, using (\ref{laplace_transform_R}), $\phi(\lambda)$ can be written as
\begin{align*} \phi(\lambda) = \exp\left(- I^{-2}s\left(e^{- \frac{2-\alpha}{\alpha}\lambda I^2}-1\right) - \left({\tfrac{2-\alpha}{\alpha}}\right)^{\frac{\alpha}{2}}I^{\alpha-2}\lambda^{\frac{\alpha}{2}} s\gamma\left(1-\tfrac{\alpha}{2}, \tfrac{2-\alpha}{\alpha}\lambda I^2\right) \right).\end{align*}
By differentiating this quantity with respect to $\lambda$, we obtain \begin{align*}
 &   \EE[R_s^\varepsilon] \\ &= \left. - \tfrac{2-\alpha}{\alpha}s + \left({\tfrac{2-\alpha}{\alpha}}\right)^{\frac{\alpha}{2}}I^{\alpha-2}s\left(\tfrac{\alpha}{2}\lambda^{\frac{\alpha}{2}-1} \gamma\left(1-\tfrac{\alpha}{2}, \tfrac{2-\alpha}{\alpha}\lambda I^2\right)+ \lambda^{\frac{\alpha}{2}} \frac{d\gamma }{d \lambda}\left(1-\tfrac{\alpha}{2}, \tfrac{2-\alpha}{\alpha}\lambda I^2\right)\right) \right|_{\lambda=0}.\end{align*} 

This can be calculated by considering the following identities for $\gamma$. First, by the definition of $\gamma$ and a change of variables\begin{align*}\gamma\left(1-\tfrac{\alpha}{2}, \tfrac{2-\alpha}{\alpha}\lambda I^2\right) = I^{2-\alpha} \int_0^{\frac{2-\alpha}{\alpha}\lambda} z^{-\frac{\alpha}{2}}e^{-zI^2}dz.\end{align*} Therefore \begin{align*} \lambda^{\frac{\alpha}{2}}\frac{d\gamma}{d\lambda} \left.\left(1-\tfrac{\alpha}{2}, \tfrac{2-\alpha}{\alpha}\lambda I^2\right)\right|_{\lambda=0} &=\left. I^{2-\alpha} \left(\tfrac{2-\alpha}{\alpha}\right)^{1-\frac{\alpha}{2}}e^{-\frac{2-\alpha}{\alpha}\lambda I^2}\right|_{\lambda=0}\\
&=  I^{2-\alpha} \left(\tfrac{2-\alpha}{\alpha}\right)^{1-\frac{\alpha}{2}}.\end{align*} Under a different change of variables, $\gamma$ also satisfies \begin{align*} \left.\lambda^{\frac{\alpha}{2}-1}\gamma\left(1-\tfrac{\alpha}{2}, \tfrac{2-\alpha}{\alpha}\lambda I^2\right)\right|_{\lambda=0} &= \left(\tfrac{2-\alpha}{{\alpha}}\right)^{1-\frac{\alpha}{2}}\int_0^{I^2}\left.z^{-\frac{\alpha}{2}}e^{-\frac{2-\alpha}{\alpha}\lambda z}dz\right|_{\lambda=0} \\ 
&=\left(\tfrac{2-\alpha}{{\alpha}}\right)^{1-\frac{\alpha}{2}} \left(\tfrac{2}{2-\alpha}\right)I^{2-\alpha}.\end{align*} Putting this all together we obtain \begin{align*} \EE[R_s^\varepsilon] &= - \tfrac{2-\alpha}{\alpha}s + \left(\tfrac{2-\alpha}{\alpha}\right)^{\frac{\alpha}{2}}I^{\alpha-2}s\left(\tfrac{\alpha}{2} \left(\tfrac{2-\alpha}{\alpha}\right)^{1-\frac{\alpha}{2}}\left(\tfrac{2}{2-\alpha}\right)I^{2-\alpha}+ \left(\tfrac{2-\alpha}{\alpha}\right)^{1-\frac{\alpha}{2}}I^{2-\alpha}\right) \\
&= s.\qedhere\end{align*} 
\end{proof}
\noindent With this, we can now prove Proposition~\ref{boundonsubordinator}. 
\begin{proposition}\label{boundonsubordinator}
Let $k\in \NN$ and $(R^\varepsilon_s)_{s\geq 0}$ be as above. There exists $\varepsilon_k>0$ such that, for all $\varepsilon\in (0,\varepsilon_k)$ and $s\leq \varepsilon^2|\log\varepsilon|$, $$\PP\left[|R^\varepsilon_s-s| \geq (k+1)I(\varepsilon)^2|\log\varepsilon|\right]\leq \varepsilon^k.$$
\end{proposition}

\begin{proof}
    Let $\varepsilon>0$. As before we set $I:= I(\varepsilon)$. Fix $k\in \NN$ and $s\leq \varepsilon^2 |\log\varepsilon|$. Note that \begin{align}\label{edsheeran}& \PP\left[|R^\varepsilon_s-s|\geq (k+1)I^2|\log\varepsilon|\right]\nonumber \\ &=\PP\left[R^\varepsilon_s\geq (k+1)I^2|\log\varepsilon|+s\right] + \PP\left[R^\varepsilon_s \leq s-(k+1)I^2|\log\varepsilon|\right]. \end{align} By Assumption~\ref{assumptions2} (b), $\varepsilon I^{-1}\to 0$ as $\varepsilon\to 0$, and since $s\leq \varepsilon^2|\log \varepsilon|$, we may decrease $\varepsilon$ as necessary to ensure that $s-(k+1)I^2|\log\varepsilon|\leq 0$, and the second term in (\ref{edsheeran}) equals zero. \\ 
    \indent To bound the first term in (\ref{edsheeran}), we note that, by \cite[Theorem~25.17]{sato1999levy}, and since the L\'evy measure of $R^\varepsilon_s$ has compact support, the Laplace transform $\phi(\lambda)$ from (\ref{laplace_transform_R}) can be extended to all of $\RR$, and the exponential moments of $R^\varepsilon_s$ exist and satisfy $\EE\left[\exp\left(\lambda R^\varepsilon_s\right)\right] = \phi(\lambda)$ for all $\lambda\in \RR$. It is straightforward to verify using Lemma~\ref{expectedvalueofR} that $\EE\left[R_s^\varepsilon\right]$ satisfies
   $$\EE\left[R_s^\varepsilon\right] = \tfrac{\alpha}{2}\left(\tfrac{2-\alpha}{\alpha}\right)^{\frac{\alpha}{2}} I^{\alpha-2}s\int_0^{\frac{2-\alpha}{\alpha}I^2} {y^{-\frac{\alpha}{2}}}dy,$$ therefore for any $s, \lambda\geq 0$, by Taylor's theorem
\begin{align}\label{exp_bound}
{\EE}\left[\exp\left({\lambda R^\varepsilon_s - \lambda \EE[R^\varepsilon_s]}\right)\right]\nonumber  &= \exp\left( \tfrac{\alpha}{2}\left(\tfrac{2-\alpha}{\alpha}\right)^{\frac{\alpha}{2}} I^{\alpha-2}s\int_0^{\frac{2-\alpha}{\alpha}I^2} \frac{e^{\lambda v } - 1 -\lambda v}{v^{1+\frac{\alpha}{2}}} dv  \right) \\  \nonumber
&\leq \exp\left( \tfrac{\alpha}{4} \left(\tfrac{2-\alpha}{\alpha}\right)^{\frac{\alpha}{2}} I^{\alpha-2} e^{\frac{2-\alpha}{\alpha}I^2 \lambda}\lambda^2  s \int_0^{\frac{2-\alpha}{\alpha}I^2} v^{1-\frac{\alpha}{2}} dv\right) \\&= \exp\left(\tfrac{1}{2}\left(\tfrac{2-\alpha}{\alpha}\right)^2\tfrac{\alpha}{4-\alpha} e^{\frac{2-\alpha}{\alpha}I^2 \lambda} I^{2}  \lambda^2  s\right). \end{align} Choose $\lambda = I^{-2}$. Then (\ref{exp_bound}) becomes \begin{align*}
    {\EE}\left[\exp\left({ I^{-2} R^\varepsilon_s -  I^{-2} \EE[R^\varepsilon_s]}\right)\right]&\leq \exp\left(\tfrac{1}{2}\left(\tfrac{2-\alpha}{\alpha}\right)^2\tfrac{\alpha}{4-\alpha} e^{\frac{2-\alpha}{\alpha}} I^{-2}  s\right)\\
    &\leq  \exp\left(\tfrac{1}{2}\left(\tfrac{2-\alpha}{\alpha}\right)^2\tfrac{\alpha}{4-\alpha} e^{\frac{2-\alpha}{\alpha}} \frac{\varepsilon^2|\log \varepsilon|}{I^2}\right).
\end{align*}
This, together with Lemma~\ref{expectedvalueofR} and Markov's inequality gives us
\begin{align*}
{\PP}\left[R^\varepsilon_s \geq  (k+1)I^2|\log\varepsilon| + s \right] & = {\PP}\left[R^\varepsilon_s - \EE[R^\varepsilon_s] \geq  (k+1)I^2|\log\varepsilon|\right] \\ 
&= {\PP}\left[\exp\left({ I^{-2} R^\varepsilon_s - I^{-2}\EE[R^\varepsilon_s]}\right) \geq  \varepsilon^{-k-1} \right] \\ & \leq \varepsilon^{k+1}\exp\left(\tfrac{1}{2}\left(\tfrac{2-\alpha}{\alpha}\right)^2\tfrac{\alpha}{4-\alpha} e^{\frac{2-\alpha}{\alpha}}\frac{\varepsilon^2|\log \varepsilon|}{I^2}\right).
\end{align*} By Assumption~\ref{assumptions2} \ref{assumptions2_B}, $\frac{\varepsilon^2|\log \varepsilon|}{I^2}\to 0$ as $\varepsilon\to 0,$ so the previous inequality implies that, for $\varepsilon$ sufficiently small, $${\PP}\left[R^\varepsilon_s > (k+1)I^2|\log\varepsilon| + s \right]  \leq \varepsilon^k,$$ which, by (\ref{edsheeran}) gives us the result.
\end{proof} 
\begin{lemma} \label{newLemma:NegMomentsForR}
Let $(R^\varepsilon_s)_{s\geq 0}$ be as above. Then, for any $q, s>0$, there exists $D_1 = D_1(q,\alpha)>0$ and $D_2 = D_2(q,\alpha)>0$ such that
\[ \EE\left[(R^\varepsilon_{s})^{-q}\right] \leq  e^s (D_1+ D_2s^{-\frac{2q}{\alpha}}). \]
\end{lemma}
\begin{proof}
Fix $s\geq 0$, $\varepsilon>0$ and $I:= I(\varepsilon)$. Let $\gamma$ be the lower incomplete gamma function. The Laplace transform (\ref{laplace_transform_R}) of $R^\varepsilon_s$ can be written as \begin{align*}
     \phi(\lambda) &= \exp\left(- I^{-2}s\left(e^{- \frac{2-\alpha}{\alpha}\lambda I^2}-1\right) - \left({\tfrac{2-\alpha}{\alpha}}\right)^{\frac{\alpha}{2}}I^{\alpha-2}\lambda^{\frac{\alpha}{2}} s\gamma\left(1-\tfrac{\alpha}{2}, \tfrac{2-\alpha}{\alpha}\lambda I^2\right) \right)\\
     &=  \exp\left(I^{-2}s\int_0^{\frac{2-\alpha}{\alpha}\lambda I^2}e^{-z}dz - \left({\tfrac{2-\alpha}{\alpha}}\right)^{\frac{\alpha}{2}}I^{\alpha-2}\lambda^{\frac{\alpha}{2}} s\int_0^{\frac{2-\alpha}{\alpha}\lambda I^2}z^{-\frac{\alpha}{2}}e^{-z}dz \right)\\
     &= \exp\left(\tfrac{2-\alpha}{\alpha}s\int_0^{\lambda }e^{-\frac{2-\alpha}{\alpha} I^2 z}\left(1- \lambda^{\frac{\alpha}{2}}z^{-\frac{\alpha}{2}}\right)dz \right).
\end{align*} 
 By \cite[Theorem 1.1]{schurger2002laplace}, if $X$ is a non-negative random variable with Laplace transform $\rho$, then for any $q>0$, $\EE\left[X^{-q}\right] = \frac{1}{q\Gamma(q)}\int_0^
 \infty \rho\left(t^{\frac{1}{q}}\right) dt.$ Therefore, using the above expression for the Laplace transform of $R^\varepsilon_s,$ for any $q>0$ 
\begin{align*}
    \EE\left[(R^\varepsilon_s)^{-q}\right] &= \frac{1}{q\Gamma(q)}\int_0^
 \infty \exp\left(\tfrac{2-\alpha}{\alpha}s\int_0^{\lambda^{\frac{1}{q}} }e^{-\frac{2-\alpha}{\alpha} I^2 z}\left(1- \lambda^{\frac{\alpha}{2q}}z^{-\frac{\alpha}{2}}\right)dz \right)d\lambda.
\end{align*}
When $\lambda \in [0,1],$ \begin{align*}
  \int_0^{\lambda^{\frac{1}{q}} }e^{-\frac{2-\alpha}{\alpha} I^2 z}\left(1- \lambda^{\frac{\alpha}{2q}}z^{-\frac{\alpha}{2}}\right)dz  \leq 1 
\end{align*} and if $\lambda>1$, since the integrand is negative
\begin{align*}
  & \int_0^{\lambda^{\frac{1}{q}} }e^{-\frac{2-\alpha}{\alpha} I^2 z}\left(1- \lambda^{\frac{\alpha}{2q}}z^{-\frac{\alpha}{2}}\right)dz \leq  \int_0^{1 }1- \lambda^{\frac{\alpha}{2q}}z^{-\frac{\alpha}{2}} dz = 1-\tfrac{2}{2-\alpha}\lambda^{\frac{\alpha}{2q}}.
\end{align*} Therefore \begin{align*}
     \EE\left[(R^\varepsilon_s)^{-q}\right] &\leq \frac{1}{q\Gamma(q)}e^{\frac{2-\alpha}{\alpha}s}\left(1+\int_1^
 \infty \exp\left(- \tfrac{2}{\alpha}s \lambda^{\frac{\alpha}{2q}}\right)d\lambda\right).\end{align*} Let $\Gamma(s,x):= \int_x^\infty t^{s-1}e^{-t}dt $ be the upper incomplete gamma function. Then it is straightforward to show that the above is equivalent to \begin{align*}
  \EE\left[(R^\varepsilon_s)^{-q}\right] &\leq\frac{1}{q\Gamma(q)}e^{\frac{2-\alpha}{\alpha}s}\left(1+  \frac{2q}{\alpha}\left(\frac{\alpha}{2s}\right)^{\frac{2q}{\alpha}}\Gamma\left(\tfrac{2q}{\alpha}, \tfrac{2s}{\alpha}\right)\right)\\
  &\leq \frac{1}{q\Gamma(q)}e^{\frac{2-\alpha}{\alpha}s}\left(1+  \frac{2q}{\alpha}\left(\frac{\alpha}{2s}\right)^{\frac{2q}{\alpha}}\Gamma\left(\tfrac{2q}{\alpha}\right)\right).
\end{align*}
Since $\tfrac{2-\alpha}{\alpha}\leq 1,$ the result follows.
\end{proof}

\end{document}